\documentclass[11pt,a4paper,fleqn]{article}
\oddsidemargin=.\paperwidth
\topmargin=.\paperheight
\makeatletter
\newcommand\footnotetext@\relax
\let\footnotetext@\@footnotetext
\newcommand{\MSC}[2][2020]{%
 \unskip\protected@xdef\@thefnmark{}%
 \protect\footnotetext@{\kern-1.8em{\itshape MSC#1\spacefactor3000:}\/ #2}}
\newcommand{\keywords}[1]{%
 \unskip\protected@xdef\@thefnmark{}%
 \protect\footnotetext@{\kern-1.8em{\itshape Keywords\spacefactor3000:}\/ #1}}
\newcommand{\address}[1]{%
 \unskip\footnotemark
 \protected@xdef\@thanks{\@thanks\protect\footnotetext[\the\c@footnote]{{\itshape Address\spacefactor3000:}\/ #1}}}
\newcommand{\email}[1]{%
 \unskip\protected@xdef\@thanks{\@thanks\protect\footnotetext[0]{{\itshape Email\spacefactor3000:}\/ \texttt{#1}}}}
\renewcommand{\thanks}[1]{%
 \unskip\protected@xdef\@thanks{\@thanks\protect\footnotetext[0]{#1}}}
\gdef\@date{}
\makeatother

\usepackage[english]{babel}
\usepackage[utf8]{inputenc}
\usepackage[T1]{fontenc}
\usepackage{amsmath}%
 \allowdisplaybreaks[4]
\usepackage{amssymb,amsthm}%
 \theoremstyle{plain}
  \newtheorem{thm}{Theorem}[section]
  \newtheorem{prop}[thm]{Proposition}
  \newtheorem{cor}[thm]{Corollary}
  \newtheorem{lem}[thm]{Lemma}
 \theoremstyle{definition}
  \newtheorem{defn}[thm]{Definition}
  \newtheorem*{defn*}{Definition}
  \newtheorem{exm}[thm]{Example}
  
  \newtheorem*{exms*}{Examples}
 \theoremstyle{remark}
  \newtheorem{rmk}[thm]{Remark}
  \newtheorem{rmks}[thm]{Remarks}
  \newtheorem*{rmks*}{Remarks}
\usepackage[all]{xy}%
 \entrymodifiers={!!<.em,.6ex>+}
 \SelectTips{cm}{11}
\usepackage{tikz}
\usepackage{hyperref}

\DeclareMathOperator{\Ad}{Ad}% — adjoint representation
\newcommand{\aeq}{\Leftrightarrow}% — aequatur (logical connective)
\newcommand{\an}{\sharp}% — infinitesimal anchor map
\newcommand{\append}[2]{(#2)#1{\,\vphantom{#2}}\mkern+2mu}
\newcommand{\Cat}{\mathsf}% — category proper name
\newcommand{\cat}{\mathcal}% — category variable
\DeclareMathOperator{\cone}{C}% — mapping cone
\newcommand{\cplx}[1]{\mathfrak{\lowercase{#1}}}% — tangent complex
\DeclareMathOperator{\End}{End}
\newcommand{\hcat}[1]{\rlap{$\widetilde{\phantom{#1}}$}#1}% — homotopy category
\newcommand{\hmt}{\vec}% — homotopy prism
\DeclareMathOperator{\Hom}{Hom}
\newcommand{\hto}{\Rightarrow}
\newcommand{\id}{\mathrm{id}}% — identity map
\DeclareMathOperator{\im}{im}% — image
\newcommand{\into}{\rightarrowtail}
\newcommand{\justify}[1]{\mathrel{\phantom=}#1\mathopen{\mkern\medmuskip}}
\newcommand{\longsimto}{\overset\sim\longto}% — bijection, isomorphism
\newcommand{\longto}{\longrightarrow}
\newcommand{\Mfd}{\Cat{Mfd}}% — category of smooth manifolds
\DeclareMathOperator{\nerve}{N}% — nerve of a poset
\newcommand{\nor}{\mathrm{nor}}% — normalization map
\newcommand{\op}{\mathsf{op}}% — categorical opposite
\newcommand{\pr}{\mathrm{pr}}% — projection map
\newcommand{\Rep}{\Cat{Rep}}% — category of representations
\newcommand{\sdp}{\bar}% — semidirect product
\newcommand{\seq}{\Rightarrow}% — sequitur (logical connective)
\newcommand{\Set}{\Cat{Set}}% — category of sets
\DeclareMathOperator{\sgn}{sgn}% — sign
\newcommand{\simto}{\overset\sim\to}% — bijection, isomorphism
\newcommand{\sMfd}{\Cat{sMfd}}% — simplicial manifolds
\newcommand{\spl}{\hat}% — splitting construction
\newcommand{\sSet}{\Cat{sSet}}% — simplicial sets
\newcommand{\tto}{\rightrightarrows}
\newcommand{\VB}{\Cat{VB}}% — category of vector bundles
\newcommand{\xfrom}{\xleftarrow}
\newcommand{\xto}{\xrightarrow}
\newcommand{\Z}{\mathbb{Z}}% — integers

\begin{document}

\title{The Adjoint Representation of a Higher Lie Groupoid%
  \MSC[2020]{Primary 18N25, 18N50, Secondary 18G31, 18G35}
  \keywords{Lie infinity-groupoid, representation up to homotopy, %
      simplicial vector bundle, cleavage, tangent complex, adjoint representation}
}%
\author{Giorgio Trentinaglia%
  \address{Centro de Análise Matemática, Geometria e Sistemas Dinâmicos, %
      Ins\-ti\-tu\-to Su\-pe\-ri\-or Téc\-ni\-co, University of Lisbon, %
      Av.~Ro\-vis\-co Pais, 1049-001 Lisbon, Portugal}
  \email{gtrentin@math.tecnico.ulisboa.pt}
  \thanks{The author acknowledges the support %
      of the Portuguese Foundation for Science and Technology %
      through grants SFRH/BPD/81810/2011 and UID/MAT/04459/2020.}
}%
\maketitle

\begin{abstract}
We extend the standard construction of the adjoint representation of a Lie groupoid to the case of an arbitrary higher Lie groupoid. As for a Lie groupoid, the adjoint representation of a higher Lie groupoid turns out to be a representation up to homotopy which is well defined up to isomorphism. Its existence and uniqueness are immediate consequences of a more general result in the theory of simplicial vector bundles: the representation up to homotopy obtained by splitting a higher vector bundle by means of a cleavage is, to within isomorphism, independent of the choice of the cleavage.
\end{abstract}

\tableofcontents

\section{Introduction}

Higher groups and groupoids are of ever growing importance for physics and mathematical physics, as well as for pure mathematics. On the mathematics side, they find application in various areas, including homotopy theory \cite{BHS11}, Poisson geometry and its generalizations \cite{MT11}, and the study of nonabelian cohomology and gerbes in algebraic geometry \cite{BM05,Bre90,Gir71}. On the physics side, they appear naturally in the description of anomalies, higher-form symmetries, and topological defects in quantum field theory \cite{Fre14,Sha15}, in the formulation of “categorified” gauge field theories suitable for the analysis of extended objects such as strings \cite{BH11,BL04}, and in the construction of topological quantum field theories and of models for quantum gravity \cite{BBFW12}.

The notion of adjoint representation plays a fundamental role in the representation theory of Lie groups. It plays an equally important role in physics, by singling out the correct transformation rules for gauge bosons in Yang–Mills theory, on which the Standard Model of particles is based. It seems reasonable that a corresponding notion should exist, and be of comparable importance, in any meaningful generalization of representation theory to higher Lie groups. It is precisely the purpose of this paper to discuss one such generalization, enabling the definition of a natural notion of adjoint representation for arbitrary higher Lie groups and groupoids. Our starting point is the well-known theory, available for ordinary Lie groupoids, elaborated by Abad and Crainic \cite{AAC13}: we make their philosophy of \emph{representations up to homotopy} \cite{AAC12,AAC13,AACD11,AAS13} our own. Representations up to homotopy offer a number of practical advantages, for example, they form a dg-category with good structural and computational properties, and come with an inherent cohomology theory.

For the sake of generality and conceptual economy and also in order to avoid ad~hoc assumptions as much as possible, we shall espouse Duskin's views on higher category theory, see \cite{Dus02} and the references therein. This means that the notion of higher Lie groupoid relevant to us will be the one formulated in the language of simplicial sets; other commonly used notions, including those adopted in the references cited above, can be understood, in terms of the one we take on in this paper, as involving additional data or conditions that are completely inessential from the abstract viewpoint of general higher representation theory as discussed here.

A distinctive trait of the theory of representations up to homotopy is that the adjoint representation of a Lie groupoid is in general only well defined up to isomorphism; the ambiguity in its definition comes from the fact that its construction requires choosing an Ehresmann connection on the groupoid \cite{AAC13}. Our approach employs a far-reaching generalization of the same idea: we use certain analogs of Ehresmann connections, called \emph{cleavages}, to turn simplicial vector bundles of a suitable kind, called \emph{higher vector bundles} or \emph{vector fibrations}, into representations up to homotopy. The relevant framework, elaborated in \cite{2018a}, is a simultaneous extension of the classical Dold–Kan correspondence and of the familiar Grothendieck construction to higher vector bundles, and includes earlier work of Gracia-Saz and Mehta \cite{GSM17} as a special case. The tangent bundle of a higher Lie groupoid $G$ is an example of a higher vector bundle; when applied to it, our \emph{splitting construction} produces a representation up to homotopy which may be regarded naturally as an incarnation of the adjoint representation of $G$.

This is as far as the theoretic tools of \cite{2018a} enable us to go. We need to introduce new ones in order to prove that our model for the adjoint representation is, to within isomorphism, independent of the cleavage we use to build it. To achieve that, we are going to extend the splitting construction of \cite{2018a} for higher vector bundles equipped with cleavages to \emph{morphisms} between them. When applied to the identity transformation of one such higher vector bundle into itself equipped with a different choice of cleavage, our splitting construction for morphisms produces an \emph{isomorphism} between the two corresponding representations up to homotopy: the isomorphism class of the representation up to homotopy obtained by splitting a higher vector bundle by means of a cleavage is an intrinsic invariant of the higher vector bundle. We regard the latter result, which generalizes those of \cite{AAC13,GSM17}, as the chief contribution of this paper; not only does it imply that the adjoint representation is well defined, it is also crucial for the overall conceptual consistency of the approach to higher representation theory set forth in \cite{2018a,2022a}. Its practical implications are just beginning to be explored.

Section \ref{sec:morphism} of this article covers material that is original even in relation to \cite{2018a,2022a}, and contains most of our new contributions. Our principal result, Theorem \ref{thm:2018b} (along with its corollary, \ref{cor:2018b}), is stated right at the beginning of that section. A bird's eye view of our splitting construction for morphisms is given next, in section \ref{sub:overview}. Understandably, most proofs are left out; we take them on, with the due diligence, in the two successive sections, \ref{sub:auxiliary} and \ref{sub:canonical}, which are genuinely technical and may be skipped on a first reading. Section \ref{sub:final} brings our paper to an all-embracing conclusion by illustrating the key ideas involved in the demonstration of yet another fundamental fact in our theory: the Dold–Kan “pseudofunctor” defined by our splitting construction descends to an equivalence between the homotopy category of higher vector bundles and the derived category of representations up to homotopy. Many details are omitted but we plan to make them available as part of a separate publication. As to the remaining sections, \ref{sub:cleavages} and \ref{sub:splitting} review a selection of topics from \cite{2018a} which constitute the indispensable conceptual background for the subject matter of section \ref{sec:morphism}. Section \ref{sub:coherent} stands a bit on its own in the overall design of the paper: while consisting of examples addressing certain possible objections, it does not otherwise affect the rest of the paper. Finally, section \ref{sec:representations}, being devoted for the most part to basic preliminaries, contains little or no original material.

\paragraph*{Acknowledgments.} The author wishes to thank Mi\-quel Cue\-ca, John Huerta, Roger Picken, Ste\-fa\-no Ron\-chi, Jim Stasheff, and Chen\-chang Zhu for their interest and for stimulating discussions.

\section{Representations up to homotopy of higher Lie groupoids}\label{sec:representations}

In this preliminary section we review a few basic concepts and facts which are essential for the understanding of our paper. Although some of the material may not be entirely known or easy to find in the literature, we make no claims to originality. Given the elementary level of the exposition, we advise experts to skip the present section altogether and to go directly to page \pageref{sec:splitting}.

The finite ordinals $[n] = \{0,1,\dotsc,n\}$, $n \geq 0$, are the objects of a full subcategory, $\Cat{\Delta}$, of the category of all posets (partially ordered sets). For an arbitrary category, $\cat{C}$, we refer to the category $[\Cat{\Delta}^\op,\cat{C}]$ of contravariant functors from $\Cat{\Delta}$ to $\cat{C}$ and their natural transformations as the category of \emph{simplicial objects} of $\cat{C}$. When $\cat{C} = \Set$ is the category of sets, we write $\sSet = [\Cat{\Delta}^\op,\Set]$, and speak of \emph{simplicial sets} and \emph{simplicial maps}; when $\cat{C} = \Mfd$ is the category of smooth manifolds (of class $C^\infty$, Hausdorff, and paracompact), we write $\sMfd = [\Cat{\Delta}^\op,\Mfd]$, and speak of \emph{simplicial manifolds} and \emph{smooth simplicial maps}.

Let $X$ be a simplicial object of $\cat{C}$. We abbreviate $X([n])$ to $X_n$. Any poset map $\theta: [m] \to [n]$ gives rise to a morphism $X_\theta = X(\theta): X_n \to X_m$ in $\cat{C}$. Of special relevance are the \emph{face} morphisms $d_i = X_{\delta_i}: X_n \to X_{n-1}$, where for each $i \in [n]$, $n \geq 1$ we let $\delta_i = \delta^n_i: [n - 1] \to [n]$ denote the injection whose image does not contain $i$, and the \emph{degeneracy} morphisms $u_j = X_{\upsilon_j}: X_n \to X_{n+1}$, where for each $j \in [n]$, $n \geq 0$ we let $\upsilon_j = \upsilon^n_j: [n + 1] \to [n]$ denote the surjection satisfying $\upsilon_j(j) = \upsilon_j(j + 1)$. In terms of these we can for each $k = 0$,~$\dotsc$,~$n$ define the \emph{back} (resp.~\emph{front}) \emph{$k$-face} morphism $s_k = d_{k+1} \dotsm d_n$ (resp.~$t_k = d_0 \dotsm d_0$) from $X_n$ to $X_k$, with the understanding that $s_n = t_n = \id$. We call $s = s_0$ the \emph{source} and $t = t_0$ the \emph{target}. We further set $1_n = u_0 \dotsm u_0: X_0 \to X_n$ and call $1 = 1_1$ the \emph{unit}. For any simplicial morphism $f: X \to Y$ in $\cat{C}$ we write $f_n = f([n]): X_n \to Y_n$. We shall be making liberal use of these notations throughout the paper.

In the theory of simplicial sets, the elements of $X_n$ are called \emph{$n$-simplices}, the $0$-simplices are also called \emph{vertices}, and the $1$-simplices \emph{edges}. To keep the typography terse we shall for all $n$-simplices $x \in X_n$ and poset maps $\theta: [m] \to [n]$ write $X_\theta x$ or $x\theta$ in place of $X_\theta(x) \in X_m$. For any simplicial map $f: X \to Y$ we shall write $fx = f_n(x) \in Y_n$. A basic operation on simplicial sets of which we shall make extensive use is the \emph{product}: for any $X$,~$Y \in \sSet$ this is given by $(X \times Y)_n = X_n \times Y_n$ for all $n$ and by $(X \times Y)_\theta = X_\theta \times Y_\theta$ for all $\theta$.

As a basic example of a simplicial set, consider an arbitrary small category $C_1 \tto C_0$: its \emph{nerve} is the simplicial set $C$ constructed as follows. For $n \geq 2$, \[%
	C_n = \{x_n \xfrom{c_n} x_{n-1} \xfrom{c_{n-1}} \dotsb \xfrom{c_2} x_1 \xfrom{c_1} x_0\}
\] is the set of all length~$n$ composable strings of arrows $c_i \in C_1$ of $C_1 \tto C_0$; $d_0$ substitutes $x_1$ for $x_1 \xfrom{c_1} x_0$; $d_i$ replaces $x_{i+1} \xfrom{c_{i+1}} x_i \xfrom{c_i} x_{i-1}$ with $x_{i+1} \xfrom{c_{i+1}c_i} x_{i-1}$ for $0 < i < n$; $d_n$ substitutes $x_{n-1}$ for $x_n \xfrom{c_n} x_{n-1}$; $u_j$ replaces $x_j$ with $x_j \xfrom{1x_j} x_j$ (unit arrow).

An arbitrary poset $P = (P,\leq)$ may be regarded as a small category $C_1 \tto C_0$ having $C_0 = P$ and exactly one arrow $x_0 \to x_1$ whenever $x_0 \leq x_1$. The elements of $C_n$ may be identified with the order-preserving maps $[n] \to P$, and $C_\theta: C_n \to C_m$ is given for each $\theta: [m] \to [n]$ by composition with $\theta$. In the specific context of posets we shall occasionally write $\nerve P$ for the nerve of $P$. Any order-preserving map $P \to Q$ between two posets gives rise to a corresponding simplicial map $\nerve P \to \nerve Q$ between their nerves; we shall normally blur the distinction between the two maps. We have an obvious identification of simplicial sets \(%
	\nerve(P \times Q) \simeq \nerve P \times \nerve Q
\), natural in $P$ and $Q$, where $P \times Q$ denotes the product poset.

\subsection{Lie infinity-groupoids and regular Kan fibrations}\label{sub:fibrations}

Useful references for this section are \cite{BG17,Dus02,Hen08,Zhu09} in addition to the two classical monographs \cite{GZ67,May67}.

Among all simplicial sets the \emph{fundamental\/ $n$-simplex} $\varDelta^n = \Cat{\Delta}(-,[n])$ (the functor $\Cat{\Delta}^\op \to \Set$ represented by the object $[n]$ of $\Cat{\Delta}$) and, for any pair of integers $n \geq k \geq 0$, the \emph{fundamental\/ $n,k$-horn} $\varLambda^n_k \subset \varDelta^n$ (the simplicial subset of $\varDelta^n$ consisting of all those poset maps $\theta: [m] \to [n]$ that satisfy the condition $\im(\theta) \not\supset [n] \smallsetminus \{k\}$) play a pivotal role.

Let $X$ be a simplicial set. For any other simplicial set $A$ let $A(X) = \sSet(A,X)$ be the set of all simplicial maps $A \to X$. There is a natural identification $\varDelta^n(X) \simeq X_n$ between simplicial maps $\varDelta^n \xto{x} X$ and $n$-simplices $x \in X_n$ (Yoneda lemma). There is a similar identification between elements $\varLambda^n_k \xto{x} X$ of $\varLambda^n_k(X)$ and length~$n$ sequences \[%
	(x_i)_{i\neq k} = (x_0,\dotsc,x_{k-1},x_{k+1},\dotsc,x_n)
\] of elements $x_i$ of $X_{n-1}$ satisfying $d_ix_j = d_{j-1}x_i$ for all $i < j$. We refer to $\varLambda^n_k(X)$ as the set of all \emph{$n,k$-horns} in $X$. Each $n$-simplex $\varDelta^n \xto{x} X$ gives rise upon restriction to a corresponding $n,k$-horn, $x \mathbin| \varLambda^n_k = (d_ix)_{i\neq k}$, but unless $X$ is a simplicial set of a special type known as a \emph{Kan complex} not every $n,k$-horn arises in this way. If an $n,k$-horn is of the form $x \mathbin| \varLambda^n_k$ we say it \emph{can be filled} (by at least one and in principle more than one $x$). More generally given a simplicial map $f: X \to Y$, an $n,k$-horn $(x_i)_{i\neq k}$ in $X$, and an $n$-simplex $y$ in $Y$ for which $y \mathbin| \varLambda^n_k = (fx_i)_{i\neq k}$, we may ask whether $(x_i)_{i\neq k}$ can be filled by an $n$-simplex $x \in X_n$ sitting over $fx = y$. If it can, we say that the \emph{Kan lifting problem} for $f$ is solvable for the givens $(x_i)_{i\neq k}$ and $y$. If the Kan lifting problem for $f$ is solvable for all givens $(x_i)_{i\neq k}$, $y$ for all $n \geq k \geq 0$, we call $f$ a \emph{Kan fibration}. For any given integer $N \geq 0$ we express the circumstance that a solution $x$ not only exists but is also unique for all $n > N$ by saying that $f$ is an \emph{$N + 1$-strict} Kan fibration. Note that the Kan lifting condition for $n = 0$ already conveys nontrivial information: it tells us that the map $f_0: X_0 \to Y_0$ must be onto.

Suppose $X$ is a simplicial \emph{manifold} now. In this case, rather than as a set, it is more natural to view $A(X)$ as a \emph{functor} $\Mfd^\op \to \Set$: to each smooth manifold $S$, associate the set $A(X)(S)$ of all those maps $x: S \to \sSet(A,X)$ which are “smooth” in the sense that for every $n$-simplex $a$ in $A$ the map $S \to X_n$ given by $s \mapsto x(s)a$ is smooth. It will always be clear from the context whether by $A(X)$ we mean the set or the functor. Note that $A(X)$ is not only a functor but actually a \emph{sheaf} (it is separated and satisfies the familiar gluing axiom). We shall be interested in the case when $A(X)$ is \emph{representable} meaning $A(X) \simeq \Mfd(-,M)$ for some smooth manifold $M$. In such case, on the set $A(X)$ there will be a unique differentiable structure for which we can take $M = A(X)$.

A simplicial manifold $X$ is a \emph{Lie\/ $\infty$-groupoid} or a \emph{higher Lie groupoid} if for all $n \geq k \geq 0$ the sheaves $\varLambda^n_k(X)$ are representable and the restriction maps \[%
	X_n \simeq \varDelta^n(X) \longto \varLambda^n_k(X), \quad%
		x \mapsto x \mathbin| \varLambda^n_k
\] are surjective submersions. If in addition the same maps are diffeomorphisms for all $n$ strictly greater than a given integer $N \geq 0$, $X$ is a \emph{Lie\/ $N$-groupoid}. The underlying simplicial set of a Lie $\infty$-groupoid is a Kan complex. It happens that for a Lie $\infty$-groupoid $X$ the sheaves $A(X)$ and $B(X)$ are representable and the “restriction” morphisms
\begin{equation}
	i^* = i(X): A(X) \longto B(X), \quad%
		A(X)(S) \ni x \longmapsto \{s \mapsto x(s) \circ i\} \in B(X)(S)
\end{equation}
(viewed as smooth maps) are surjective submersions for a class of monic simplicial maps $B \xto{i} A$ much larger than the class of all horn inclusions $\varLambda^n_k \xto\subset \varDelta^n$, $n \geq k \geq 0$.

\begin{lem}\label{lem:17A.25.1} For any simplicial manifold\/ $X$, the class\/ $\mathcal{R} = \mathcal{R}_X$ of all simplicial maps\/ $B \xto{i} A$ whose corresponding sheaf morphisms\/ $i^*: A(X) \to B(X)$ are surjective submersions is\/ \emph{regularly saturated} in the sense that it is closed under the following formation rules:
\begin{enumerate}
\def\labelenumi{\upshape(\roman{enumi})}
 \item Whenever\/ $B \xto{i} A$ lies in\/ $\mathcal{R}$, so do both\/ $A \xto{\id} A$ and\/ $B \xto{\id} B$ (short: $A$,~$B \in \mathcal{R}$).
 \item An isomorphism of simplicial sets\/ $B \simto A$ belongs to\/ $\mathcal{R}$ whenever\/ $A \in \mathcal{R}$.
 \item In any pushout diagram\/ \(%
\xymatrix@=1.67em{%
 B \ar[d]^-i \ar[r]
 &	B' \ar[d]^-{i'}
\\ A \ar[r]
 &	A'
}\) in which both\/ $i$ and\/ $B'$ belong to\/ $\mathcal{R}$, the pushout map, $i'$, belongs to\/ $\mathcal{R}$ as well.
 \item In any retraction diagram\/ \(%
\xymatrix@=1.67em{%
 B' \ar[d]^-{i'} \ar[r]
 &	B \ar[d]^-i \ar[r]
	&	B' \ar[d]^-{i'}
\\ A' \ar[r]
 &	A \ar[r]
	&	A'
}\) in which all three of\/ $i$, $A'$, and\/ $B'$ belong to\/ $\mathcal{R}$, the retract map, $i'$, belongs to\/ $\mathcal{R}$ as well.
 \item The composition, $C \xto{i\circ j} A$, of\/ $C \xto{j} B$ and\/ $B \xto{i} A$ lying in\/ $\mathcal{R}$ also lies in\/ $\mathcal{R}$.
 \item The coproduct, $i_1 \sqcup i_2: B_1 \sqcup B_2 \to A_1 \sqcup A_2$, of\/ $B_1 \xto{i_1} A_1$ and\/ $B_2 \xto{i_2} A_2$ lying in\/ $\mathcal{R}$ also lies in\/ $\mathcal{R}$.
\end{enumerate} \end{lem}

\begin{proof} Our formation rules are an adaptation of those laid down in \cite[pp.~60–61]{GZ67}. The task of checking that $\mathcal{R}$ is closed under them is a straightforward application of the basic properties of submersions. \end{proof}

\begin{defn}\label{defn:17A.25.2} The \emph{regular anodyne extensions} are the members of the smallest regularly saturated class of simplicial maps containing every horn inclusion $\varLambda^n_k \xto\subset \varDelta^n$, $n \geq k \geq 0$. \end{defn}

Let us now consider an arbitrary smooth simplicial map $f: X \to Y$, a “parametric family” of simplicial manifolds, rather than a single simplicial manifold $X$. For any simplicial set $A$, we obtain a morphism of $A(X)$ to $A(Y)$ in $[\Mfd^\op,\Set]$ upon setting%
\begin{equation}
	f_* = A(f): A(X) \longto A(Y), \quad%
		A(X)(S) \ni x \longmapsto \{s \mapsto f \circ x(s)\} \in A(Y)(S).
\label{eqn:A(f)}
\end{equation}
This prescription defines a \emph{functor} $X \mapsto A(X)$ from $\sMfd$ to sheaves in $[\Mfd^\op,\Set]$. For every simplicial map $B \xto{i} A$ it makes perfect sense within the (complete) category $[\Mfd^\op,\Set]$ to talk about the morphism of sheaves
\begin{equation}
	(f_*,i^*): A(X) \longto A(Y) \times_{B(Y)} B(X).
\label{eqn:17A.26.1}
\end{equation}
Our previous lemma admits the following generalization, whose proof is left to the reader.

\begin{lem}\label{lem:17A.26.1} For any smooth simplicial map\/ $f: X \to Y$ the class\/ $\mathcal{R} = \mathcal{R}_f$ of all those simplicial maps\/ $B \xto{i} A$ belonging to\/ $\mathcal{R}_X \cap \mathcal{R}_Y$ for which\/ \eqref{eqn:17A.26.1} is a surjective submersion is regularly saturated in other words is closed under the formation rules\/ {\upshape (i)–(vi)} of\/ {\upshape Lemma \ref{lem:17A.25.1}}. \qed \end{lem}

\begin{defn}\label{defn:17A.26.2} The smooth simplicial map $f: X \to Y$ is a \emph{regular fibration} if every horn inclusion $\varLambda^n_k \xto\subset \varDelta^n$, $n \geq k \geq 0$, belongs to $\mathcal{R}_f$. \end{defn}

\begin{prop}\label{prop:17A.26.4} Let\/ $f: X \to Y$ be a regular fibration. Given any regular anodyne extension\/ $i: B \to A$, both\/ $f_* = A(f)$ defined by\/ \eqref{eqn:A(f)} above and the following are surjective submersions.
\begin{equation}
	\pr: A(Y) \times_{B(Y)} B(X) \longto A(Y)
\end{equation} \end{prop}

\begin{proof} We observe first of all that our claim about $f_*$ is equivalent to that about $\pr$, for every smooth simplicial map $f: X \to Y$, for any simplicial map $i: B \to A$ belonging to $\mathcal R_f$. It is not hard then to check that the subcollection of $\mathcal R_f$ consisting of all those $i: B \to A$ for which either claim holds is itself a regularly saturated class of simplicial maps and that it contains every horn inclusion whenever $f$ is a regular fibration. \end{proof}

\begin{lem}\label{lem:17A.28.1} Suppose a coequalizer is given in\/ $\sSet$ of the form
\begin{equation*}
\xymatrix{%
 C \ar@<+.333ex>[r] \ar@<-.333ex>[r]
 &	\coprod B^\rho \ar[r]^-{\coprod i^\rho}
	&	A}
\end{equation*}
where\/ $\{i^\rho: B^\rho \to A\}$ is a family of simplicial maps. For every smooth manifold\/ $S$ and for every simplicial manifold\/ $X$, a map\/ $x: S \to A(X)$ is smooth iff so are all maps\/ $i^\rho(X) \circ x: S \to B^\rho(X)$. \end{lem}

\begin{proof} The ‘only if’ direction is clear. As to the ‘if’ direction, we note that each $a \in A_n$ must be equal to $i^\rho b$ for some $b \in B^\rho_n$ because a simplicial map is epic iff it is onto on $n$-simplices for each $n$ \cite[p.~21, II.1.1]{GZ67}; the map $S \to X_n$, $s \mapsto x(s)a = x(s)i^\rho b = [i^\rho(X) \circ x](s)b$ must then be smooth since so is $i^\rho(X) \circ x$ by hypothesis. \end{proof}

\subsection{The dg-category of representations up to homotopy}\label{sub:intertwiners}

We follow \cite{AAC13,AACD11,AAS13}. Let $G$ be a Lie $\infty$-groupoid or, for that matter, an arbitrary simplicial manifold: the notion of representation up to homotopy makes sense in such generality. We regard $G$ as fixed throughout the present section.

Let $E = \bigoplus_{n\in\Z} E^{-n}$ be a smooth graded vector bundle over $G_0$ indexed by the integers. Write $E_x$ for the fiber of $E$ at $x \in G_0$ (a $\Z$-graded vector space); we refer to $E^{-n}_x$ as the \emph{cohomological degree\/~$-n$} or \emph{homological degree\/~$+n$} homogeneous component of $E_x$. Whenever we say ‘degree’ without further specification, we mean ‘cohomological degree’.

By definition, a \emph{representation up to homotopy} of $G$ on $E$ comprises a sequence $R = \{R_m\}_{m\geq 0}$ of smooth vector bundle morphisms $R_m^{-n}: s^*E^{-n} \to t^*E^{1-m-n}$ covering the identity transformation of $G_m$ which satisfy the equations below for all $g \in G_m$, $m = 0$,~$1$,~$2$,~$\dotsc$,
\begin{subequations}
\label{eqn:RUTH}
\begin{equation*}
	\sum_{k+l=m} (-1)^kR_k(t_kg)R_l(s_lg) = \sum_{i=1}^{m-1} (-1)^{m-i}R_{m-1}(d_ig)
\tag{\theparentequation}
\end{equation*}
where $R_m(g) = R_m^{-n}(g)$ is the linear map, assigning a vector $R_m^{-n}(g)e$ in $E^{1-m-n}_{tg}$ to each vector $e$ in $E^{-n}_{sg}$, that $R_m^{-n}$ induces upon restriction.%
\footnote{%
 We stick to the traditional “cohomological” notation for representations up to homotopy \cite{AAC13,AACD11,AAS13} as we find it somewhat handier than its “homological” counterpart \cite{2018a,2022a}.
} %
The first three of these equations read
\begin{alignat}{2}
	m&= 0\colon &\quad & R_0(x)R_0(x) = 0
\label{eqn:RUTH0}
\\	m&= 1\colon &\quad & R_0(tg)R_1(g) = R_1(g)R_0(sg)
\label{eqn:RUTH1}
\\	m&= 2\colon &\quad & R_1(d_0g)R_1(d_2g) - R_1(d_1g) = R_0(tg)R_2(g) + R_2(g)R_0(sg)
\label{eqn:RUTH2}
\end{alignat}
\end{subequations}
and have the following meaning: \eqref{eqn:RUTH0} says that for each vertex $x \in G_0$ the linear map $R_0(x)$ is a cochain differential on $E_x$ turning $E_x$ into a cochain complex of vector spaces; \eqref{eqn:RUTH1} says that for every edge $g \in G_1$ the linear map $R_1(g)$ is a homomorphism $E_{sg} \to E_{tg}$ of cochain complexes of vector spaces; \eqref{eqn:RUTH2} expresses the failure of $R_1$ to be compatible with the “composition” of edges in $G$ in terms of the “curvature” tensor $R_2$.

We shall primarily be interested in representations up to homotopy which are \emph{unital} in the sense that they enjoy the properties: $R_1(1x) = \id$ for all $x$ in $G_0$; $R_{m+1}(u_jg) = 0$ for $j = 0$,~$\dotsc$,~$m$ for all $g$ in $G_m$, $m \geq 1$.

Let $E = (E,R)$, $F = (F,S)$ be unital representations up to homotopy of $G$. An \emph{intertwiner} $\varPhi: E \to F$ is a sequence $\varPhi = \{\varPhi_m\}_{m\geq 0}$ of smooth vector bundle morphisms $\varPhi_m^{-n}: s^*E^{-n} \to t^*F^{-m-n}$ covering the identity transformation of $G_m$ which satisfy the equations
\begin{subequations}
\label{eqn:intertwiner}
\begin{equation*}
 \sum_{k+l=m} S_k(t_kg)\varPhi_l(s_lg)
	= \sum_{k+l=m} (-1)^k\varPhi_k(t_kg)R_l(s_lg) - \sum_{i=1}^{m-1} (-1)^{m-i}\varPhi_{m-1}(d_ig)
\tag{\theparentequation}
\end{equation*}
as well as the conditions $\varPhi_{m+1}(u_jg) = 0$ for all $m \geq 0$, $j = 0$,~$\dotsc$,~$m$, and $g \in G_m$. The first two of these equations, namely,
\begin{alignat}{2}
	m&= 0\colon &\quad & S_0(x)\varPhi_0(x) = \varPhi_0(x)R_0(x)
\label{eqn:intertwiner0}
\\	m&= 1\colon &\quad & \varPhi_0(tg)R_1(g) - S_1(g)\varPhi_0(sg) = S_0(tg)\varPhi_1(g) + \varPhi_1(g)R_0(sg)
\label{eqn:intertwiner1}
\end{alignat}
\end{subequations}
may be understood as postulating that $\varPhi_0$ be a homomorphism of cochain complexes of vector bundles from $(E,R_0)$ to $(F,S_0)$ and that its failure to intertwine the two “pseudo-representations” $R_1$ and $S_1$ be measured by $\varPhi_1$. A \emph{strict} intertwiner is one such that $\varPhi_m = 0$ for all $m \geq 1$; for any such $\varPhi$, equation \eqref{eqn:intertwiner} boils down to
\begin{equation*}
	S_m(g)\varPhi_0(sg) = \varPhi_0(tg)R_m(g).
\end{equation*}

Let $F(\varPhi) \xfrom{\varPhi} E(\varPhi) = F(\varPsi) \xfrom{\varPsi} E(\varPsi)$ be any two composable intertwiners of representations up to homotopy of $G$. Their \emph{composition}, $F(\varPhi) \xfrom{\varPhi\circ\varPsi} E(\varPsi)$, is given by the formula
\begin{subequations}
\label{eqn:composition}
\begin{equation*}
	(\varPhi \circ \varPsi)_m(g) = \sum_{k=0}^m \varPhi_k(t_kg)\varPsi_{m-k}(s_{m-k}g)
\tag{\theparentequation}
\end{equation*}
\cite[§2]{AACD11}. By way of example:
\begin{alignat}{2}
	m&= 0\colon &\quad & (\varPhi \circ \varPsi)_0(x) = \varPhi_0(x)\varPsi_0(x)
\label{eqn:composition0}
\\	m&= 1\colon &\quad & (\varPhi \circ \varPsi)_1(g) = \varPhi_0(tg)\varPsi_1(g) + \varPhi_1(g)\varPsi_0(sg)
\label{eqn:composition1}
\end{alignat}
\end{subequations}
It is straightforward to check that \eqref{eqn:composition} does in fact define a new intertwiner of representations up to homotopy and that the composition law thus obtained turns unital representations up to homotopy of $G$ into the objects of a category hereafter notated $\Rep^\infty(G)$; for each $E \in \Rep^\infty(G)$ the strict intertwiner $I: E \to E$ given by $I_0(x) = \id$ is the identity automorphism of $E$ under the composition law \eqref{eqn:composition}. Given an integer $N \geq 0$, we shall write $\Rep^N(G)$ for the full subcategory of $\Rep^\infty(G)$ comprising all those $E$ having $E^{-n} = 0$ for $\lvert n\rvert > N$. We shall further write $\Rep^N_+(G)$ for the full subcategory of $\Rep^N(G)$ comprising all \emph{$N + 1$-term} representations up to homotopy ($0 \leq N \leq \infty$) i.e.~all $E \in \Rep^N(G)$ that vanish in every negative homological degree as well and thus have $E^{-n} = 0$ for both $n < 0$ and $n > N$.

It is possible to promote $\Rep^\infty(G)$ to a dg-category (differential graded category) as follows. Let $E = (E,R)$, $F = (F,S)$ be representations up to homotopy of $G$ (as always, unital). Write $\Hom(E,F)$ for the graded vector space whose degree~$k$ homogeneous component $\Hom^k(E,F)$ consists of all possible sequences $\varPhi = \{\varPhi_m\}_{m\geq 0}$ of smooth vector bundle morphisms $\varPhi_m^{-n}: s^*E^{-n} \to t^*F^{k-m-n}$ covering the identity transformation of $G_m$ that satisfy ${u_j}^*\varPhi_{m+1} = 0$ for all $0 \leq j \leq m$; write $k = \deg\varPhi \in \Z$ and call this the degree of $\varPhi$. The \emph{derivative} of $\varPhi$ is the degree~$\deg\varPhi + 1$ homogeneous vector $D_{E,F}\varPhi \in \Hom(E,F)$ given by
\begin{align}
 (D_{E,F}\varPhi)_m(g)
	&= \sum_{k+l=m} (-1)^{k\deg\varPhi}S_k(t_kg)\varPhi_l(s_lg)
	   + (-1)^{\deg\varPhi}\sum_{i=1}^{m-1} (-1)^{m-i}\varPhi_{m-1}(d_ig) \notag\\* &\justify%
	   - (-1)^{\deg\varPhi}\sum_{k+l=m} (-1)^k\varPhi_k(t_kg)R_l(s_lg).
\label{eqn:17A.38.2}
\end{align}
For $\deg\varPhi = 0$ the condition $D_{E,F}\varPhi = 0$ is tantamount to saying that $\varPhi$ is an intertwiner $E \to F$ \eqref{eqn:intertwiner}. For $\varPhi$,~$\varPsi$,~$\varOmega \in \Hom(E,F)$ homogeneous of degrees $\deg\varPhi = \deg\varPsi = 0$, $\deg\varOmega = -1$ the equation $D_{E,F}\varOmega = \varPsi - \varPhi$ defines the concept of $\varOmega$ being a \emph{homotopy} $\varPhi \hto \varPsi$:
\begin{align}
 (\varPsi - \varPhi)_m(g)
	&= \sum_{k+l=m} (-1)^k\{S_k(t_kg)\varOmega_l(s_lg) + \varOmega_k(t_kg)R_l(s_lg)\}
	   - \sum_{i=1}^{m-1} (-1)^{m-i}\varOmega_{m-1}(d_ig).
\label{eqn:homotopy}
\end{align}
The \emph{composition} of $\varPhi \in \Hom\bigl(E(\varPhi),F(\varPhi)\bigr)$, $\varPsi \in \Hom\bigl(E(\varPsi),F(\varPsi)\bigr)$ homogeneous with $E(\varPhi) = F(\varPsi)$ is the degree~$\deg\varPhi + \deg\varPsi$ homogeneous vector $\varPhi \circ \varPsi \in \Hom\bigl(E(\varPsi),F(\varPhi)\bigr)$ given by
\begin{equation}
	(\varPhi \circ \varPsi)_m(g) = \sum_{k=0}^m (-1)^{k\deg\varPsi}\varPhi_k(t_kg)\varPsi_{m-k}(s_{m-k}g).
\label{eqn:17A.38.1}
\end{equation}
These formulas go hand in hand with the philosophy of \cite[§3.1, Rmk.~3.8]{AAC13}. It is easy to check that the operator $D_{E,F}$ is linear, squares to zero, and satisfies the familiar Leibniz rule with respect to the composition law \eqref{eqn:17A.38.1} (we are suppressing subscripts for readability):
\begin{equation}
	D(\varPhi \circ \varPsi) = D\varPhi \circ \varPsi + (-1)^{\deg\varPhi}\varPhi \circ D\varPsi.
\label{eqn:17A.38.3}
\end{equation}

\subsubsection*{A criterion for the invertibility of intertwiners}

We shall now take advantage of the above dg-categorical structure and use it to characterize the isomorphisms of the category $\Rep^\infty_+(G)$ (the invertible intertwiners). We start with a series of lemmas that constitute an adaptation of the proof of \cite[Prop.~3.28]{AAC13} to the specific setup of this paper. Our exposition is meant to be as accessible and self-contained as possible.

\begin{lem}\label{lem:isomorphism} Let\/ $E = (E,R)$ be a representation up to homotopy of\/ $G$. Let\/ $\varOmega \in \End^{-1}(E)$ be a degree minus one cochain for the dg-categorical structure defined by equations\/ \eqref{eqn:17A.38.2}, \eqref{eqn:17A.38.1}. Suppose that for some\/ $m \geq 0$, $n \in \Z$ the identities\/ \(%
	(I + D_E\varOmega)_k^{-l} = 0
\) hold for\/ $k = m$, $l < n$ as well as for\/ $k < m$, where\/ $D_E = D_{E,E}$. Then\/ \(%
	t^*R_0^{-m-n} \circ (D_E\varOmega)_m^{-n} = 0
\). \end{lem}

\begin{proof} Let us fix $g \in G_m$, $e \in E^{-n}_{sg}$. We have $R_0(sg)e \in E^{-l}_{sg}$, where $l = n - 1 < n$, so we can use our hypothesis to substitute for all terms $R_0(tg)\varOmega_k(t_kg)$, $R_0(tg)\varOmega_{m-1}(d_ig)$ in the expansion of $R_0(tg)\bigl((D_E\varOmega)_m(g)e\bigr)$ furnished by \eqref{eqn:homotopy}. After suitable rearrangements and the cancellation of a number of pairwise opposite terms, we are left with an expression to which we can once more apply our hypothesis in such a way as to end up with only $R$~terms to the left and $\varOmega$~terms to the right. The resulting expression is immediately recognized to vanish on account of the equations \eqref{eqn:RUTH} that the tensors $R_m$ are supposed to satisfy. \end{proof}

\begin{lem}\label{lem:isomorphism*} Let\/ $E = (E,R)$ be an object of\/ $\Rep^\infty(G)$ for which there exists\/ $n_0 \in \Z$ such that\/ $n < n_0$ implies\/ $E^{-n} = 0$. Any homotopy\/ $\varOmega_0: \id \hto 0$ of cochain operators on the cochain complex\/ $(E,R_0)$ is part of some homotopy\/ $\varOmega: I \hto 0$ of self-intertwiners of\/ $E$. \end{lem}

\begin{proof} There is a unique degree minus one cochain $\varOmega \in \End^{-1}(E)$ with the property that
\begin{equation}
	\varOmega_m(g) = \varOmega_0(tg) \circ \bigl((D_E\varOmega)_m(g) - R_0(tg)\varOmega_m(g)\bigr)
\text{\quad $\forall m \geq 0$, $g \in G_m$.}
\label{eqn:aux1}
\end{equation}
This can be seen by an easy inductive argument: we are given $\varOmega_0(x)$ by hypothesis and have no other option than defining $\varOmega_m^{-n}(g)$ to be zero for $n < n_0$. We contend that $\varOmega$ is the required homotopy $I \hto 0$ in other words $(I + D_E\varOmega)_m^{-n} = 0$ for all $m$, $n$. This is the case for $m = 0$ by hypothesis. We prove this is also the case for $m \geq 1$, $n \geq n_0$ by induction. Suppose namely that $(I + D_E\varOmega)_k^{-l} = 0$ for $k = m$, $l < n$ as well as for $k < m$. Then for all $g \in G_m$, $e \in E^{-n}_{sg}$
\begin{alignat*}{2}
 -R_0(tg)\varOmega_m(g)e
	&= -R_0(tg)\varOmega_0(tg) \circ (\mathellipsis)e
			&\quad &\text{by \eqref{eqn:aux1}} \\
	&= (\mathellipsis)e + \varOmega_0(tg)R_0(tg) \circ (\mathellipsis)e
			&\quad &\text{by \eqref{eqn:homotopy} for $m = 0$} \\
	&= (D_E\varOmega)_m(g)e - R_0(tg)\varOmega_m(g)e
			&\quad &\text{by Lemma \ref{lem:isomorphism} and \eqref{eqn:RUTH0}}
\end{alignat*}
whence as contended $(I + D_E\varOmega)_m^{-n} = I_m^{-n} + (D_E\varOmega)_m^{-n} = (D_E\varOmega)_m^{-n} = 0$. \end{proof}

The \emph{mapping cone} of an intertwiner $\varPhi: E \to F \in \Rep^\infty(G)$ is the (unital) representation up to homotopy $\cone(\varPhi)$ with underlying graded vector bundle \(%
	\cone(\varPhi)^{-n} = E^{-n} \oplus F^{-1-n}
\) and “curvature” tensors \[%
\begin{pmatrix}
	R_m^{-n}       & 0 \\
	\varPhi_m^{-n} & (-1)^{m-1}S_m^{-1-n}
\end{pmatrix}:
	s^*E^{-n} \oplus s^*F^{-1-n} \longto t^*E^{1-m-n} \oplus t^*F^{-m-n}
\] \cite[Ex.~3.21]{AAC13}. The proof of the next lemma is a straightforward explicit calculation which we omit.

\begin{lem}\label{lem:isomorphism**} For every homotopy\/ $\varOmega: I \hto 0$ of self-intertwiners of the mapping cone\/ $\cone(\varPhi)$ of an intertwiner\/ $\varPhi: E \to F \in \Rep^\infty(G)$, setting \[%
	\varOmega_m^{-n} =%
\begin{pmatrix}
		K_m^{-n} & (-1)^{m-1}\varPsi_m^{-1-n} \\
		\ast     & (-1)^{m-1}\varLambda_m^{-1-n}
\end{pmatrix}:
	s^*E^{-n} \oplus s^*F^{-1-n} \longto t^*E^{-1-m-n} \oplus t^*F^{-2-m-n}
\] defines an intertwiner\/ $\varPsi: F \to E$ plus homotopies\/ $K: \varPsi \circ \varPhi \hto I$, $\varLambda: \varPhi \circ \varPsi \hto I$. \qed \end{lem}

\begin{prop}\label{prop:isomorphism} Let\/ $E$,~$F \in \Rep^\infty_+(G)$ be unital representations up to homotopy vanishing in positive degrees. An intertwiner\/ $\varPhi: E \to F$ is\/ \emph{invertible}, i.e., is an isomorphism in the category\/ $\Rep^\infty_+(G)$, if and only if\/ $\varPhi_0^{-n}: E^{-n} \simto F^{-n}$ is an isomorphism of graded vector bundles over\/ $G_0$. \end{prop}

\begin{proof} The ‘only if’ direction is clear. For the converse, suppose that $\varPhi_0^{-n}: E^{-n} \simto F^{-n}$ admits an inverse, $\varPsi_0^{-n}: F^{-n} \simto E^{-n}$, for every $n \geq 0$. Setting \[%
	\varOmega_0^{-n} =%
\begin{pmatrix}
		0 & -\varPsi_0^{-1-n} \\
		0 & 0
\end{pmatrix}:
	E^{-n} \oplus F^{-1-n} \longto E^{-1-n} \oplus F^{-2-n}
\] defines a homotopy $\varOmega_0: \id \hto 0$ of cochain operators on the mapping cone complex of $\varPhi$. By Lemmas \ref{lem:isomorphism*} and \ref{lem:isomorphism**}, $\varPhi: E \to F$ admits a homotopy inverse $\varPsi: F \to E$ extending $\varPsi_0$; let us fix any pair of homotopies $K: \varPsi \circ \varPhi \hto I$, $\varLambda: \varPhi \circ \varPsi \hto I$.

Thanks to the invertibility of $\varPhi_0$, there is a unique degree minus one cochain in $\Hom(F,E)$ which solves for $\tilde{K}$ in $\tilde{K} \circ \varPhi = K$: we read it off from \eqref{eqn:17A.38.1} by recursion in $m \geq 0$. Let us consider the degree zero cochain \(%
	\tilde{\varPsi} = \varPsi + D_{F,E}\tilde{K}
\) in $\Hom(F,E)$. On account of the fact that $D_{F,E}$ squares to zero, $\tilde{\varPsi}$ itself must be an intertwiner $F \to E$: \[%
 D_{F,E}\tilde{\varPsi}
	= D_{F,E}\varPsi + {D_{F,E}}^2\tilde{K}
	= D_{F,E}\varPsi = 0.
\] In fact, $\tilde{\varPsi}$ is a left inverse for $\varPhi$: by our definitions, the Leibniz rule \eqref{eqn:17A.38.3}, and the identity $D_{E,F}\varPhi = 0$ (expressing the fact that $\varPhi$ is an intertwiner),
\begin{equation*}
 I - \varPsi \circ \varPhi
	= D_EK
	= D_E(\tilde{K} \circ \varPhi)
	= D_{F,E}\tilde{K} \circ \varPhi - \tilde{K} \circ D_{E,F}\varPhi
	= (\tilde{\varPsi} - \varPsi) \circ \varPhi - \tilde{K} \circ 0
	= \tilde{\varPsi} \circ \varPhi - \varPsi \circ \varPhi,
\end{equation*}
whence $\tilde{\varPsi} \circ \varPhi = I$.

A symmetric reasoning involving $\varLambda$ shows that $\varPhi$ must also be right-invertible, hence invertible. \end{proof}

\section{The splitting construction for higher vector bundles}\label{sec:splitting}

A \emph{simplicial vector bundle} is a simplicial object of $\VB$, the category of \emph{all} smooth vector bundles; any such may be visualized as a diagram of smooth vector bundle morphisms
\begin{equation*}
\xymatrix@C=3.67em{%
 \cdots \ar@<+1ex>[r] \ar@<+.333ex>[r] \ar@<-.333ex>[r] \ar@<-1ex>[r]
 &	V_2 \ar@<+.667ex>[r]^{d_i} \ar[r] \ar@<-.667ex>[r]
	\ar@/^/@<.778ex>[l] \ar@/^/@<1.444ex>[l] \ar@/^/@<2.111ex>[l]
	\ar[d]^{p_2}
	&	V_1 \ar@<+.333ex>[r] \ar@<-.333ex>[r]
		\ar@/^/@<.444ex>[l] \ar@/^/@<1.111ex>[l]^{u_j}
		\ar[d]^{p_1}
		&	V_0
			\ar@/^/@<.111ex>[l]
			\ar[d]^{p_0}
\\ \cdots \ar@<+1ex>[r] \ar@<+.333ex>[r] \ar@<-.333ex>[r] \ar@<-1ex>[r]
 &	G_2 \ar@<+.667ex>[r] \ar[r] \ar@<-.667ex>[r]
	\ar@/^/@<.778ex>[l] \ar@/^/@<1.444ex>[l] \ar@/^/@<2.111ex>[l]^{}% stet this!
	&	G_1 \ar@<+.333ex>[r] \ar@<-.333ex>[r]
		\ar@/^/@<.444ex>[l] \ar@/^/@<1.111ex>[l]
		&	G_0
			\ar@/^/@<.111ex>[l]}
\end{equation*}
and hence may be thought of as a smooth simplicial map $p: V \to G$ equipped with a fiberwise linear structure compatible with the smooth simplicial structure. We shall refer to the simplicial manifold $G$ as the \emph{base} of the simplicial vector bundle $V \xto{p} G$ and to $V \xto{p} G$ as a simplicial vector bundle \emph{over} $G$.

A fundamental example of a simplicial vector bundle is the \emph{tangent bundle} $TG \to G$ of a simplicial manifold $G$: for every integer $n \geq 0$, $(TG)_n = T(G_n) \to G_n$ is the tangent bundle of the smooth manifold $G_n$; for every poset map $\theta: [m] \to [n]$, $(TG)_\theta: (TG)_n \to (TG)_m$ is the tangent map $T(G_\theta): T(G_n) \to T(G_m)$ of the smooth map $G_\theta: G_n \to G_m$.

Given any simplicial manifold $G$, we can turn any simplicial vector bundle $V \xto{p} G$ into a cochain complex of vector bundles over $G_0$ by the following standard procedure. For all $n \geq 0$, $x \in G_0$, let us write $V^{-n}_x = (V_n)_{1_nx} = p_n^{-1}(1_nx)$ for the fiber of the vector bundle $p_n: V_n \to G_n$ at $1_nx \in G_n$. There is one and only one map $V^{-n} = \bigcup_{x\in G_0} V^{-n}_x \to G_0$ whose composition with $1_n: G_0 \to G_n$ equals the restriction of $p_n: V_n \to G_n$ to $V^{-n} \subset V_n$: we get a simplicial vector bundle over $G_0$ (that is, over the constant simplicial manifold of value $G_0$),
\begin{equation*}
\xymatrix@C=3.67em{%
 \cdots \ar@<+1ex>[r] \ar@<+.333ex>[r] \ar@<-.333ex>[r] \ar@<-1ex>[r]
 &	V^{-2} \ar@<+.667ex>[r]^{d_i} \ar[r] \ar@<-.667ex>[r]
	\ar@/^/@<.778ex>[l] \ar@/^/@<1.444ex>[l] \ar@/^/@<2.111ex>[l]
	\ar[d]
	&	V^{-1} \ar@<+.333ex>[r] \ar@<-.333ex>[r]
		\ar@/^/@<.444ex>[l] \ar@/^/@<1.111ex>[l]^{u_j}
		\ar[d]
		&	V^0
			\ar@/^/@<.111ex>[l]
			\ar[d]
\\ \cdots \ar@{=}[r]
 &	G_0 \ar@{=}[r]
	&	G_0 \ar@{=}[r]
		&	G_0}
\end{equation*}
equivalently, a simplicial object of $\VB(G_0)$, the category of all smooth vector bundles over $G_0$. If we set $V^{-n} = 0$ for $n < 0$, we obtain a $\Z$-graded vector bundle $V^\bullet = \bigoplus_{n\in\Z} V^{-n}$ over $G_0$ whose degree~$-n$ homogeneous component is $V^{-n}$. (The bullet here only serves to distinguish the graded vector bundle over $G_0$ from the simplicial vector bundle over $G$.) There is a natural cochain differential $d$ on $V^\bullet$ given by
\begin{equation}
	d^{-n} = (-1)^{n-1}\sum_{j=0}^n (-1)^jd_j: V^{-n} \longto V^{1-n} \text{,\quad $n \geq 1$.}
\end{equation}
We shall refer to the cochain complex of vector bundles vanishing in positive degrees $(V^\bullet,d)$ as the \emph{Moore complex} of the simplicial vector bundle $V \xto{p} G$.

The chief aim of this section is to describe a procedure by which under suitable assumptions we can extract a representation up to homotopy of $G$ on $(V^\bullet,d)$ (or on a closely related complex) from the simplicial vector bundle $V \xto{p} G$. Most of what we are going to say in the next two sections is drawn out from \cite[§4]{2018a}.

\subsection{Vector fibrations and cleavages}\label{sub:cleavages}

\begin{prop}\label{prop:hvb} The following are equivalent for any simplicial vector bundle\/ $p: V \to G$.
\begin{enumerate}
\def\labelenumi{\upshape (\roman{enumi})}
 \item As a smooth simplicial map, $p: V \to G$ is a regular fibration\/ {\upshape (Definition \ref{defn:17A.26.2})}.
 \item $G$ is a Lie\/ $\infty$-groupoid and, as a simplicial map, $p: V \to G$ is a Kan fibration.
 \item $G$ and\/ $V$ are Lie\/ $\infty$-groupoids.
\end{enumerate} \end{prop}

\begin{proof} We refer the reader to the text preceding \cite[Prop.~2.2]{2018a} for the equivalence $\text{(i)} \aeq \text{(ii)}$. The equivalence $\text{(i)} \aeq \text{(iii)}$ is a consequence of standard properties of vector bundle maps; cf.~the proof of \cite[Prop.~5.4]{2022a}. \end{proof}

\begin{defn}\label{defn:17A.27.1} A simplicial vector bundle satisfying the equivalent conditions of the proposition is called a \emph{vector fibration} or a \emph{higher vector bundle}. \end{defn}

\begin{prop}\label{prop:17A.27.2} The tangent bundle\/ $TG \to G$ of a simplicial manifold\/ $G$ is a vector fibration if, and only if, $G$ is a Lie\/ $\infty$-groupoid. \end{prop}

\begin{proof} The notations of section \ref{sub:fibrations} and specifically of Lemmas \ref{lem:17A.25.1} and \ref{lem:17A.26.1} are in force. For any $A \in \mathcal{R}_G$ we have an evident morphism of sheaves $\varphi_A: TA(G) \to A(TG)$, natural in $A$. The collection of all $B \xto{i} A \in \mathcal{R}_G$ for which $\varphi_A$, $\varphi_B$ are isomorphisms is easily seen to be regularly saturated and to be contained in $\mathcal{R}_{TG\to G}$ whenever $G$ is a Lie $\infty$-groupoid. We conclude by noting that if a regularly saturated class of maps contains all $n,l$-horn inclusions below a certain $m > n$ then it must contain every $\varLambda^m_k$, $0 \leq k \leq m$ as well. \end{proof}

A straightforward adaptation of the considerations of section \ref{sub:fibrations} and specifically of Proposition \ref{prop:17A.26.4} to the present context shows that for a vector fibration $V \xto{p} G$ and a regular anodyne extension $B \xto{i} A$ the two surjective submersions $A(V) \xto{p_*} A(G)$ and $A(G) \times_{B(G)} B(V) \xto{\pr} A(G)$ may be turned into smooth vector bundles over $A(G)$ in such a fashion that the map
\begin{equation}
	A(V) \xto{(p_*,i^*)} A(G) \times_{B(G)} B(V)
\label{eqn:17A.26.1*}
\end{equation}
becomes a vector bundle epimorphism: given $g \in A(G)$, the linear structure on the fiber $A(V)_g$ is characterized by the requirement that for every $n$-simplex $a \in A_n$ the map $A(V)_g \to (V_n)_{ga}$, $v \mapsto va$ be linear. Whenever $B \xto{i} A$ is one of the horn inclusions $\varLambda^n_k \xto\subset \varDelta^n$, we may view \eqref{eqn:17A.26.1*} as a morphism \(%
	V_n \to G_n \times_{\varLambda^n_k(G)} \varLambda^n_k(V)
\) of smooth vector bundles over $G_n \simeq \varDelta^n(G)$.

As a first practical illustration of these concepts we now introduce the \emph{Dold–Kan complex} or \emph{normalized Moore complex} $(\spl{V},\spl{d})$ of a vector fibration $V \xto{p} G$. We start by setting
\begin{equation}
\textstyle
	\spl{V}^{-n} = \bigcap\limits_{i=1}^n \ker(d_i: V^{-n} \to V^{1-n})
	             = \ker\bigl(V_n \to G_n \times_{\varLambda^n_0(G)} \varLambda^n_0(V)\bigr) \mathbin| 1_n(G_0)
\label{eqn:Vhat}
\end{equation}
for every $n$, with the understanding that $\spl{V}^0 = V^0$; while for a general simplicial vector bundle equation \eqref{eqn:Vhat} may not define a \emph{smooth} vector subbundle of $V^{-n}$, for a vector fibration it always does. We note next that $\spl{V} = \bigoplus_{n\in\Z} \spl{V}^{-n}$ (no bullet needed here) is a cochain subcomplex of the Moore complex $(V^\bullet,d)$: the restriction, $\spl{d}$, of the Moore differential to $\spl{V}$ is given by
\begin{equation}
	\spl{d}^{-n} = (-1)^{n-1}d_0: \spl{V}^{-n} \longto \spl{V}^{1-n} \text{,\quad $n \geq 1$.}
\label{eqn:dhat}
\end{equation}
There is a morphism \(%
	\nor: (V^\bullet,d) \to (\spl{V},\spl{d})
\) of cochain complexes of smooth vector bundles over $G_0$, natural in $V$, called the \emph{normalization map}: this is given by $V^0 \xto= \spl{V}^0$ for $n = 0$ and by
\begin{equation}
	\nor^{-n} = (\id - u_0d_1) \dotsm (\id - u_{n-1}d_n): V^{-n} \longto \spl{V}^{-n}
\label{eqn:nor}
\end{equation}
for $n \geq 1$. It is not hard to prove that this is a cochain homotopy equivalence, with quasi-inverse cochain map given by the subcomplex inclusion $(\spl{V},\spl{d}) \xto\subset (V^\bullet,d)$; the argument is substantially the same as that which can be found in \cite{GJ99}.

We shall explain in section \ref{sub:splitting} how to turn any vector fibration $V \xto{p} G$ into a representation up to homotopy $(V^\bullet,R)$ of $G$ with underlying cochain complex $(V^\bullet,d)$. We shall furthermore see that $(V^\bullet,R)$ descends uniquely to a representation up to homotopy $(\spl{V},\spl{R})$ of $G$ on $(\spl{V},\spl{d})$ in such a way that the normalization map becomes a strict intertwiner. The construction of $R$ will not be canonical, in the sense that it will depend on the choice of an auxiliary set of data known as a \emph{cleavage} of the vector fibration $V \xto{p} G$ (to be discussed next). We shall nevertheless see in section \ref{sec:morphism} that the isomorphism classes of the representations $R$ and $\spl{R}$ are independent of this choice, in fact, we shall prove that each one of them provides a well-defined invariant of vector fibrations.

\begin{defn}\label{defn:17A.27.4} Let $B \xto{i} A$ be a regular anodyne extension; a \emph{$B \xto{i} A$-cleavage} of the vector fibration $V \xto{p} G$ is a morphism $A(G) \times_{B(G)} B(V) \to A(V)$ of smooth vector bundles over $A(G)$ which cross-sections the epimorphism \eqref{eqn:17A.26.1*}. We refer to $\varLambda^n_k \xto\subset \varDelta^n$-cleavages as \emph{$n,k$-cleavages}, for short, and view them as morphisms $c_{n,k}: G_n \times_{\varLambda^n_k(G)} \varLambda^n_k(V) \to V_n$ over $G_n$. \end{defn}

Such cross-sections can always be built, of course, by making use of partitions of unity. The idea behind the concept of an $n,k$-cleavage is that $c_{n,k}$ chooses some specific solution \[%
	c_{n,k}\bigl(g,(v_i)_{i\neq k}\bigr) = c_{n,k}(g;v_0,\dotsc,v_{k-1},v_{k+1},\dotsc,v_n)
\] to the Kan lifting problem for $p$ in such a way that its dependence on the givens of the problem, $g$, $(v_i)_{i\neq k}$, is both \emph{smooth} and \emph{linear}. We shall refer to the subbundle $C_{n,k} = \im(c_{n,k}) \subset V_n$ also as an $n,k$-cleavage (this being legitimate since $c_{n,k}$ and $C_{n,k}$ determine each other uniquely).

\begin{defn}\label{defn:16A.3.1} The $n,k$-cleavage $C_{n,k}$ is \emph{normal} if it contains every degenerate $n$-simplex of $V$, in other words, $C_{n,k} \supset \im(u_j: V_{n-1} \to V_n)$ for every $0 \leq j < n$. \end{defn}

\begin{prop}\label{prop:normal} Any vector fibration admits normal\/ $n,k$-cleavages for all\/ $n \geq k \geq 0$. \end{prop}

\begin{proof} The argument in \cite[§4.4]{2018a} was written intendedly so as to remain valid without any modifications for an arbitrary vector fibration (over an arbitrary Lie $\infty$-groupoid). \end{proof}

Since the horn inclusions $\varLambda^n_k \xto\subset \varDelta^n$ generate the regular anodyne extensions $B \xto{i} A$, any choice of $n,k$-cleavages for all $n \geq k \geq 0$ gives rise to a corresponding choice of $B \xto{i} A$-cleavages for all those $B \xto{i} A$ that come with a canonical presentation as simplicial maps built from horn inclusions by iterated application of the formation rules listed in the statement of Lemma \ref{lem:17A.25.1}. This is the case for any simplicial map of the form
\begin{equation}
	A \times 0 \cup B \times I \xto\subset A \times I,
\label{eqn:sqcup}
\end{equation}
where $A \supset B$ are simplicial sets of \emph{finite type} and $I = \varDelta^1$ denotes the \emph{fundamental interval}; the reader is referred to \cite[p.~26, Lem.~4.1]{2018a} for the relevant technical details. We shall write $c_\sqcup^{A,B}$ for the corresponding “canonical” $A \times 0 \cup B \times I \xto\subset A \times I$-cleavage, where $c = \{c_{n,k}\}_{n>k\geq 0}$ indicates a choice of $n,k$-cleavages for all $n > k \geq 0$; we need not specify $n,n$-cleavages here, as they do not enter into the construction of $c_\sqcup^{A,B}$.

\begin{defn}\label{defn:cleavage} By a \emph{cleavage} $c = \{c_{n,k}\}_{n>k\geq 0}$ of the vector fibration $V \xto{p} G$ we mean a choice of $n,k$-cleavages $c_{n,k}$ for all $n > k \geq 0$. We say $c$ is \emph{normal} if so is every $C_{n,k} = \im(c_{n,k})$. \end{defn}

Our interest in the \emph{canonical} $A \times 0 \cup B \times I \xto\subset A \times I$-cleavage $c_\sqcup^{A,B}$ (as opposed to a general, “noncanonical” one) lies exclusively in that $c_\sqcup^{A,B}$ enjoys the properties summarized in our next two lemmas; any other cleavage enjoying the same properties would serve us just as well.

\begin{lem}[{\cite[p.~36]{2018a}}]\label{lem:16A.24.1} Let\/ $c = \{c_{n,k}\}_{n>k\geq 0}$ be a normal cleavage of\/ $V \xto{p} G$. Let\/ $A \supset B$ be simplicial sets, let\/ $\kappa: A' \to A$ be a simplicial map, and let\/ $B'$ be a subset of\/ $A'$ containing\/ $\kappa^{-1}(B)$. Then, for every simplicial map\/ $v: A \times I \to V$,
\begin{multline}
 v = c_\sqcup^{A,B}(pv,\left.v \mathbin| A \times 0 \cup B \times I\right.) \quad \text{implies}\\
	c_\sqcup^{A',B'}\bigl(pv \circ (\kappa \times \id),\left.v \circ (\kappa \times \id) \mathbin| A' \times 0 \cup B' \times I\right.\bigr) = v \circ (\kappa \times \id).
\label{eqn:16A.24.1}
\end{multline} \end{lem}

\begin{lem}[{\cite[p.~36]{2018a}}]\label{lem:16A.24.2} Let\/ $c = \{c_{n,k}\}_{n>k\geq 0}$ be a normal cleavage of\/ $V \xto{p} G$. Let\/ $A \supset B$ be simplicial sets. Write\/ $\pr: A \times I \to A$ for the projection. For every simplicial map\/ $v: A \to V$,
\begin{equation}
	c_\sqcup^{A,B}(pv \circ \pr,\left.v \circ \pr \mathbin| A \times 0 \cup B \times I\right.) = v \circ \pr.
\label{eqn:16A.24.2}
\end{equation} \end{lem}

\subsection{Splitting a vector fibration by a cleavage}\label{sub:splitting}

We fix an arbitrary vector fibration $V \xto{p} G$ and an arbitrary \emph{normal} cleavage $c = \{c_{n,k}\}_{n>k\geq 0}$ of $V \xto{p} G$.

Let $[1]^m$, $m \geq 0$, be the product poset $[1] \times \dotsb \times [1]$ ($m$ factors); its elements are the $m$-tuples $(r_1,\dotsc,r_m)$ with $r_i \in \{0,1\}$, the only such $m$-tuple for $m = 0$ being $(\mkern 3mu)$, the empty one. We define $I^m$, the \emph{fundamental\/ $m$-cube,} to be the nerve of $[1]^m$. We write $\delta_{i|r} = \delta^m_{i|r}: I^{m-1} \to I^m$ for the \emph{cubical face inclusion} that inserts $r = 0$,~$1$ as the $i$-th coordinate and set $\partial_{i|r}I^m = \im(\delta^m_{i|r})$ for $i = 1$,~$\dotsc$,~$m$. The \emph{total boundary} of $I^m$ is the simplicial subset $\partial I^m = \bigcup_{\substack{1\leq i\leq m \\ r=0,1\hfill}} \partial_{i|r}I^m \subset I^m$; the boundary $\partial\varDelta^n$ of $\varDelta^n$ is the union of all faces $\partial_j\varDelta^n = \im(\delta_j)$, $j = 0$,~$\dotsc$,~$n$; conventionally, both $\partial I^m$ and $\partial\varDelta^n$ are empty for $m = n = 0$. We shall be interested in the simplicial subset of $\varDelta^n \times I^m$ comprising every face but $\varDelta^n \times \partial_{m|1}I^m \subset \partial(\varDelta^n \times I^m) = \partial\varDelta^n \times I^m \cup \varDelta^n \times \partial I^m$:%
\begin{subequations}
\begin{gather}
\textstyle%
	\varPi^m = \bigl(\bigcup\limits_{1\leq i\leq m} \partial_{i|0}I^m\bigr) \cup \bigl(\bigcup\limits_{1\leq i<m} \partial_{i|1}I^m\bigr)
		\text{,\quad $m \geq 1$;\qquad $\varPi^0 = I^0$;} \\
	\varPi^{n,m} = \partial\varDelta^n \times I^m \cup \varDelta^n \times \varPi^m
		\text{,\quad $m$,~$n \geq 0$.}
\end{gather}
\end{subequations}
The inclusion $\varPi^{n,m} \xto\subset \varDelta^n \times I^m$ is easily recognized to be a regular anodyne extension; for $m \geq 1$ this is identifiable as a map of the special type $A \times 0 \cup B \times I \xto\subset A \times I$ considered in section \ref{sub:cleavages}, namely, the one having $A = \varDelta^n \times I^{m-1}$, $B = \partial\varDelta^n \times I^{m-1} \cup \varDelta^n \times \partial I^{m-1}$; we shall write $c^{n,m}$ for the $\varPi^{n,m} \xto\subset \varDelta^n \times I^m$-cleavage corresponding (for this choice of $A$ and $B$) to $c_\sqcup^{A,B}$ under the identification of simplicial sets $A \times I \simeq \varDelta^n \times I^m$.

The idea behind the construction of a vector $R_m^{-n}(g)v \in V^{1-m-n}_{tg}$ out of $g \in G_m$, $v \in V^{-n}_{sg}$ is relatively simple. Instead of considering $\varDelta^m \xto{g} G$, we consider its composition $I^m \xto{g\alpha} G$ with the simplicial map
\begin{equation}
	\alpha: I^m \to \varDelta^m, \quad%
		(r_1,\dotsc,r_m) \mapsto \max\{i: r_i = 1\}.
\label{eqn:alpha}
\end{equation}
The effect of this “cubical blow-up” of $\varDelta^m$ is illustrated in Figure \ref{fig:alpha} for $m = 1$,~$2$,~$3$.%
\begin{figure}
\begin{tikzpicture}[scale=0.5]
% m = 1:
 \draw[->,thick] (0,1)--(5,1);
 \node[left] at (0,1) {$0$}; \node[right] at (5,1) {$1$};
% m = 2:
 \path[fill=gray] (10,1)--(15,1)--(15,6)--(10,1);
 \draw[->,thick] (10,1)--(15,1); \draw[->,thick] (15,1)--(15,6);
 \draw[->,thick] (10,1)--(10,6); \draw (10,1)--(15,6); \draw[dotted] (10,6)--(15,6);
 \node[below left] at (10,1) {$0$}; \node[below right] at (15,1) {$1$};
 \node[above left] at (10,6) {$2$}; \node[above right] at (15,6) {$2$};
% m = 3:
 \path[fill=lightgray] (20,0)--(25,0)--(27,3)--(27,8)--(22,3)--(20,0);
 \path[fill=gray] (20,0)--(25,0)--(25,5)--(20,0);
 \path[fill=gray] (20,0)--(22,3)--(22,8)--(20,0);
 \path[fill=gray] (25,0)--(27,3)--(27,8)--(25,0);
 \draw[dashed] (20,0)--(27,3); \draw[dotted] (22,3)--(27,3); \draw (20,0)--(27,8) (22,3)--(27,8);
 \draw[->,thick] (20,0)--(25,0); \draw[->,thick] (25,0)--(27,3); \draw[->,thick] (27,3)--(27,8);
 \draw[->,thick] (20,0)--(22,3); \draw[->,thick] (22,3)--(22,8); \draw (20,0)--(22,8);
 \draw[dotted] (20,5)--(27,8); \draw[dotted] (22,8)--(27,8); \draw[dotted] (25,5)--(27,8);
 \draw[->,thick] (20,0)--(20,5); \draw (20,0)--(25,5); \draw[dotted] (20,5)--(22,8);
 \draw[->,thick] (25,0)--(25,5); \draw (25,0)--(27,8); \draw[dotted] (20,5)--(25,5);
 \node[below left] at (20,0) {$0$}; \node[below right] at (25,0) {$1$};
 \node[left] at (22,3) {$2$}; \node[right] at (27,3) {$2$};
 \node[above left] at (22,8) {$3$}; \node[above right] at (27,8) {$3$};
 \node[left] at (20,5) {$3$}; \node[right] at (25,5) {$3$};
\end{tikzpicture}
\caption{Cubical blow-up, $\alpha: I^m \to \varDelta^m$, $m = 1$,~$2$,~$3$}
\label{fig:alpha}
\end{figure}
We observe that the following identities of simplicial maps $I^{m-1} \to \varDelta^m$ hold for every $i = 1$,~$\dotsc$,~$m$.%
\begin{subequations}
\label{eqn:17A.6.1}
\begin{gather}
	\alpha\delta_{i|0} = \delta_i\alpha \\
	\alpha\delta_{i|1} = {\delta_0}^i{\upsilon_0}^{i-1}\alpha
\label{eqn:17A.6.1b}
\end{gather}
\end{subequations}

\begin{lem}[{\cite[beginning of §4.3]{2018a}}]\label{lem:P(g,v)} There is one and only one family of simplicial maps \[%
	P(g,v): \varDelta^n \times I^m \to V \text{,\quad $m$,~$n \geq 0$, $g \in G_m$, $v \in V^{-n}_{sg}$}
\] satisfying all of the following requirements.
\begin{enumerate}
\def\labelenumi{\upshape \arabic{enumi}.}
 \item Each one of them is a solution to the following lifting problem;
\begin{equation}
 \begin{split}
\xymatrix@C=1.67em@R=11ex{%
 \varDelta^n \ar[rrr]^-v \ar[d]_{(\id,0)} & & & V \ar[d]^p \\
 \varDelta^n \times I^m \ar@{-->}[rrru]^(.5){P(g,v)} \ar[r]^-\pr
 & I^m \ar[r]^-\alpha & \varDelta^m \ar[r]^-g & G
}\end{split}
\label{eqn:P(g,v)}
\end{equation}
in particular, $P(x,v) = (\varDelta^n \times I^0 \simeq \varDelta^n \xto{v} V)$ for all\/ $x \in G_0$, $v \in V^{-n}_x$.
 \item $P(g,v) = c^{n,m}\bigl(g\alpha \circ \pr,\left.P(g,v) \mathbin| \varPi^{n,m}\right.\bigr)$.
 \item The restriction of\/ $P(g,v)$ to\/ $\varPi^{n,m}$ is expressed, in conformity with the prescriptions below, in terms of simplicial maps of the form\/ $P(h,w): \varDelta^l \times I^k \to V$ with\/ $h \in G_k$, $w \in V^{-l}_{sh}$, and either\/ $k < m$ or else\/ $k = m$, $l < n$.%
\begin{subequations}
\label{eqn:S(g,v)}
\begin{align}
	P(g,v)&\circ (\delta_j \times \id) = P(g,d_jv) & &0 \leq j \leq n
\label{eqn:P(g,v).j}\\
	P(g,v)&\circ (\id \times \delta_{i|0}) = P(d_ig,v) & &1 \leq i \leq m
\label{eqn:P(g,v).i|0}\\
	P(g,v)&\circ (\id \times \delta_{i|1}) =%
		P\bigl(t_{m-i}g,\left.P(s_ig,v) \circ (\id \times \delta_{i|1})\right.\bigr) & &1 \leq i < m
\label{eqn:P(g,v).i|1}
\end{align}
\end{subequations}
\end{enumerate} \end{lem}

Concerning \eqref{eqn:P(g,v).i|1}, one can for any $k$-simplex $h \in G_k$ with $k < m$ and for any simplicial map $T: A \to V$ such that $A \xto{T} V \xto{p} G$ factors through $\varDelta^0 \xto{sh} G$ make sense of the expression $P(h,T)$ by demanding that $P(h,T): A \times I^k \to V$ be the only simplicial map such that for every $l \geq 0$ and for every $l$-simplex $\varDelta^l \xto{a} A$ the composition $\varDelta^l \times I^k \xto{a\times\id} A \times I^k \xto{P(h,T)} V$ be equal to $P(h,Ta)$; we refer the reader to \cite[p.~28]{2018a} for a deeper technical discussion. In the case at hand, we have $k = m - i$, $A = \varDelta^n \times I^{i-1}$ and so the right-hand side of \eqref{eqn:P(g,v).i|1} is, as we expect it to be, a map $\varDelta^n \times I^{m-1} \simeq \varDelta^n \times I^{i-1} \times I^{m-i} \to V$.

In virtue of \eqref{eqn:17A.6.1b} the “missing” face $P(g,v) \circ (\id \times \delta_{m|1})$ lies over $g\alpha\delta_{m|1} \circ \pr = (1_{m-1}tg)\alpha \circ \pr$ in particular its composition with $p$ factors through $\varDelta^0 \xto{tg} G$; see Figure \ref{fig:alpha}. This means that each one of its maximal simplices is an element of the vector space $V^{1-m-n}_{tg}$. We take the “alternating sum” of all these maximal simplices where each simplex occurs with the sign given by its natural geometric orientation: by definition, this gives $R_m^{-n}(g)v$, $m \geq 1$. The details are as follows.

The maximal simplices of $\varDelta^n \times I^m$ are in one-to-one association with the poset injections $[n + m] \to [n] \times [1]^m$. The latter, in turn, correspond to those permutations of $[n + m] \smallsetminus \{0\}$ whose inverses restricted to $\{1,\dotsc,n\}$ are increasing: to any such permutation $\pi$, assign the maximal simplex \[%
	\sigma_\pi: \varDelta^{m+n} \to \varDelta^n \times I^m
\] whose associated poset injection makes a move forward in the direction of the $i$-th coordinate of $[1]^m$ whenever $\pi(j) = n + i$. Let $\mathfrak S^{n,m}$ be the set of all such permutations, and let $\sgn(\pi) = \pm 1$ denote the parity of $\pi$. Then, by definition, for all $m \geq 1$, $g \in G_m$, and $v \in V^{-n}_{sg}$,
\begin{equation}
	R_m^{-n}(g)v = \sum_{\pi\in\mathfrak S^{n,m-1}} \sgn(\pi)P(g,v) \circ (\id \times \delta_{m|1})\sigma_\pi.
\label{eqn:R(g)v}
\end{equation}
It results easily from the considerations of sections \ref{sub:fibrations} and \ref{sub:cleavages} that the map $R_m^{-n}(g): V^{-n}_{sg} \to V^{1-m-n}_{tg}$ thus defined is linear and that all these maps fit together into a single \emph{smooth} vector bundle morphism $R_m^{-n}: s^*V^{-n} \to t^*V^{1-m-n}$ in $\VB(G_m)$; indeed, as one can show by induction on the recursive definition of the maps $P(g,v)$ by applying Lemma \ref{lem:17A.28.1} to a suitable presentation of $A = \varPi^{n,m}$ as a coequalizer of simplicial maps, the rule $(g,v) \mapsto P(g,v)$ provides a morphism in $\VB(G_m)$ from $s^*V^{-n}$ to the pullback of the smooth vector bundle $(\varDelta^n \times I^m)(V) \to (\varDelta^n \times I^m)(G)$ along the smooth map \(%
	G_m = \varDelta^m(G) \xto{(\alpha\circ\pr)(G)} (\varDelta^n \times I^m)(G)
\), and composition with $(\id \times \delta_{m|1})\sigma_\pi$ provides another morphism in $\VB(G_m)$ from the same pullback to $t^*V^{-n}$.

We claim that completing our definition by setting $R_0^{-n}(x) = d^{-n}_x: V^{-n}_x \to V^{1-n}_x$ for all $x$ in $G_0$ yields a representation up to homotopy of $G$ on the Moore complex $(V^\bullet,d)$. To see it, we are going to need a couple of auxiliary lemmas.

\begin{lem}[{\cite[p.~34]{2018a}}]\label{prop:17A.7.2} For every simplicial map\/ $T: \varDelta^n \times I^m \to V$ whose composition with\/ $V \xto{p} G$ factors through\/ $\varDelta^0$ viz.~equals\/ $\varDelta^n \times I^m \to \varDelta^0 \xto{x} G$ for some\/ $x \in G_0$,
\begin{multline}
 \sum_{k=0}^{n+m} (-1)^k\Bigl\{\sum_{\pi\in\mathfrak S^{n,m}} \sgn(\pi)T\sigma_\pi\Bigr\}\delta_k
	= \sum_{j=0}^n (-1)^j\sum_{\pi\in\mathfrak S^{n-1,m}} \sgn(\pi)T \circ (\delta_j \times \id)\sigma_\pi + {}\\
		(-1)^n\sum_{i=1}^m (-1)^i\sum_{\pi\in\mathfrak S^{n,m-1}} \sgn(\pi)\{T \circ (\id \times \delta_{i|0})\sigma_\pi
			- T \circ (\id \times \delta_{i|1})\sigma_\pi\}.
\label{eqn:17A.7.2}
\end{multline} \end{lem}

\begin{lem}[{\cite[p.~35]{2018a}}]\label{lem:17A.10.3} For every simplicial map\/ $T: \varDelta^n \times I^k \to V$ whose composition with\/ $V \xto{p} G$ equals\/ $\varDelta^n \times I^k \to \varDelta^0 \xto{sh} G$ for a given\/ $h \in G_{m-k}$ (\/$m > k$),
\begin{multline}
	\sum_{\pi\in\mathfrak S^{n,m-1}} \sgn(\pi)P(h,T) \circ (\id \times \delta_{m|1})\sigma_\pi
		= \sum_{\pi\in\mathfrak S^{n+k,m-k-1}} \sgn(\pi) \cdot {}\\
			P\bigl(h,\left.\textstyle\sum_{\pi'\in\mathfrak S^{n,k}} \sgn(\pi')T\sigma_{\pi'}\right.\bigr) \circ (\id \times \delta_{m-k|1})\sigma_\pi.
\label{eqn:17A.10.1}
\end{multline} \end{lem}

With the aid of these lemmas we can now prove that our tensors $R_m: s^*V^{-n} \to t^*V^{1-m-n}$ satisfy the defining equations \eqref{eqn:RUTH} for a representation up to homotopy:
\begin{alignat*}{2}
 (-1)^mR_0(tg)R_m(g)v &
	= (-1)^n\sum_{k=0}^{n+m-1} (-1)^k\Bigl\{\sum_{\pi\in\mathfrak S^{n,m-1}} \sgn(\pi)P(g,v)(\id \times \delta_{m|1})\sigma_\pi\Bigr\}\delta_k
\\	&\hskip-7em%
	= (-1)^n\sum_{j=0}^n (-1)^j\sum_{\pi\in\mathfrak S^{n-1,m-1}} \sgn(\pi)P(g,v)(\delta_j \times \delta_{m|1})\sigma_\pi \\* &\hskip-7em\justify%
		+ \sum_{i=1}^{m-1} (-1)^i\sum_{\pi\in\mathfrak S^{n,m-2}} \sgn(\pi)\{P(g,v)(\id \times \underset{\makebox[.em][c]{$= \delta_{i|0}\delta_{m-1|1}$}}{\underline{\delta_{m|1}\delta_{i|0}}})\sigma_\pi - P(g,v)(\id \times \underset{\makebox[.em][c]{$= \delta_{i|1}\delta_{m-1|1}$}}{\underline{\delta_{m|1}\delta_{i|1}}})\sigma_\pi\}
\\	&\hskip-7em%
	= (-1)^n\sum_{j=0}^n (-1)^j\sum_{\pi\in\mathfrak S^{n-1,m-1}} \sgn(\pi)P(g,d_jv)(\id \times \delta_{m|1})\sigma_\pi \\* &\hskip-7em\justify%
		+ \sum_{i=1}^{m-1} (-1)^i\sum_{\pi\in\mathfrak S^{n,m-2}} \sgn(\pi)\bigl\{P(d_ig,v)(\id \times \delta_{m-1|1})\sigma_\pi - {}\\[-.3\baselineskip] & & &\llap{$%
			P\bigl(t_{m-i}g,\left.P(s_ig,v)(\id \times \delta_{i|1})\right.\bigr) \circ (\id \times \delta_{m-1|1})\sigma_\pi\bigr\}$}
\\	&\hskip-7em%
	= (-1)^n\sum_{j=0}^n (-1)^jR_m(g)d_jv
		+ \sum_{i=1}^{m-1} (-1)^iR_{m-1}(d_ig)v \\* &\hskip-7em\justify%
		- \sum_{i=1}^{m-1} (-1)^i\sum_{\pi\in\mathfrak S^{n+i-1,m-i-1}} \sgn(\pi) \cdot {}\\ & & &\llap{$\textstyle%
			P\bigl(t_{m-i}g,\left.\sum_{\pi'\in\mathfrak S^{n,i-1}} \sgn(\pi')P(s_ig,v)(\id \times \delta_{i|1})\sigma_{\pi'}\right.\bigr) \circ (\id \times \delta_{m-i|1})\sigma_\pi$}
\\	&\hskip-7em%
	= \sum_{i=1}^{m-1} (-1)^iR_{m-1}(d_ig)v
		- R_m(g)R_0(sg)v
		- \sum_{i=1}^{m-1} (-1)^iR_{m-i}(t_{m-i}g)R_i(s_ig)v.
\end{alignat*}

The representation up to homotopy $R$ of $G$ on $(V^\bullet,d)$ that we have just constructed, the \emph{Moore representation} associated with the cleavage $c = \{c_{n,k}\}_{n>k\geq 0}$ of $V \xto{p} G$, has a couple of drawbacks. First, it operates on the largely redundant Moore complex $(V^\bullet,d)$, whereas it would be preferable to have one operating on the much smaller Dold–Kan complex $(\spl{V},\spl{d})$. Second, even though as we shall see $R$ is always \emph{weakly unital} for a normal $c$ in the sense that $R_1(1x) = \id$ for all $x \in G_0$, it may nevertheless not be unital.

Both drawbacks are eliminated simply by applying the normalization map \eqref{eqn:nor} to $R$: as we shall see presently, the outcome is a \emph{unital} representation up to homotopy of $G$ on $(\spl{V},\spl{d})$. We shall call this the \emph{Dold–Kan representation} obtained by splitting $V \xto{p} G$ by means of the normal cleavage $c$, and denote it by $\spl{R}$.

Asking the normalization map \eqref{eqn:nor}, a cochain homotopy equivalence $(V^\bullet,d) \to (\spl{V},\spl{d})$, to be a \emph{strict} intertwiner of representations up to homotopy $\nor: (V^\bullet,R) \to (\spl{V},\spl{R})$ is tantamount to requiring the equality below to hold for all $g \in G_m$, $v \in V^{-n}_{sg}$.
\begin{equation}
	\nor^{1-m-n}_{tg}\bigl(R_m^{-n}(g)v\bigr) = \spl{R}_m^{-n}(g)(\nor^{-n}_{sg}v)
\label{eqn:16A.7.3}
\end{equation}
We take this to be the definition of $\spl{R}_m^{-n}(g)v \in \spl{V}^{1-m-n}_{tg}$ for $v = \nor^{-n}_{sg}v \in \spl{V}^{-n}_{sg}$; we contend that the left-hand side of \eqref{eqn:16A.7.3} does in fact only depend on $\nor^{-n}_{sg}v$, not on $v \in V^{-n}_{sg}$; knowing this, it is immediate to check that the “normalized” tensors $\spl{R}_m$ do themselves constitute a representation up to homotopy $\spl{R} = \{\spl{R}_m\}$, i.e., satisfy the defining equations \eqref{eqn:RUTH} (see \cite[p.~32]{2018a}). Now, our contention is a fairly straightforward consequence of our next remark:

\begin{lem}\label{lem:16A.7.1} The identity below is valid for all\/ $g \in G_m$, $v \in V^{-n}_{sg}$ for each\/ $j = 0$,~$\dotsc$,~$n$.
\begin{equation}
	P(g,u_jv) = P(g,v) \circ (\upsilon_j \times \id): \varDelta^{n+1} \times I^m \longto V
\label{eqn:16A.7.1}
\end{equation} \end{lem}

\begin{proof} On applying Lemma \ref{lem:16A.24.1} with $\kappa = \upsilon_j \times \id: \varDelta^{n+1} \times I^{m-1} \to \varDelta^n \times I^{m-1}$ to the simplicial map $\varDelta^n \times I^{m-1} \times I \simeq \varDelta^n \times I^m \xto{P(g,v)} V$, $m \geq 1$, we get
\begin{equation}
	c^{n+1,m}\bigl(g\alpha \circ \pr,\left.P(g,v) \circ (\upsilon_j \times \id) \mathbin| \varPi^{n+1,m}\right.\bigr) = P(g,v) \circ (\upsilon_j \times \id).
\label{eqn:16A.24.3*}
\end{equation}
We then argue by induction on the recursive definition of the maps $P(g,v)$ with the aid of the latter equation; cf.~\cite[p.~36]{2018a}. \end{proof}

At this stage, we have shown that there exists a unique representation up to homotopy $\spl{R}$ of $G$ on $(\spl{V},\spl{d})$ for which $\nor: (V^\bullet,R) \to (\spl{V},\spl{R})$ is a strict intertwiner. There remains to be shown that $\spl{R}$ is also unital. In order to do so, we introduce a new auxiliary simplicial map, which we call the \emph{$i$-th cubical degeneracy map}, for each $i = 0$,~$\dotsc$,~$m$:
\begin{multline}
	\varepsilon_i = \varepsilon^m_i: I^{m+1} \longto I^m \text{,} \\*
		(r_0,\dotsc,r_m) \longmapsto%
\begin{cases}
	(r_1,\dotsc,r_m)                                            &\text{for $i = 0$} \\
	(r_0,\dotsc,r_{i-2},\max\{r_{i-1},r_i\},r_{i+1},\dotsc,r_m) &\text{for $i > 0$.}\quad%
\end{cases}
\label{eqn:epsilon}
\end{multline}
The effect of $\varepsilon^m_i$ is exemplified in Figure \ref{fig:epsilon} for $m = 2$.%
\begin{figure}
\begin{tikzpicture}[scale=0.5]
% i = 0:
 \path[fill=gray] (0,0)--(2,3)--(2,8)--(0,0); \draw[thick] (0,0)--(2,8);
 \path[fill=lightgray] (0,0)--(2,3)--(7,8)--(0,0); \draw (2,3)--(7,8);
 \path[fill=lightgray] (0,0)--(7,3)--(7,8)--(0,0); \draw (0,0)--(7,3) (0,0)--(7,8);
 \path[fill=gray] (5,0)--(7,3)--(7,8)--(5,0); \draw[dashed] (0,0)--(7,3); \draw[thick] (5,0)--(7,8);
 \draw[dotted,thick] (0,0)--(5,0); \draw[->,thick] (5,0)--(7,3); \draw[->,thick] (7,3)--(7,8);
 \draw[->,thick] (0,0)--(2,3); \draw[->,thick] (2,3)--(2,8); \draw[dotted] (2,8)--(7,8);
 \draw[->,thick] (0,0)--(0,5); \draw[thick] (0,5)--(2,8); \draw[dotted] (2,3)--(7,3);
 \draw[->,thick] (5,0)--(5,5); \draw[dotted] (0,5)--(5,5);
 \draw (0,0)--(5,5); \draw (0,5)--(7,8); \draw[thick] (5,5)--(7,8);
 \node[below left] at (0,0) {$00$}; \node[below right] at (5,0) {$00$};
 \node[left] at (2,3) {$10$}; \node[right] at (7,3) {$10$};
 \node[above left] at (2,8) {$11$}; \node[above right] at (7,8) {$11$};
 \node[left] at (0,5) {$01$}; \node[right] at (5,5) {$01$};
% i = 1:
 \path[fill=gray] (10,0)--(12,3)--(12,8)--(10,0); \draw[thick] (10,0)--(12,8);
 \path[fill=lightgray] (10,0)--(12,3)--(17,8)--(10,0); \draw (12,3)--(17,8);
 \path[fill=lightgray] (10,0)--(17,3)--(17,8)--(10,0); \draw (10,0)--(17,3);
 \path[fill=lightgray] (10,0)--(15,0)--(17,8)--(10,0); \draw (10,0)--(17,8);
 \path[fill=gray] (10,0)--(15,0)--(15,5)--(10,0);
 \draw[dotted] (12,3)--(17,3); \draw[dashed] (10,0)--(17,3);
 \draw[->,thick] (10,0)--(15,0); \draw[dotted,thick] (15,0)--(17,3); \draw[->,thick] (17,3)--(17,8);
 \draw[->,thick] (10,0)--(12,3); \draw[->,thick] (12,3)--(12,8); \draw[dotted] (12,8)--(17,8);
 \draw[->,thick] (10,0)--(10,5); \draw[thick] (10,5)--(12,8); \draw (10,5)--(17,8);
 \draw[->,thick] (15,0)--(15,5); \draw[thick] (10,5)--(15,5); \draw[dotted] (15,5)--(17,8);
 \draw[thick] (10,0)--(15,5); \draw (15,0)--(17,8);
 \node[below left] at (10,0) {$00$}; \node[below right] at (15,0) {$10$};
 \node[left] at (12,3) {$10$}; \node[right] at (17,3) {$10$};
 \node[above left] at (12,8) {$11$}; \node[above right] at (17,8) {$11$};
 \node[left] at (10,5) {$01$}; \node[right] at (15,5) {$11$};
% i = 2:
 \path[fill=gray] (20,0)--(25,0)--(27,3)--(20,0); \draw[thick] (20,0)--(27,3) (22,3)--(27,3);
 \path[fill=lightgray] (20,0)--(25,0)--(27,8)--(20,0); \draw (20,0)--(27,8);
 \path[fill=gray] (20,0)--(25,0)--(25,5)--(20,0); \draw[dashed,thick] (20,0)--(27,3) (22,3)--(27,3);
 \draw[->,thick] (20,0)--(25,0); \draw[->,thick] (25,0)--(27,3); \draw[dotted,thick] (27,3)--(27,8);
 \draw[->,thick] (20,0)--(22,3); \draw[->,thick] (20,0)--(20,5); \draw[->,thick] (25,0)--(25,5);
 \draw (20,0)--(22,8) (20,5)--(27,8) (22,3)--(27,8);
 \draw[thick] (20,5)--(25,5) (20,0)--(25,5) (22,8)--(27,8);
 \draw (25,0)--(27,8); \draw[dotted] (20,5)--(22,8) (22,3)--(22,8) (25,5)--(27,8);
 \node[below left] at (20,0) {$00$}; \node[below right] at (25,0) {$10$};
 \node[left] at (22,3) {$01$}; \node[right] at (27,3) {$11$};
 \node[above left] at (22,8) {$01$}; \node[above right] at (27,8) {$11$};
 \node[left] at (20,5) {$01$}; \node[right] at (25,5) {$11$};
\end{tikzpicture}
\caption{Cubical degeneracy maps $\varepsilon_0$,~$\varepsilon_1$,~$\varepsilon_2: I^3 \to I^2$}
\label{fig:epsilon}
\end{figure}
The following equality of simplicial maps $I^{m+1} \to \varDelta^m$ is easily verified and, along with our next lemma, justifies our definitions.
\begin{equation}
	\alpha\varepsilon_i = \upsilon_i\alpha
\label{eqn:16A.8.2b}
\end{equation}

\begin{lem}\label{lem:16A.10.1} The following equation holds for every\/ $i = 0$,~$\dotsc$,~$m$ for all\/ $g \in G_m$, $v \in V^{-n}_{sg}$.
\begin{equation}
	P(u_ig,v) = P(g,v) \circ (\id \times \varepsilon_i): \varDelta^n \times I^{m+1} \longto V
\label{eqn:16A.10.1}
\end{equation} \end{lem}

\begin{proof} For any simplicial map $w: \varDelta^n \times I^m \to V$ such that $m \geq 1 \seq w = c^{n,m}(pw,\left.w \mathbin| \varPi^{n,m}\right.)$, the identity below is valid in consequence of Lemmas \ref{lem:16A.24.1} and \ref{lem:16A.24.2}; cf.~\cite[Lem.~4.10]{2018a}.%
\begin{equation}
	c^{n,m+1}\bigl(pw \circ (\id \times \varepsilon_i),\left.w \circ (\id \times \varepsilon_i) \mathbin| \varPi^{n,m+1}\right.\bigr) = w \circ (\id \times \varepsilon_i)
\label{eqn:16A.9.1}
\end{equation}
When combined with our recursive definition of the maps $P(g,v)$, this identity implies \eqref{eqn:16A.10.1}; cf.~\cite[p.~37]{2018a}. \end{proof}

With the help of the preceding lemma we can easily check that ($R$ is weakly unital and that) $\spl{R}$ is unital; the details are worked out in \cite[p.~33]{2018a}.

\begin{exms*} {\itshape (a)}\spacefactor3000\/ Let $G$ be a Lie $0$-groupoid; this includes and is substantially the same as the case of a constant simplicial manifold $G = G_0$. Every simplicial vector bundle $V \to G$ is a vector fibration and admits a unique normal cleavage \cite[Ex.~2.4, p.~13]{2018a}. The associated Dold–Kan representation is the only possible unital representation up to homotopy of $G$ on the cochain complex $(\spl{V},\spl{d})$.

{\itshape (b)}\spacefactor3000\/ A Lie $1$-groupoid $G$ is essentially the same thing as (the nerve of) a \emph{Lie groupoid} $G = G_1 \tto G_0$. Every $1$-strict vector fibration $V \to G$ admits a unique cleavage, necessarily normal. We have $\spl{V} = \spl{V}^0$, and $\spl{R}_1: s^*\spl{V}^0 \simto t^*\spl{V}^0$ is an ordinary ($1$-term) representation of $G_1 \tto G_0$ on $\spl{V}^0$ satisfying $\spl{R}_1(g_2g_1) = \spl{R}_1(g_2)\spl{R}_1(g_1)$ for all composable pairs of arrows $x_2 \xfrom{g_2} x_1 \xfrom{g_1} x_0$.

{\itshape (c)}\spacefactor3000\/ In the literature, $2$-strict vector fibrations over Lie groupoids are known as \emph{VB-groupoids}. Given any such $V \to G$, the simplicial manifold $V$ itself may be regarded as a Lie groupoid $V = V_1 \tto V_0$. Fixing a normal cleavage of $V \to G$ is tantamount to fixing a normal $1,0$-cleavage $c_{1,0}$. The nontrivial Dold–Kan tensors may easily be expressed in terms of the groupoid structure of $V_1 \tto V_0$: $\spl{V} = \spl{V}^{-1} \oplus \spl{V}^0$ and, given $x_2 \xfrom{g_2} x_1 \xfrom{g_1} x_0$, $v \in \spl{V}^0_{x_0}$, and $w \in \spl{V}^{-1}_{x_0}$, writing $\an w = tw$ and $g_1v = tc_{1,0}(g_1;v)$,
\begin{gather*}
	\spl{R}_0^{-1}(x_0)w = \an w, \quad
	\spl{R}_1^0(g_1)v = g_1v, \quad
	\spl{R}_1^{-1}(g_1)w = c_{1,0}(g_1;\an w) \cdot w \cdot 0_{g_1^{-1}} \text{,\quad and} \\
	\spl{R}_2^0(g_2,g_1)v = \nor_{tg_2}\bigl(c_{1,0}(g_2;g_1v) \cdot c_{1,0}(g_1;v) \cdot c_{1,0}(g_2g_1;v)^{-1}\bigr).
\end{gather*}
Our general prescription thus boils down to the formulas of \cite{GSM17} in the example under consideration. \end{exms*}

The tangent bundle $V = TG \to G$ of a Lie ($1$-)groupoid $G = G_1 \tto G_0$ is of course a VB-groupoid; its (normal) $1,0$-cleavages go under the name of (unital) \emph{Ehresmann connections} on $G_1 \tto G_0$; $\spl{V}^{-1} = 1^*\ker Ts$ is the \emph{Lie algebroid bundle} of $G_1 \tto G_0$, $\spl{R}_0^{-1} = \an$ is its \emph{infinitesimal anchor map} (ranging over $\spl{V}^0 = TG_0$), and $\spl{R} = \Ad^c$ coincides with the \emph{adjoint representation} of $G_1 \tto G_0$ associated with $c$ as constructed originally in \cite[Def.~3.13]{AAC13}. The latter example motivates our next definition. Given an arbitrary Lie $\infty$-groupoid $G$, let us write $(\cplx{G},\an)$ for the Dold–Kan complex of its tangent bundle $TG \to G$ or, as it is more popularly known, its \emph{tangent complex} \cite{CZ23}.

\begin{defn}\label{defn:Ad} Let $c = \{c_{n,k}\}_{n>k\geq 0}$ be a normal cleavage of the tangent bundle $TG \to G$ of the Lie $\infty$-groupoid $G$. The \emph{adjoint representation} $\cplx{G} = (\cplx{G},\Ad^c)$ of $G$ on $(\cplx{G},\an)$ associated with $c$ is the Dold–Kan representation obtained by splitting $TG \to G$ by means of $c$. \end{defn}

\subsection{Cleavages with special properties}\label{sub:coherent}

Readers familiar with \cite{2022a} may be wondering whether our construction of the adjoint representation could be simplified along the lines of \cite[§7]{2022a} by splitting the tangent bundle of a higher Lie groupoid $G$ by means of a cleavage of the more restrictive kind considered in loc.~cit.\ rather than an arbitrary normal cleavage. In this section we provide evidence that doing so is actually impossible for a good number if not the majority of $G$.

\begin{defn}\label{defn:coherent} A cleavage $c = \{c_{n,k}\}_{n>k\geq 0}$ is \emph{coherent} if $C_{n,k} = \im(c_{n,k})$ is independent of $k$, in other words, for each $n$ there exists a subbundle $C_n \subset V_n$ such that $C_{n,k} = C_n$ for all $k$. \end{defn}

All cleavages intervening in the splitting construction of \cite[§7]{2022a} must be \emph{coherent}, by what “cleavage” means in loc.\ cit.; in fact, even stronger requirements have to be made on the cleavages in order for that construction to work but it can be shown that as long as there is some coherent normal cleavage at all one can always arrange for these extra requirements to be met (this nontrivial result, whose proof we omit, is yet another application of the interpolation construction to be discussed in the next section). We shall presently give examples of higher Lie groupoids whose tangent bundles fail in a dramatic way to admit coherent cleavages.

\begin{lem}\label{lem:21A.2.1} Let\/ $V \to G$ be a vector fibration. For all\/ $n \geq k \geq 0$ put
\begin{equation*}
\textstyle%
	E_n = \bigcap_{j=0}^n \ker(d_j: V_n \to V_{n-1}), \quad
		K_{n,k} = \ker\bigl(V_n \to G_n \times_{\varLambda^n_k(G)} \varLambda^n_k(V)\bigr).
\end{equation*}
Suppose that for some\/ $0 \leq k < l \leq n + 1$ there is some smooth vector subbundle\/ $C_{n+1} \subset V_{n+1}$ over\/ $G_{n+1}$ such that \[%
	V_{n+1} = C_{n+1} \oplus K_{n+1,k} = C_{n+1} \oplus K_{n+1,l}.
\] Then, the restriction of\/ $K_{n,k} \subset V_n$ along each fiber of the submersion\/ $d_k: G_n \to G_{n-1}$ admits a smooth vector bundle trivialization carrying\/ $E_n \subset K_{n,k}$ into itself. \end{lem}

\begin{proof} Let us fix any $g$ in $G_{n-1}$ and write $Z = d_k^{-1}(g)$ for the corresponding $d_k$~fiber~$\subset G_n$. Let $h: Z \to G_{n+1}$ be the smooth map given by $z \mapsto u_{k+1}z$ when $l = k + 1$ and by $z \mapsto u_{l-1}z$ when $l > k + 1$; in either case the composite map $d_k \circ h$ has constant value $u_jg$, where $j = \upsilon_k(l - 1)$, and $d_l \circ h$ equals the inclusion of $Z$ into $G_n$.

By assumption we have an isomorphism of smooth vector bundles over $G_{n+1}$ \[%
\xymatrix{%
 G_{n+1} \times_{\varLambda^{n+1}_l(G)} \varLambda^{n+1}_l(V)
 &	\ar[l]_-\backsim C_{n+1} \ar[r]^-\sim
	&	G_{n+1} \times_{\varLambda^{n+1}_k(G)} \varLambda^{n+1}_k(V)
}\] under which the images of the evident embeddings
\begin{gather*}
	{d_k}^*K_{n,l-1} \longto G_{n+1} \times_{\varLambda^{n+1}_l(G)} \varLambda^{n+1}_l(V) \\
	{d_l}^*K_{n,k} \longto G_{n+1} \times_{\varLambda^{n+1}_k(G)} \varLambda^{n+1}_k(V)
\end{gather*}
correspond bijectively to each other. This yields an isomorphism \(%
	{d_k}^*K_{n,l-1} \simto {d_l}^*K_{n,k}
\) of smooth vector bundles over $G_{n+1}$ which carries ${d_k}^*E_n$ onto ${d_l}^*E_n$. Pulling this back along $h: Z \to G_{n+1}$ in turn yields an isomorphism of smooth vector bundles over $Z$ \[%
	Z \times (K_{n,l-1})_{u_jg} \simeq (d_k \circ h)^*K_{n,l-1} \simto (d_l \circ h)^*K_{n,k} \simeq K_{n,k} \mathbin| Z.
\qedhere\] \end{proof}

Let $G$ be an arbitrary simplicial manifold; as usual for any poset map $\theta: [m] \to [n]$ we write $G_\theta: G_n \to G_m$ for the corresponding map on $n$-simplices. Let us fix an arbitrary sequence of smooth vector bundles $E_n \to G_n$, $n = 0$,~$1$,~$2$,~$\dotsc$, and define $V_n \to G_n$ to be the smooth vector bundle \[%
\textstyle
	V_n = \bigoplus\limits_{\alpha:[k]\into[n]} {G_\alpha}^*E_k
\] where the direct sum runs over all poset injections $\alpha: [k] \into [n]$ (for all $k$ between $0$ and $n$); for every poset map $\theta: [m] \to [n]$ let $V_\theta: V_n \to V_m$ be the smooth vector bundle morphism covering $G_\theta: G_n \to G_m$ given by \[%
	\pr_\alpha \circ V_\theta =%
\begin{cases}
		\pr_{\theta\alpha} &\text{if $\theta\alpha$ is injective} \\
		0                  &\text{otherwise,}
\end{cases}
\] where $\pr_\alpha: V_n \to E_{k=\dim\alpha}$ is the projection onto the $\alpha$-th direct summand. The correspondence $\theta \mapsto V_\theta$ is obviously a contravariant functor $\Cat{\Delta}^\op \to \VB$ and hence is a simplicial vector bundle $V \to G$. The latter must be a vector fibration whenever $G$ is a Lie $\infty$-groupoid because for all $n \geq k \geq 0$ the kernel of the fiberwise linear map \(%
	V_n \to G_n \times_{\varLambda^n_k(G)} \varLambda^n_k(V)
\) is the vector subbundle ${d_k}^*E_{n-1} \oplus \id^*E_n$ of $V_n$ whose obvious complement evidently splits that map. The notations of our last lemma are consistent with the present ones: $E_n \simeq \id^*E_n$ coincides with the intersection of the kernels of the maps $d_j: V_n \to V_{n-1}$, $j = 0$,~$\dotsc$,~$n$.

\begin{exm}\label{npar:21A.2.2} Let $G$ be any Lie $\infty$-groupoid such that the following two items exist for at least one $x$ in $G_0$, where we write $G^x$ for the target fiber $t^{-1}(x) = {d_0}^{-1}(x) \subset G_1$:
\begin{enumerate}
\def\labelenumi{(\alph{enumi})}
 \item a smooth map $h: G^x \to G_2$ such that $d_0h(g) = g$ and $d_1h(g) = 1x$;
 \item a smooth vector bundle $E \to G_1$ whose restriction over $G^x$ is stably nontrivializable.
\end{enumerate}
For instance when $G$ is a Lie groupoid the first item is furnished by $g \mapsto (g,g^{-1})$ (at every $x$) while the second item only fails to exist in exceptional cases \cite[§4]{MS74}.

We are going to exhibit an $m + 1$-strict vector fibration $V \to G$ such that for all $0 \leq k < l < n \leq m$ one cannot possibly have $C_{n,k} = C_{n,l}$ in other words no $n,k$-cleavage $C_{n,k}$ can function simultaneously as an $n,l$-cleavage $C_{n,l}$; a fortiori, $V_n \to G_n$ cannot, for any $2 \leq n \leq m$, admit subbundles $C_n$ of the sort described in Definition \ref{defn:coherent}. The vector fibration in question will be a special case of our previous construction.

Let us fix an arbitrary smooth vector bundle $E_0 \to G_0$ in $\VB(G_0)$. For every $n \geq 1$ let us consider the object $E_n \to G_n$ of $\VB(G_n)$ given by \[%
\textstyle
	E_n =%
\begin{cases}
	\bigoplus\limits_{j=0}^{n-1} {d_{[n-1]\smallsetminus\{j\}}}^*E &\text{for $n < m$} \\
	0                                                              &\text{for $n \geq m$,}
\end{cases}
\] where $d_{[n-1]\smallsetminus\{j\}}: G_n \to G_1$ is the smooth map associated with the poset injection $[1] \into [n]$ of image $\{j,n\}$. Of course the vector fibration $V \to G$ corresponding via our general construction to the sequence $E_n \to G_n$ will be $m + 1$-strict. For each $k = 0$,~$\dotsc$,~$n - 1$ with $0 < n < m$ let $h^n_k: G^x \to G_n$ be the map defined by the prescription \[%
	h^n_k(g) =%
\begin{cases}
	u_{n-1} \dotsm u_1g                   &\text{for $k = 0$} \\
	u_{n-1} \dotsm u_{k+1}{u_0}^{k-1}h(g) &\text{for $k > 0$.}
\end{cases}
\] It is easy to check that $d_kh^n_k(g) = 1_{n-1}x$ for all $k$ and so $h^n_k$ ranges within ${d_k}^{-1}(1_{n-1}x) \subset G_n$. The pullback of $E_n \to G_n$ along $h^n_k$ is stably isomorphic to $E \mathbin| G^x$ because the composition of $h^n_k$ with $d_{[n-1]\smallsetminus\{j\}}$ is $g \mapsto g$ when $j = k$ and is $g \mapsto 1tg = 1x$ (a constant map) otherwise. This means that $E_n \to G_n$ cannot be trivializable over ${d_k}^{-1}(1_{n-1}x)$ for $0 \leq k < n$. Now Lemma \ref{lem:21A.2.1} implies the impossibility of the identity $C_{n+1,k} = C_{n+1,l}$ for all $0 \leq k < l \leq n$. \end{exm}

\begin{exm}\label{npar:21A.2.3} Our construction of a vector fibration $V \to G$ out of the sequence of vector bundles $E_n \to G_n$ has an obvious Lie $\infty$-groupoid theoretic counterpart. Let us fix an arbitrary sequence of pointed smooth manifolds $X_0$,~$X_1$,~$X_2$,~$\dots$ and form the cartesian products
\begin{equation*}
\textstyle
	G_n = \prod\limits_{\alpha:[k]\into[n]} X_k \text{,\quad $n \geq 0$.}
\end{equation*}
For each poset map $\theta: [m] \to [n]$ let us declare $G_\theta: G_n \to G_m$ to have components
\begin{equation*}
	\pr_\alpha \circ G_\theta =%
\begin{cases}
		\pr_{\theta\alpha}  &\text{if $\theta\alpha$ is injective} \\
		\ast_{k=\dim\alpha} &\text{otherwise,}
\end{cases}
\end{equation*}
where $\ast_k$ denotes the (constant map of value the) base point of $X_k$. The correspondence $\theta \mapsto G_\theta$ is obviously a simplicial manifold $G$. Clearly $G$ will be a Lie $m$-groupoid whenever $X_n = \ast$ for all $n \geq m$.

Fix $x \in G_0$. For any $n \geq 1$ the intersection $Z_n = \bigcap_{k=0}^n {d_k}^{-1}(1_{n-1}x)$ is a smooth submanifold of $G_n$ diffeomorphic to $X_n$. In the notations of Lemma \ref{lem:21A.2.1} as applied to the tangent bundle $TG \to G$, we have \[%
\textstyle
	E_n \mathbin| Z_n
		= \bigcap_{k=0}^n \ker(Td_k: TG_n \to TG_{n-1}) \mathbin| Z_n
		= TZ_n \subset TG_n \mathbin| Z_n
\] and so $E_n \to G_n$ cannot be trivializable over ${d_k}^{-1}(1_{n-1}x)$ for $0 \leq k \leq n$ unless $TZ_n$ itself is trivializable in other words $X_n$ is parallelizable; whenever that fails to be the case, Lemma \ref{lem:21A.2.1} makes the identity $C_{n+1,k} = C_{n+1,l}$ for $0 \leq k < l \leq n$ an impossibility. \end{exm}

In conclusion\spacefactor3000: \em The study of adjoint representations of higher Lie groupoids\/ \emph{demands} that we consider\/ \emph{arbitrary} (normal) cleavages---not just coherent ones.\em

\section{Splitting a morphism of higher vector bundles by a choice of cleavages}\label{sec:morphism}

Higher vector bundles are the objects of a full subcategory $\VB^\infty \subset [\Cat{\Delta}^\op,\VB]$ of the category of all simplicial vector bundles. For a higher Lie groupoid $G$ we write $\VB^\infty(G)$ for the subcategory of $\VB^\infty$ formed by all higher vector bundles over $G$ and their morphisms covering the identity transformation of $G$:
\begin{equation*}
\xymatrix{%
 V \ar[r]^\phi
 \ar[d]^p
 &	W
	\ar[d]^q
\\ G \ar[r]^=
 &	G\rlap{.}}
\end{equation*}
Suppose now that each of $V \xto{p} G$, $W \xto{q} G$ is equipped with a normal cleavage. Our main goal in this section is to explain how the cleavages may be used to extract, out of any given morphism $\phi: V \to W \in \VB^\infty(G)$, an intertwiner $\spl{\phi}: \spl{V} \to \spl{W} \in \Rep^\infty_+(G)$ of the corresponding Dold–Kan representations $\spl{V} = (\spl{V},\spl{R})$, $\spl{W} = (\spl{W},\spl{S})$ in such a way that $\spl{\phi}_0^{-n}: \spl{V}^{-n} \to \spl{W}^{-n}$ coincides with the morphism in $\VB(G_0)$ induced by $\phi_n: V_n \to W_n$ upon restriction to the subbundle $\spl{V}^{-n}$ of $V^{-n} = V_n \mathbin| 1_n(G_0)$. Once our task is carried out we shall automatically have proven what follows on account of Proposition \ref{prop:isomorphism} as applied to $\varPhi = \spl{\phi}$, $\phi = \id: V \to V$.

\begin{thm}\label{thm:2018b} The Dold–Kan representations arising from any two different choices of a normal cleavage of a vector fibration\/ $V \xto{p} G$ are isomorphic objects of the category\/ $\Rep^\infty_+(G)$. \end{thm}

\begin{cor}\label{cor:2018b} The isomorphism class\/ $[\Ad^c]$ of the adjoint representation of a Lie\/ $\infty$-groupoid\/ $G$ is a well-defined invariant of\/ $G$ independent of the choice of a normal cleavage, $c$, of the tangent bundle\/ $TG \to G$. \end{cor}

\subsection{The interpolation construction for morphisms}\label{sub:overview}

Our construction of $\spl{\phi}$ out of $\phi$ will parallel the construction of $\spl{V}$ out of $V$ reviewed in the previous section. To begin with, we introduce a simplicial complex which, in the present context, is going to play a role equivalent to that played by the fundamental $m$-cube $I^m$ in section \ref{sub:splitting}.

\begin{defn}\label{defn:interpolation} For each $a = 1$,~$\dotsc$,~$m$ let $I^m_a$ be the simplicial subset of $\nerve([1]^m \times [m])$~(= the nerve of the poset $[1]^m \times [m]$) whose $k$-simplices are all those maps $[k] \to [1]^m \times [m]$ that range within the subset \[%
	\bigl\{(r_1,\dotsc,r_m;r_+) \in [1]^m \times [m]: \textstyle\sum_{i=1}^{a-1} r_i \leq r_+ \leq \sum_{i=1}^a r_i\bigr\}
\] and that are order preserving with respect to the following partial ordering $\leq_a$ of $[1]^m \times [m]$:
\begin{multline}
 (r_1,\dotsc,r_m;r_+) \leq_a (s_1,\dotsc,s_m;s_+) \quad \text{iff} \\*
	r_i \leq s_i \quad \text{$\forall i = 1$,~$\dotsc$,~$m$}
\quad \text{and} \quad
	r_+ \leq s_+ - \textstyle\sum_{i=1}^{a-1}(s_i - r_i).
\end{multline}
The \emph{fundamental} (\emph{$m$-cubical}) \emph{interpolation} is the simplicial subset $I^m_+$ of $\nerve([1]^m \times [m])$ given by $I^m_+ = \bigcup_{a=1}^m I^m_a$ for $m \geq 1$ and by $I^0_+ = \nerve([1]^0 \times [0]) \simeq \varDelta^0$ for $m = 0$. \end{defn}

\begin{rmks}\label{rmks:unexpected} Performing the splitting construction on morphisms is the purpose our simplicial complex $I^m_+$ was originally devised for. Quite unexpectedly the same simplicial complex turned out to have other important applications\spacefactor3000: {\itshape (a)}\spacefactor3000\nobreak\/ It provides a “geometric” way to derive the formulas underlying the \emph{semidirect product} correspondence of \cite{2018a,2022a}; we shall touch on this indirectly in section \ref{sub:final}. {\itshape (b)}\spacefactor3000\nobreak\/ It leads to a profound characterization of the essential image of the semidirect product correspondence in terms of the existence of \emph{coherent} cleavages (Definition \ref{defn:coherent}), see \cite[Cor.~6.5]{2018a}. {\itshape (c)}\spacefactor3000\nobreak\/ It can be employed to process an arbitrary coherent normal cleavage into one that is in addition “weakly flat” in the sense of \cite[§2]{2022a}, cf.~our comments right after Definition \ref{defn:coherent}; as a matter of fact, the latter discovery helped figure out the very notion of weak flatness. \end{rmks}

Let $\alpha_+: I^m_+ \to \varDelta^m$ be the simplicial map $(r_1,\dotsc,r_m;r_+) \mapsto \alpha(r_1,\dotsc,r_m)$; Figure \ref{fig:alpha+} shows what $I^m_+ \xto{g\alpha_+} G$, the “interpolation blow-up” of $g \in G_m$, looks like for $m = 1$,~$2$ (hopefully, the pictures will shed some light on our definitions, as well as on our choice of terminology).%
\begin{figure}
\begin{tikzpicture}[scale=0.5]
% m = 1:
 \draw[->,thick] (0,1)--(5,1); \draw[->,thick] (0,1)--(5,6);
 \draw[dotted] (5,1)--(5,6);
 \node[left] at (0,1) {$0$}; \node[right] at (5,1) {$1$};
 \node[above right] at (5,6) {$1$};
% m = 2:
 \path[fill=gray] (10,0)--(15,0)--(17,3)--(10,0); \draw[thick] (10,0)--(17,3);
 \path[fill=lightgray] (10,0)--(15,0)--(17,8)--(10,0); \draw[dashed,thick] (10,0)--(17,3);
 \path[fill=lightgray] (10,0)--(15,5)--(17,8)--(10,0); \draw[thick] (10,0)--(17,8);
 \path[fill=gray] (10,0)--(15,5)--(17,13)--(10,0); \draw[dashed,thick] (10,0)--(17,8);
 \draw[->,thick] (10,0)--(15,0); \draw[->,thick] (15,0)--(17,3); \draw[dotted,thick] (12,3)--(17,3);
 \draw[->,thick] (10,0)--(15,5); \draw[->,thick] (15,5)--(17,8); \draw[dotted,thick] (12,3)--(17,8);
 \draw[->,thick] (15,5)--(17,13); \draw[->,thick] (10,0)--(12,8); \draw[dotted,thick] (12,8)--(17,13);
 \draw[->,dashed,thick] (10,0)--(12,3); \draw[dotted] (12,3)--(12,8);
 \draw (15,0)--(17,8); \draw[dotted] (15,0)--(15,5) (17,3)--(17,8);
 \draw[thick] (10,0)--(17,13); \draw[dotted] (17,8)--(17,13) (12,3)--(17,13);
 \node[below left] at (10,0) {$0$}; \node[below right] at (15,0) {$1$};
 \node[left] at (12,3) {$2$}; \node[right] at (17,3) {$2$};
 \node[above left] at (12,8) {$2$}; \node[right] at (17,8) {$2$};
 \node[above right] at (17,13) {$2$}; \node[right] at (15,5) {$1$};
\end{tikzpicture}
\caption{Interpolation blow-up, $\alpha_+: I^m_+ \to \varDelta^m$, $m = 1$,~$2$}
\label{fig:alpha+}
\end{figure}
In full analogy with the theory of section \ref{sec:splitting}, our construction of $\spl{\phi}$ is going to involve a family of simplicial maps%
\begin{subequations}
\begin{equation}
	P_+(g,v): \varDelta^n \times I^m_+ \longto W \text{,\quad $m$,~$n \geq 0$, $g \in G_m$, $v \in V^{-n}_{sg}$}
\label{eqn:P+}
\end{equation}
which individually are going to be solutions to the lifting problem%
\begin{equation}
 \begin{split}
\xymatrix@C=1.67em@R=11ex{%
 \varDelta^n \ar[rrr]^-{\phi v} \ar[d]_{(\id,0)} & & & W \ar[d]^q \\
 \varDelta^n \times I^m_+ \ar@{-->}[rrru]^(.5){P_+(g,v)} \ar[r]^-\pr
 & I^m_+ \ar[r]^-{\alpha_+} & \varDelta^m \ar[r]^-g & G\rlap{.}
}\end{split}
\label{eqn:P+(g,v)}
\end{equation}
\end{subequations}
The maps in this family will be constructed canonically via a recursive lifting process in terms of the given cleavages of $V \xto{p} G$, $W \xto{q} G$. Before we can state the appropriate analog of Lemma \ref{lem:P(g,v)}, we need to analyze the boundary $\partial I^m_+$ of $I^m_+$. We shall do this by decomposing $\partial I^m_+$ into elementary building blocks as follows.

\begin{defn}\label{defn:faces+} For all $a$,~$i \in \{1,\dotsc,m\}$, $a \neq i$ and $r = 0$,~$1$ let $\delta_{a,i|r} = \delta^m_{a,i|r}: I^{m-1}_{\upsilon_i(a)} \to I^m_a$ be the simplicial map%
\begin{equation}
	(r_1,\dotsc,r_{m-1};r_+) \longmapsto%
		\bigl(r_1,\dotsc,r_{i-1},r,r_i,\dotsc,r_{m-1};\left.r_+ + [a - \upsilon_i(a)]r\right.\bigr).
\label{npar:17A.14.2}
\end{equation}
For all $a = 1$,~$\dotsc$,~$m$ let $\delta_{a,a|1} = \delta^m_{a,a|1}$, $\delta_{a,\top} = \delta^m_{a,\top}$, and $\delta_{a,\bot} = \delta^m_{a,\bot}$ be the simplicial maps of $I^m$ into $I^m_a$ given by%
\begin{alignat}{2}
	(r_1,\dotsc,r_m)&\longmapsto \bigl(r_1,\dotsc,r_{a-1},1,r_{a+1},\dotsc,r_m;\left.\textstyle\sum_{i=1}^a r_i\right.\bigr),
\label{npar:17A.14.3}\\%
	(r_1,\dotsc,r_m)&\longmapsto \bigl(r_1,\dotsc,r_m;\left.\textstyle\sum_{i=1}^a r_i\right.\bigr) \text{,\quad and}
\label{npar:17A.14.1t}\\%
	(r_1,\dotsc,r_m)&\longmapsto \bigl(r_1,\dotsc,r_m;\left.\textstyle\sum_{i=1}^{a-1} r_i\right.\bigr) \text{,\quad respectively.}
\label{npar:17A.14.1b}
\end{alignat}
The same notations will be used for the maps arising upon composition with the inclusion $I^m_a \subset I^m_+$. \end{defn}

The above will be referred to as the elementary \emph{interpolation face inclusions}. A few auxiliary maps may now be defined in terms of these. First, our formula for $r = 0$ does not depend on $a$, so that the maps $\delta_{a,i|0}$ for a given $i$ agree with one another on the overlaps of their domains of definition and thus may be packed together into a single map we call $\delta_{+,i|0}: I^{m-1}_+ \to I^m_+$. Second, the maps $\delta_{a,1|1}$ with $a > 1$ may be assembled for similar reasons into a single map $\delta_{+,1|1}: I^{m-1}_+ \to I^m_+$, while the several $\delta_{a,m|1}$ with $a < m$ combine to give $\delta_{+,m|1}: I^{m-1}_+ \to I^m_+$. Finally, the abbreviations $\delta_\top = \delta_{m,\top}$, $\delta_\bot = \delta_{1,\bot}: I^m \to I^m_+$ will be in force.

Of course the image $\im(\delta_{+,i|0})$ of the map $\delta_{+,i|0}$ is the subcomplex $\partial_{i|0}I^m_+$ of $I^m_+$ formed by all those simplices whose vertices $(r_1,\dotsc,r_m;r_+)$ satisfy the condition $r_i = 0$. If we define $\partial_{i|1}I^m_+$ similarly, we have $\partial_{i|1}I^m_+ = \bigcup_{a=1}^m \partial_{i|1}I^m_a$, where $\partial_{i|1}I^m_a = \im(\delta_{a,i|1})$. (Have a look at Figure \ref{fig:i|1+}.)%
\begin{figure}
\begin{tikzpicture}[scale=0.5]
% i = 1:
 \path[fill=gray] (0,0)--(5,0)--(7,3)--(0,0); \draw[thick] (0,0)--(7,3);
 \path[fill=lightgray] (0,0)--(5,0)--(7,8)--(0,0); \draw[dashed,thick] (0,0)--(7,3);
 \draw (0,0)--(7,8) (5,0)--(7,8); \draw (0,0)--(2,8) (0,0)--(5,5);
 \path[fill=lightgray] (0,5)--(5,5)--(7,8)--(0,5); \draw[thick] (0,5)--(7,8);
 \path[fill=lightgray] (0,5)--(5,5)--(7,13)--(0,5); \draw[dashed] (0,0)--(7,8);
 \draw[dashed,thick] (0,5)--(7,8);
 \path[fill=lightgray] (0,5)--(5,10)--(7,13)--(0,5); \draw (5,5)--(7,13); \draw[thick] (0,5)--(7,13);
 \path[fill=gray] (0,5)--(5,10)--(7,18)--(0,5); \draw[dashed,thick] (0,5)--(7,13); \draw[thick] (0,5)--(7,18);
 \draw[->,thick] (0,0)--(5,0); \draw[->,thick] (5,0)--(7,3); \draw[->,thick] (0,0)--(2,3);
 \draw[->,thick] (0,5)--(5,5); \draw[->,thick] (5,5)--(7,8); \draw[->,dashed,thick] (0,5)--(2,8);
 \draw[->,thick] (0,5)--(5,10); \draw[->,thick] (5,10)--(7,13); \draw[dotted] (0,0)--(0,5);
 \draw[->,thick] (0,5)--(2,13); \draw[->,thick] (5,10)--(7,18); \draw[dashed] (0,0)--(2,8);
 \draw[dotted] (2,3)--(2,8) (2,3)--(7,8) (5,0)--(5,5) (5,5)--(5,10) (7,3)--(7,8);
 \draw[dotted,thick] (2,3)--(7,3) (2,8)--(7,8) (2,8)--(7,13) (2,13)--(7,18);
 \draw[dotted] (2,8)--(2,13) (2,8)--(7,18) (7,8)--(7,13) (7,13)--(7,18);
 \node[below left] at (0,0) {$1$}; \node[below right] at (5,0) {$2$};
 \node[left] at (2,3) {$3$}; \node[right] at (7,3) {$3$};
 \node[left] at (2,8) {$3$}; \node[right] at (7,8) {$3$};
 \node[above left] at (2,13) {$3$}; \node[right] at (7,13) {$3$};
 \node[left] at (0,5) {$1$}; \node[right] at (5,5) {$2$};
 \node[above right] at (7,18) {$3$}; \node[right] at (5,10) {$2$};
% i = 2:
 \draw[dotted] (17,3)--(17,8); \draw[dotted,thick] (12,3)--(17,8) (12,3)--(17,3);
 \draw[->,thick] (10,0)--(12,3); \draw[->,thick] (15,0)--(17,3); \draw[->,thick] (15,5)--(17,8);
 \draw (15,0)--(17,8); \draw[thick] (10,0)--(17,3) (10,0)--(17,8);
 \draw[dotted] (15,0)--(15,5); \draw[dotted,thick] (10,0)--(15,0) (10,0)--(15,5);
 \draw[dotted] (12,3)--(12,8) (12,3)--(17,13) (17,8)--(17,13);
 \draw (10,0)--(12,8) (10,0)--(17,13) (15,5)--(17,13);
 \draw[dotted] (10,0)--(10,5) (10,0)--(15,10) (15,5)--(15,10);
 \draw[dotted] (12,8)--(12,13) (12,8)--(17,18) (17,13)--(17,18);
 \draw[->,thick] (10,5)--(12,8); \draw[->,thick] (15,10)--(17,13);
 \draw[thick] (10,5)--(17,13); \draw[dotted,thick] (10,5)--(15,10) (12,8)--(17,13);
 \draw[->,thick] (10,5)--(12,13); \draw[->,thick] (15,10)--(17,18);
 \draw[thick] (10,5)--(17,18); \draw[dotted,thick] (12,13)--(17,18);
 \node[below left] at (10,0) {$2$}; \node[below right] at (15,0) {$2$};
 \node[left] at (12,3) {$3$}; \node[right] at (17,3) {$3$};
 \node[left] at (12,8) {$3$}; \node[right] at (17,8) {$3$};
 \node[above left] at (12,13) {$3$}; \node[right] at (17,13) {$3$};
 \node[left] at (10,5) {$2$}; \node[right] at (15,5) {$2$};
 \node[above right] at (17,18) {$3$}; \node[right] at (15,10) {$2$};
% i = 3:
 \draw[dotted,thick] (20,0)--(25,0) (20,0)--(27,3) (22,3)--(27,3) (25,0)--(27,3);
 \draw[dotted] (25,0)--(25,5) (25,0)--(27,8) (27,3)--(27,8);
 \draw[dotted,thick] (20,0)--(22,3) (20,0)--(27,8) (22,3)--(27,8) (25,5)--(27,8);
 \draw[dotted] (22,3)--(22,8) (22,3)--(27,13) (27,8)--(27,13);
 \draw[dotted,thick] (20,0)--(22,8) (20,0)--(25,5) (20,0)--(27,13) (22,8)--(27,13) (25,5)--(27,13);
 \draw[dotted] (20,0)--(20,5) (20,0)--(22,13) (20,0)--(27,18) (22,8)--(22,13) (22,8)--(27,18);
 \draw[dotted] (20,0)--(25,10) (25,5)--(25,10) (25,5)--(27,18) (27,13)--(27,18);
 \draw[dotted,thick] (20,5)--(22,13) (20,5)--(25,10) (20,5)--(27,18) (22,13)--(27,18) (25,10)--(27,18);
 \draw[fill] (20,0) circle [radius=.05] (20,5) circle [radius=.05] (22,3) circle [radius=.05];
 \draw[fill] (22,8) circle [radius=.05] (22,13) circle [radius=.05] (25,0) circle [radius=.05];
 \draw[fill] (25,5) circle [radius=.05] (25,10) circle [radius=.05] (27,3) circle [radius=.05];
 \draw[fill] (27,8) circle [radius=.05] (27,13) circle [radius=.05] (27,18) circle [radius=.05];
 \node[below left] at (20,0) {$3$}; \node[below right] at (25,0) {$3$};
 \node[left] at (22,3) {$3$}; \node[right] at (27,3) {$3$};
 \node[left] at (22,8) {$3$}; \node[right] at (27,8) {$3$};
 \node[above left] at (22,13) {$3$}; \node[right] at (27,13) {$3$};
 \node[left] at (20,5) {$3$}; \node[right] at (25,5) {$3$};
 \node[above right] at (27,18) {$3$}; \node[right] at (25,10) {$3$};
\end{tikzpicture}
\caption{Restriction of $\alpha_+: I^3_+ \to \varDelta^3$ to $\partial_{i|1}I^3_+$, $i = 1$,~$2$,~$3$}
\label{fig:i|1+}
\end{figure}
The remaining components of the boundary of $I^m_+$ are $\partial_\top I^m_+ = \im(\delta_\top)$ and $\partial_\bot I^m_+ = \im(\delta_\bot)$. We shall be interested in the union of all these boundary components except $\partial_{m|1}I^m_+$:%
\begin{subequations}
\label{npar:17A.15.2}
\begin{gather}
\textstyle%
	\varPi^m_+ = \bigl(\bigcup\limits_{1\leq i\leq m} \partial_{i|0} I^m_+\bigr)
		\cup \bigl(\bigcup\limits_{1\leq i<m} \partial_{i|1} I^m_+\bigr)
		\cup \partial_\top I^m_+ \cup \partial_\bot I^m_+
			\text{,\quad $m \geq 0$;} \\
	\varPi^{n,m}_+ = \partial\varDelta^n \times I^m_+ \cup \varDelta^n \times \varPi^m_+
		\text{,\quad $m$,~$n \geq 0$.}
\end{gather}
\end{subequations}
The inclusion $\varPi^{n,m}_+ \xto\subset \varDelta^n \times I^m_+$ is easily seen to be a regular anodyne extension. However, unlike $\varPi^{n,m} \xto\subset \varDelta^n \times I^m$, it is not one of the special type \eqref{eqn:sqcup} considered in section \ref{sub:cleavages}. It is nevertheless possible to construct a \emph{canonical} $\varPi^{n,m}_+ \xto\subset \varDelta^n \times I^m_+$-cleavage $c^{n,m}_+$ with good normality properties out of the given cleavage $c = \{c_{n,k}\}_{n>k\geq 0}$ of $W \xto{q} G$. In order not to interrupt the flow of our exposition, we shall postpone the construction of $c^{n,m}_+$ to section \ref{sub:canonical}. The proof of our next lemma will be postponed to section \ref{sub:auxiliary} for similar reasons. Let us write $\chi_a = \chi^{k,m}_a: I^m_a \simto I^k_a \times I^{m-k}$ ($1 \leq a \leq k \leq m$), resp., $\chi_a = \chi^{k,m}_a: I^m_a \simto I^k \times I^{m-k}_{a-k}$ ($0 \leq k < a \leq m$) for the isomorphisms of simplicial sets
\begin{alignat*}{2}
	(r_1,\dotsc,r_k,r_{k+1},\dotsc,r_m;r_+)&\longmapsto \bigl((r_1,\dotsc,r_k;r_+),(r_{k+1},\dotsc,r_m)\bigr) \text{,\quad resp.,} \\
	(r_1,\dotsc,r_k,r_{k+1},\dotsc,r_m;r_+)&\longmapsto
		\bigl((r_1,\dotsc,r_k),\bigl(r_{k+1},\dotsc,r_m;\left.r_+ - \textstyle\sum_{i=1}^k r_i\right.\bigr)\bigr).
\end{alignat*}
For typographic consistency we shall be writing $\chi = \chi^{k,m}$ for the identification $I^m \simeq I^k \times I^{m-k}$.

\begin{lem}\label{lem:P+(g,v)} There exists one and only one family of simplicial maps\/ \eqref{eqn:P+} satisfying all of the following requirements:
\begin{enumerate}
\def\labelenumi{\upshape \arabic{enumi}.}
 \item Taken individually, they are solutions to the lifting problem\/ \eqref{eqn:P+(g,v)}; in particular, $P_+(x,v) = (\varDelta^n \times I^0_+ \simeq \varDelta^n \xto{v} V \xto{\phi} W)$ for all\/ $x \in G_0$, $v \in V^{-n}_x$.
 \item $P_+(g,v) = c^{n,m}_+\bigl(g\alpha_+ \circ \pr,\left.P_+(g,v) \mathbin| \varPi^{n,m}_+\right.\bigr)$.
 \item The restriction of\/ $P_+(g,v)$ to\/ $\varPi^{n,m}_+$ is expressed, in conformity with the prescriptions below, in terms of simplicial maps of the form\/ $P_+(h,w): \varDelta^l \times I^k_+ \to W$ with\/ $h \in G_k$, $w \in V^{-l}_{sh}$, and either\/ $k < m$ or else\/ $k = m$, $l < n$.%
\begin{subequations}
\label{eqn:S+(g,v)}
\begin{alignat}{2}
	P_+(g,v)&\circ (\delta_j \times \id) = P_+(g,d_jv) &\rlap{\qquad $0 \leq j \leq n$}
\label{eqn:P+(g,v).j}\\
	P_+(g,v)&\circ (\id \times \delta_{+,i|0}) = P_+(d_ig,v) &\rlap{\qquad $1 \leq i \leq m$}
\label{eqn:P+(g,v).i|0}\\
	P_+(g,v)&\circ (\id \times \delta_{a,i|1}) = {}\notag\\* &\hskip-1em \left\lbrace%
\begin{alignedat}{2}
	&	P\bigl(t_{m-i}g,\left.P_+(s_ig,v) \circ (\id \times \delta_{a,i|1})\right.\bigr) \circ (\id \times \chi^{i-1,m-1}_a)   &\rlap{\qquad $a < i < m$} \\
	&	P\bigl(t_{m-i}g,\left.P_+(s_ig,v) \circ (\id \times \delta_{i,i|1})\right.\bigr) \circ (\id \times \chi^{i,m})         &\rlap{\qquad $a = i < m$} \\
	&	P_+\bigl(t_{m-i}g,\left.P(s_ig,v) \circ (\id \times \delta_{i|1})\right.\bigr) \circ (\id \times \chi^{i-1,m-1}_{a-1}) &\rlap{\qquad $i < a \leq m$}
\end{alignedat}
\right.\kern-\nulldelimiterspace
\label{eqn:P+(g,v).i|1}\\
	P_+(g,v)&\circ (\id \times \delta_\top) = \phi \circ P(g,v)
\label{eqn:P+(g,v).top}\\
	P_+(g,v)&\circ (\id \times \delta_\bot) = P(g,\phi v)
\label{eqn:P+(g,v).bot}
\end{alignat}
\end{subequations}
\end{enumerate} \end{lem}

Our formulas \eqref{eqn:P+(g,v).i|1} are to be understood in the first two cases in exactly the same way as equation \eqref{eqn:P(g,v).i|1} of section \ref{sub:splitting}. Although the expressions of the form $P(h,T)$ in these two formulas look pretty much the same, the $T$ in the first expression is a simplicial map $\varDelta^n \times I^{i-1}_a \to W$, therefore our formula defines a map $\varDelta^n \times I^{m-1}_a \simeq \varDelta^n \times I^{i-1}_a \times I^{m-i} \to W$, whereas in the second expression it is one $\varDelta^n \times I^i \to W$, so that we get a map $\varDelta^n \times I^m \simeq \varDelta^n \times I^i \times I^{m-i} \to W$. The interpretation of \eqref{eqn:P+(g,v).i|1} in the case $a > i$ is similar: for any $k$-simplex $h \in G_k$ with $k < m$ and for any simplicial map $T: A \to V$ whose composition with $V \xto{p} G$ factors through $\varDelta^0 \xto{sh} G$ we consider the simplicial map $P_+(h,T): A \times I^k_+ \to W$ characterized by the requirement that for every $l \geq 0$ and for every $l$-simplex $\varDelta^l \xto{a} A$ the composition $\varDelta^l \times I^k_+ \xto{a\times\id} A \times I^k_+ \xto{P_+(h,T)} W$ be equal to $P_+(h,Ta)$; in the case of interest to us we have $T: \varDelta^n \times I^{i-1} \to V$ and so our formula produces a map $\varDelta^n \times I^{m-1}_{a-1} \simeq \varDelta^n \times I^{i-1} \times I^{m-i}_{a-i} \subset \varDelta^n \times I^{i-1} \times I^{m-i}_+ \to W$.

As in the construction of the curvature tensors $R_m$, the next step in the construction of $\spl{\phi}$ involves taking the “alternating sum” of the maximal simplices of the $m$-dimensional simplicial complex $P_+(g,v) \mathbin| \varDelta^n \times \partial_{m|1}I^m_+$. We notice first of all that $\partial_{m|1}I^m_+$ decomposes into a copy of $I^{m-1}_+$, namely, $\im(\delta_{+,m|1})$, and, sitting on top of that, one of $I^m$, namely, $\im(\delta_{m,m|1})$; cf.~Figure \ref{fig:i|1+}. We have already classified the maximal simplices $\sigma_\pi: \varDelta^{n+m} \to \varDelta^n \times I^m$, $\pi \in \mathfrak{S}^{n,m}$, of the $n$-simplex times the $m$-cube. As to those of $\varDelta^n \times I^{m-1}_+$, we observe that for each $a < m$ we have an embedding of simplicial sets
\begin{equation}
	\iota_a = \iota^{m-1}_a: I^{m-1}_a \longto I^m, \quad%
		(r_1,\dotsc,r_{m-1};r_+) \longmapsto \bigl(r_1,\dotsc,r_{m-1},\left.r_+ - \textstyle\sum_{i=1}^{a-1} r_i\right.\bigr)
\label{eqn:iota}
\end{equation}
under which every maximal simplex of $\varDelta^n \times I^{m-1}_a$ transforms into one of $\varDelta^n \times I^m$ whose corresponding permutation $\pi \in \mathfrak S^{n,m}$ satisfies $\pi^{-1}(n + a) < \pi^{-1}(n + m)$. Let $\mathfrak{S}^{n,m-1}_a \subset \mathfrak{S}^{n,m}$ be the set of all such permutations. For each $\pi \in \mathfrak{S}^{n,m-1}_a$ there is exactly one maximal simplex $\sigma_{a,\pi} = \sigma^{n,m-1}_{a,\pi}: \varDelta^{n+m} \to \varDelta^n \times I^{m-1}_a$ whose composition with $\id \times \iota_a$ equals $\sigma_\pi$. Now, by definition, for all $g \in G_m$, $v \in V^{-n}_{sg}$,%
\begin{multline}
 \phi_m^{-n}(g)v
	= \sum_{a=1}^{m-1}\Bigl\{\sum_{\pi\in\mathfrak S^{n,m-1}_a} \sgn(\pi)P_+(g,v) \circ (\id \times \delta_{a,m|1})\sigma_{a,\pi}\Bigr\} + {}\\[-.2\baselineskip]
		\sum_{\pi\in\mathfrak S^{n,m}} \sgn(\pi)P_+(g,v) \circ (\id \times \delta_{m,m|1})\sigma_\pi \in W^{-m-n}_{tg},
\label{eqn:phi}
\end{multline}
in particular, $\phi_0^{-n}(x)v = \phi v \in W^{-n}_x$ for all $x \in G_0$, $v \in V^{-n}_x$. The considerations of sections \ref{sub:fibrations} and \ref{sub:cleavages} imply that the maps \(%
	\phi_m^{-n}(g): V^{-n}_{sg} \to W^{-m-n}_{tg}
\) thus defined must be part of a single morphism $\phi_m^{-n}: s^*V^{-n} \to t^*W^{-m-n} \in \VB(G_m)$; indeed, on the analogy of our argument for the Moore curvature tensors $R_m^{-n}$, it will suffice to show that the assignment $(g,v) \to P_+(g,v)$ provides a morphism in $\VB(G_m)$ from $s^*V^{-n}$ to the pullback of $(\varDelta^n \times I^m_+)(W) \to (\varDelta^n \times I^m_+)(G)$ along $G_m \simeq \varDelta^m(G) \xto{(\alpha_+\circ\pr)^*} (\varDelta^n \times I^m_+)(G)$, which we can do with the aid of Lemma \ref{lem:17A.28.1} by induction on the recursive definition of the maps $P_+(g,v)$ after suitably expressing $A = \varPi^{n,m}_+$ as a coequalizer of simplicial maps.

We proceed to check that our tensors $\phi_m^{-n}$ satisfy the defining equations \eqref{eqn:intertwiner} for an intertwiner of the two Moore representations $V^\bullet = (V^\bullet,R)$ and $W^\bullet = (W^\bullet,S)$. We are going to need the following three auxiliary lemmas, whose proofs will be given in section \ref{sub:auxiliary}:

\begin{lem}\label{prop:17A.17.7} The equality hereafter is valid for every simplicial map\/ $T: \varDelta^n \times I^{m-1}_+ \to W$ whose composition with\/ $W \xto{q} G$ factors through a single vertex\/ $\varDelta^0 \xto{x} G$.
\begin{alignat}{2}
 \sum_{k=0}^{n+m} (-1)^k\sum_{a=1}^{m-1}\Bigl\{\sum_{\pi\in\mathfrak S^{n,m-1}_a} \sgn(\pi)T\sigma_{a,\pi}\Bigr\}\delta_k &
	= \sum_{j=0}^n (-1)^j\sum_{a=1}^{m-1}\Bigl\{\sum_{\pi\in\mathfrak S^{n-1,m-1}_a} \sgn(\pi) \cdot {}\notag\\[-.4\baselineskip] & &\llap{$%
			T \circ (\delta_j \times \id)\sigma_{a,\pi}\Bigr\}$} \notag\\* &\hskip-11em%
		+ (-1)^n\sum_{i=1}^{m-1} (-1)^i\Biggl[\sum_{a=1}^{m-2}\Bigl\{\sum_{\pi\in\mathfrak S^{n,m-2}_a} \sgn(\pi)T \circ (\id \times \delta_{+,i|0})\sigma_{a,\pi}\Bigr\} - {}\notag\\ &\hskip-2em%
			\sum_{i<a<m}\Bigl\{\sum_{\pi\in\mathfrak S^{n,m-2}_{a-1}} \sgn(\pi)T \circ (\id \times \delta_{a,i|1})\sigma_{a-1,\pi}\Bigr\} + {}\notag\\ &
			(-1)^{m-i}\sum_{\pi\in\mathfrak S^{n,m-1}} \sgn(\pi)T \circ (\id \times \delta_{i,i|1})\sigma_\pi - {}\notag\\ &\hskip+2em%
			\sum_{1\leq a<i}\Bigl\{\sum_{\pi\in\mathfrak S^{n,m-2}_a} \sgn(\pi)T \circ (\id \times \delta_{a,i|1})\sigma_{a,\pi}\Bigr\}\Biggr] \notag\\* &\hskip-11em%
		- (-1)^{n+m}\sum_{\pi\in\mathfrak S^{n,m-1}} \sgn(\pi)\{T \circ (\id \times \delta_\top)\sigma_\pi - T \circ (\id \times \delta_\bot)\sigma_\pi\}
\label{eqn:17A.17.11}
\end{alignat} \end{lem}

\begin{lem}\label{lem:17A.20.3} Let\/ $1 \leq a < i < m$. For every simplicial map\/ $T: \varDelta^n \times I^{i-1}_a \to W$ whose composition with\/ $W \xto{q} G$ factors through\/ $\varDelta^0 \xto{sh} G$ for a given\/ $h \in G_{m-i}$ (\/$m > i$),
\begin{multline}
 \sum_{\pi\in\mathfrak S^{n,m-2}_a} \sgn(\pi)P(h,T) \circ (\id \times \chi_a\delta_{a,m-1|1})\sigma_{a,\pi}
	= (-1)^{m-i-1}\sum_{\pi\in\mathfrak S^{i+n,m-i-1}} \sgn(\pi) \cdot {}\\
		P\bigl(h,\left.\textstyle\sum_{\pi'\in\mathfrak S^{n,i-1}_a} \sgn(\pi')T\sigma_{a,\pi'}\right.\bigr) \circ (\id \times \delta_{m-i|1})\sigma_\pi.
\label{eqn:17A.20.2}
\end{multline} \end{lem}

\begin{lem}\label{lem:17A.20.2} Let\/ $T: \varDelta^n \times I^k \to V$ be a simplicial map whose composition with\/ $V \xto{p} G$ factors through\/ $\varDelta^0 \xto{sh} G$ for a given\/ $h \in G_{m-k}$ (\/$m > k$). Then, the equality%
\begin{subequations}
\label{eqn:17A.20.1}
\begin{multline}
 \sum_{\pi\in\mathfrak S^{n,m-1}_a} \sgn(\pi)P_+(h,T) \circ (\id \times \chi_a\delta_{a,m|1})\sigma_{a,\pi}
	= \sum_{\pi\in\mathfrak S^{k+n,m-k-1}_{a-k}} \sgn(\pi) \cdot {}\\
		P_+\bigl(h,\left.\textstyle\sum_{\pi'\in\mathfrak S^{n,k}} \sgn(\pi')T\sigma_{\pi'}\right.\bigr) \circ (\id \times \delta_{a-k,m-k|1})\sigma_{a-k,\pi}
\label{eqn:17A.20.1a}
\end{multline}
is satisfied for every\/ $a$ between\/ $k + 1$ and\/ $m - 1$. Furthermore, the following holds.
\begin{multline}
 \sum_{\pi\in\mathfrak S^{n,m}} \sgn(\pi)P_+(h,T) \circ (\id \times \chi_m\delta_{m,m|1})\sigma_\pi
	= \sum_{\pi\in\mathfrak S^{k+n,m-k}} \sgn(\pi) \cdot {}\\
		P_+\bigl(h,\left.\textstyle\sum_{\pi'\in\mathfrak S^{n,k}} \sgn(\pi')T\sigma_{\pi'}\right.\bigr) \circ (\id \times \delta_{m-k,m-k|1})\sigma_\pi
\label{eqn:17A.20.1b}
\end{multline}
\end{subequations} \end{lem}

We can now verify at once that our tensors $\phi_m^{-n}$ obey the intertwiner axioms \eqref{eqn:intertwiner}; we only need to use Lemmas \ref{prop:17A.17.7} and \ref{prop:17A.7.2}, our formulas \eqref{eqn:S+(g,v)}, and Lemmas \ref{lem:17A.20.2}, \ref{lem:17A.10.3}, and \ref{lem:17A.20.3}, successively:
\begin{alignat*}{2}
 (-1)^{n+m-1}S_0(tg)\phi_m(g)v &
	= \sum_{k=0}^{n+m} (-1)^k\sum_{a=1}^{m-1}\Bigl\{\sum_{\pi\in\mathfrak S^{n,m-1}_a} \sgn(\pi)P_+(g,v)(\id \times \delta_{+,m|1})\sigma_{a,\pi}\Bigr\}\delta_k \\* &\justify%
		+ \sum_{k=0}^{n+m} (-1)^k\Bigl\{\sum_{\pi\in\mathfrak S^{n,m}} \sgn(\pi)P_+(g,v)(\id \times \delta_{m,m|1})\sigma_\pi\Bigr\}\delta_k
\\	&\hskip-10em%
	= \sum_{j=0}^n (-1)^j\sum_{a=1}^{m-1}\Bigl\{\sum_{\pi\in\mathfrak S^{n-1,m-1}_a} \sgn(\pi)P_+(g,v)(\delta_j \times \delta_{+,m|1})\sigma_{a,\pi}\Bigr\} \\* &\hskip-10em\justify%
		+ (-1)^n\sum_{i=1}^{m-1} (-1)^i\Biggl[\sum_{a=1}^{m-2}\Bigl\{\sum_{\pi\in\mathfrak S^{n,m-2}_a} \sgn(\pi)P_+(g,v)(\id \times \underset{\makebox[.em][c]{$= \delta_{+,i|0}\delta_{+,m-1|1}$}}{\underline{\delta_{+,m|1}\delta_{+,i|0}}})\sigma_{a,\pi}\Bigr\} - {}\\ &\hskip-.em%
			\sum_{i<a<m}\Bigl\{\sum_{\pi\in\mathfrak S^{n,m-2}_{a-1}} \sgn(\pi)P_+(g,v)(\id \times \underset{\makebox[.em][c]{$= \delta_{a,i|1}\delta_{a-1,m-1|1}$}}{\underline{\delta_{+,m|1}\delta_{a,i|1}}})\sigma_{a-1,\pi}\Bigr\} + {}\\ &\hskip+2em%
			(-1)^{m-i}\sum_{\pi\in\mathfrak S^{n,m-1}} \sgn(\pi)P_+(g,v)(\id \times \underset{\makebox[.em][c]{$= \delta_{i,i|1}\delta_{m|1}$}}{\underline{\delta_{+,m|1}\delta_{i,i|1}}})\sigma_\pi - {}\\ & & &\llap{$\displaystyle%
			\sum_{1\leq a<i}\Bigl\{\sum_{\pi\in\mathfrak S^{n,m-2}_a} \sgn(\pi)P_+(g,v)(\id \times \underset{\makebox[.em][c]{$= \delta_{a,i|1}\delta_{a,m-1|1}$}}{\underline{\delta_{+,m|1}\delta_{a,i|1}}})\sigma_{a,\pi}\Bigr\}\Biggr]$} \\* &\hskip-10em\justify%
		- (-1)^{n+m}\sum_{\pi\in\mathfrak S^{n,m-1}} \sgn(\pi)\{\underset{\makebox[.em][c]{cancels out}}{\underline{P_+(g,v)(\id \times \delta_{+,m|1}\delta_\top)\sigma_\pi}} - P_+(g,v)(\id \times \underset{\makebox[.em][c]{$= \delta_\bot\delta_{m|1}$}}{\underline{\delta_{+,m|1}\delta_\bot}})\sigma_\pi\} \\* &\hskip-10em\justify%
		+ \sum_{j=0}^n (-1)^j\sum_{\pi\in\mathfrak S^{n-1,m}} \sgn(\pi)P_+(g,v)(\delta_j \times \delta_{m,m|1})\sigma_\pi \\* &\hskip-10em\justify%
		+ (-1)^{n+m}\sum_{\pi\in\mathfrak S^{n,m-1}} \sgn(\pi)\{\underset{\makebox[.em][c]{cancels out}}{\underline{P_+(g,v)(\id \times \delta_{m,m|1}\delta_{m|0})\sigma_\pi}} - {}\\[-.5\baselineskip] & & &\llap{$%
			P_+(g,v)(\id \times \underset{\makebox[.em][c]{$= \delta_\top\delta_{m|1}$}}{\underline{\delta_{m,m|1}\delta_{m|1}}})\sigma_\pi\}$} \\*[-.5\baselineskip] &\hskip-10em\justify%
		+ (-1)^n\sum_{i=1}^{m-1} (-1)^i\sum_{\pi\in\mathfrak S^{n,m-1}} \sgn(\pi)\{P_+(g,v)(\id \times \underset{\makebox[.em][c]{$= \delta_{+,i|0}\delta_{m-1,m-1|1}$}}{\underline{\delta_{m,m|1}\delta_{i|0}}})\sigma_\pi - {}\\ & & &\llap{$%
			P_+(g,v)(\id \times \underset{\makebox[.em][c]{$= \delta_{m,i|1}\delta_{m-1,m-1|1}~$}}{\underline{\delta_{m,m|1}\delta_{i|1}}})\sigma_\pi\}$}
\\	&\hskip-10em%
	= \sum_{j=0}^n (-1)^j\Biggl[\sum_{a=1}^{m-1}\Bigl\{\sum_{\pi\in\mathfrak S^{n-1,m-1}_a} \sgn(\pi)P_+(g,d_jv)(\id \times \delta_{+,m|1})\sigma_{a,\pi}\Bigr\} + {}\\[-.5\baselineskip] & & &\llap{$\displaystyle%
			\sum_{\pi\in\mathfrak S^{n-1,m}} \sgn(\pi)P_+(g,d_jv)(\id \times \delta_{m,m|1})\sigma_\pi\Biggr]$} \\* &\hskip-10em\justify%
		+ (-1)^n\sum_{i=1}^{m-1} (-1)^i\Biggl[\sum_{a=1}^{m-2}\Bigl\{\sum_{\pi\in\mathfrak S^{n,m-2}_a} \sgn(\pi)P_+(d_ig,v)(\id \times \delta_{+,m-1|1})\sigma_{a,\pi}\Bigr\} + {}\\[-.5\baselineskip] & & &\llap{$\displaystyle%
			\sum_{\pi\in\mathfrak S^{n,m-1}} \sgn(\pi)P_+(d_ig,v)(\id \times \delta_{m-1,m-1|1})\sigma_\pi\Biggr]$} \\* &\hskip-10em\justify%
		- (-1)^n\sum_{i=1}^{m-1} (-1)^i\Biggl[\sum_{\pi\in\mathfrak S^{n,m-1}} \sgn(\pi)P_+\bigl(t_{m-i}g,\left.P(s_ig,v)(\id \times \delta_{i|1})\right.\bigr) \circ {}\\[-.5\baselineskip] & & &\llap{$%
				(\id \times \chi_{m-1}\delta_{m-1,m-1|1})\sigma_\pi + {}\hphantom{\Biggr]}$} \\ &\hskip-4em%
			\sum_{i<a<m}\Bigl\{\sum_{\pi\in\mathfrak S^{n,m-2}_{a-1}} \sgn(\pi)P_+\bigl(t_{m-i}g,\left.P(s_ig,v)(\id \times \delta_{i|1})\right.\bigr) \circ {}\\[-.5\baselineskip] & & &\llap{$%
				(\id \times \chi_{a-1}\delta_{a-1,m-1|1})\sigma_{a-1,\pi}\Bigr\} - {}\hphantom{\Biggr]}$} \\ & & &\llap{$\displaystyle%
			(-1)^{m-i}\sum_{\pi\in\mathfrak S^{n,m-1}} \sgn(\pi)P\bigl(t_{m-i}g,\left.P_+(s_ig,v)(\id \times \delta_{i,i|1})\right.\bigr) \circ (\id \times \chi\delta_{m|1})\sigma_\pi + {}\hphantom{\Biggr]}$} \\ & & &\llap{$\displaystyle%
			\sum_{1\leq a<i}\Bigl\{\sum_{\pi\in\mathfrak S^{n,m-2}_a} \sgn(\pi)P\bigl(t_{m-i}g,\left.P_+(s_ig,v)(\id \times \delta_{a,i|1})\right.\bigr) \circ (\id \times \chi_a\delta_{a,m-1|1})\sigma_{a,\pi}\Bigr\}\Biggr]$} \\* &\hskip-10em\justify%
		+ (-1)^{n+m}\sum_{\pi\in\mathfrak S^{n,m-1}} \sgn(\pi)\{P(g,\phi v)(\id \times \delta_{m|1})\sigma_\pi - \phi P(g,v)(\id \times \delta_{m|1})\sigma_\pi\}
\\	&\hskip-10em%
	= \sum_{j=0}^n (-1)^j\phi_m(g)d_jv
		+ (-1)^n\sum_{i=1}^{m-1} (-1)^i\phi_{m-1}(d_ig)v \\* &\hskip-10em\justify%
		- (-1)^n\sum_{i=1}^{m-1} (-1)^i\Biggl[\sum_{\pi\in\mathfrak S^{i-1+n,m-i}} \sgn(\pi)P_+\bigl(t_{m-i}g,R_i(s_ig)v\bigr)(\id \times \delta_{m-i,m-i|1})\sigma_\pi + {}\\ & & &\llap{$\displaystyle%
			\sum_{a-i=1}^{m-i-1}\Bigl\{\sum_{\pi\in\mathfrak S^{i-1+n,m-i-1}_{a-i}} \sgn(\pi)P_+\bigl(t_{m-i}g,R_i(s_ig)v\bigr) \circ (\id \times \delta_{a-i,m-i|1})\sigma_{a-i,\pi}\Bigr\} - {}\hphantom{\Biggr]}$} \\ & & &\llap{$\displaystyle%
			(-1)^{m-i}S_{m-i}(t_{m-i}g)\mathinner{\bigl(\textstyle\sum_{\pi\in\mathfrak S^{n,i}} \sgn(\pi)P_+(s_ig,v)(\id \times \delta_{i,i|1})\sigma_\pi\bigr)} - {}\hphantom{\Biggr]}$} \\ & & &\llap{$\displaystyle%
			(-1)^{m-i}\sum_{1\leq a<i} S_{m-i}(t_{m-i}g)\mathinner{\bigl(\textstyle\sum_{\pi\in\mathfrak S^{n,i-1}_a} \sgn(\pi)P_+(s_ig,v)(\id \times \delta_{a,i|1})\sigma_{a,\pi}\bigr)}\Biggr]$} \\* &\hskip-10em\justify%
		+ (-1)^{n+m}S_m(g)\phi v
		- (-1)^{n+m}\phi R_m(g)v
\\	&\hskip-10em%
	= (-1)^{n-1}\phi_m(g)R_0(sg)v
		+ (-1)^n\sum_{i=1}^{m-1} (-1)^i\phi_{m-1}(d_ig)v \\* &\hskip-10em\justify%
		- (-1)^n\sum_{i=1}^{m-1} (-1)^i\phi_{m-i}(t_{m-i}g)R_i(s_ig)v
		+ (-1)^{n+m}\sum_{i=1}^{m-1} S_{m-i}(t_{m-i}g)\phi_i(s_ig)v \\* &\hskip-10em\justify%
		+ (-1)^{n+m}S_m(g)\phi_0(sg)v
		- (-1)^{n+m}\phi_0(tg)R_m(g)v.
\end{alignat*}

Thus $\phi^\bullet = \{\phi_m^{-n}\}$ is an “intertwiner” of the two (non-unital) Moore representations $V^\bullet$ and $W^\bullet$. In order to obtain an intertwiner $\spl{\phi} = \{\spl{\phi}_m^{-n}\}$ of the two Dold–Kan representations $\spl{V}$ and $\spl{W}$, it will be enough to apply the normalization map \eqref{eqn:nor} to $\phi^\bullet$ or, more exactly, prove that for any given $g \in G_m$ there is one and only one linear map $\spl{\phi}_m^{-n}(g): \spl{V}^{-n}_{sg} \to \spl{W}^{-m-n}_{tg}$ such that for every $v \in V^{-n}_{sg}$
\begin{equation}
	\spl{\phi}_m^{-n}(g)(\nor_{sg}^{-n}v) = \nor^{-m-n}_{tg}\bigl(\phi_m^{-n}(g)v\bigr).
\label{eqn:16A.7.3+}
\end{equation}
Showing that the right-hand side of this equation only depends on $\nor_{sg}^{-n}v$ will require the following analog of Lemma \ref{lem:16A.7.1}, whose proof relies on certain properties of the canonical interpolation cleavage and therefore will have to wait until section \ref{sub:canonical}; once \eqref{eqn:16A.7.3+} is proven, we shall also be able to see that our new normalized curvature tensors $\spl{\phi}_m^{-n}$ satisfy the defining equations \eqref{eqn:intertwiner} for an intertwiner of representations up to homotopy $\spl{V} \to \spl{W}$; it will suffice to apply the normalization map to both members of the homologous equations satisfied by the “unnormalized” tensors $\phi_m^{-n}$, and then invoke \eqref{eqn:16A.7.3+}.

\begin{lem}\label{lem:16A.7.1+} For all\/ $g \in G_m$, $v \in V^{-n}_{sg}$, and\/ $j = 0$,~$\dotsc$,~$n$,
\begin{equation}
	P_+(g,u_jv) = P_+(g,v) \circ (\upsilon_j \times \id): \varDelta^{n+1} \times I^m_+ \longto W.
\label{eqn:16A.7.1+}
\end{equation} \end{lem}

Back to the proof of \eqref{eqn:16A.7.3+}, all we need to do is make sure that $\phi_m(g)(u_jv)$ lies in the subspace spanned by all degenerate simplices, for then its normalization will be zero and so the desired identity will hold because by linearity \(%
	\nor_{tg}\bigl(\phi_m(g)(\id - u_jd_{j+1})v\bigr)
		= \nor_{tg}\bigl(\phi_m(g)v\bigr)
\) for all $j$, $v$. Now, by \eqref{eqn:phi} and \eqref{eqn:16A.7.1+},%
\begin{multline*}
 \phi_m(g)(u_jv)
	= \sum_{a=1}^{m-1}\Bigl\{\sum_{\pi\in\mathfrak S^{n+1,m-1}_a} \sgn(\pi)P_+(g,v)(\id \times \delta_{a,m|1}) \circ (\upsilon_j \times \id)\sigma_{a,\pi}\Bigr\} + {}\\[-.2\baselineskip]
		\sum_{\pi\in\mathfrak S^{n+1,m}} \sgn(\pi)P_+(g,v)(\id \times \delta_{m,m|1}) \circ (\upsilon_j \times \id)\sigma_\pi
\end{multline*}
where $(\upsilon_j \times \id)\sigma_{a,\pi} = \sigma_{a,\pi'}\upsilon_{j'}$ for a suitable choice of $\pi' \in \mathfrak S^{n,m-1}_a$, $j' \in [n + m]$ (unique factorization into monic plus epic) and similarly $(\upsilon_j \times \id)\sigma_\pi = \sigma_{\pi'}\upsilon_{j'}$ for a suitable choice of $\pi' \in \mathfrak S^{n,m}$, $j' \in [n + m]$.

In order to be rightfully entitled to call $\spl{\phi}: \spl{V} \to \spl{W}$ an intertwiner we still need to make sure it enjoys one last property, whose verification we shall postpone to section \ref{sub:canonical}:

\begin{lem}\label{cor:16A.10.3+} For all\/ $g \in G_m$, $i = 0$,~$\dotsc$,~$m$,
\begin{equation}
	\spl{\phi}_{m+1}(u_ig) = 0.
\label{eqn:16A.10.3+}
\end{equation} \end{lem}

\subsection{A few technical lemmas}\label{sub:auxiliary}

This section is dedicated to demonstrating Lemmas \ref{lem:P+(g,v)} to \ref{lem:17A.20.2}. In order to keep track of domains and codomains of maps more efficiently in the present context we shall be writing $A$ rather than $\id$ for the identity transformation $A \xto= A$ of a simplicial set $A$; this will cause no confusion.

\subsubsection*{Proof of Lemma \ref{lem:P+(g,v)}}

When $m = 0$ our maps $P_+(g,v)$ do of course exist and are determined uniquely by \eqref{eqn:P+(g,v)} for all $v \in V^{-n}_{sg}$, $n \geq 0$. In order to show they also exist for all $g \in G_m$, $m \geq 1$ we argue by induction on the “lexicographic” ordering of the set of all pairs of nonnegative integers: $(k,l) < (m,n)$ iff either $k = m$ and $l < n$ or else $k < m$. Suppose that the required maps \(%
	P_+(h,w): \varDelta^l \times I^k_+ \to W
\) exist (and are unique) for all $(k,l) < (m,n)$, $h \in G_k$, and $w \in V^{-l}_{sh}$. Then, given $g \in G_m$ and $v \in V^{-n}_{sg}$, the right-hand sides of our formulas \eqref{eqn:S+(g,v)} have a well-defined meaning: writing $\partial_jI^{n,m}_+ = \partial_j\varDelta^n \times I^m_+$, $\partial_{i|r}I^{n,m}_+ = \varDelta^n \times \partial_{i|r}I^m_+$, $\partial_{i|1}I^{n,m}_a = \varDelta^n \times \partial_{i|1}I^m_a$, $\partial_\top I^{n,m}_+ = \varDelta^n \times \partial_\top I^m_+$, and $\partial_\bot I^{n,m}_+ = \varDelta^n \times \partial_\bot I^m_+$, we have $S_{+,j}: \partial_jI^{n,m}_+ \to W$, $S_{+,i|0}: \partial_{i|0}I^{n,m}_+ \to W$, $S_{a,i|1}: \partial_{i|1}I^{n,m}_a \to W$, $S_\top: \partial_\top I^{n,m}_+ \to W$, and $S_\bot: \partial_\bot I^{n,m}_+ \to W$ characterized respectively by the relations $S_{+,j} \circ (\delta_j \times I^m_+) = P_+(g,d_jv)$, $S_{+,i|0} \circ (\varDelta^n \times \delta_{+,i|0}) = P_+(d_ig,v)$, and so forth. These partial maps will be the restrictions of a single map, $S_+: \varPi^{n,m}_+ \to W$ (whose existence will make the inductive step possible), provided they agree with one another on the intersections of their domains of definition. For logical clarity, we shall subdivide the verification into five steps. Before starting, however, let us make a couple of general observations:

\begin{rmk}\label{rmk:P(h,T)} The following equality holds, for any monic simplicial map $B \xto{j} A$, for all simplicial maps $T: A \to V$ (or~$\to W$) that sit over $sh$ i.e.~factor through $\varDelta^0 \xto{sh} G$ for a given $h \in G_k$.%
\begin{subequations}
\label{eqn:P(h,T)}
\begin{equation}
	P(h,T) \circ (j \times I^k) = P(h,T \circ j)
\label{eqn:P(h,T).j}
\end{equation}
In addition, the following identities hold for all $1 \leq i \leq k$.
\begin{gather}
	P(h,T) \circ (A \times \delta_{i|0}) = P(d_ih,T)
\label{eqn:P(h,T).i|0}\\
	P(h,T) \circ (A \times \delta_{i|1}) =	P\bigl(t_{k-i}h,\left.P(s_ih,T) \circ (A \times \delta_{i|1})\right.\bigr) \circ (A \times \chi)
\label{eqn:P(h,T).i|1}
\end{gather}
\end{subequations}
The proof of \eqref{eqn:P(h,T)} simply uses the defining property of $P(h,T)$ plus the validity of \eqref{eqn:S(g,v)}. \end{rmk}

\begin{rmk}\label{rmk:P+(h,T)} Let $h \in G_k$, $0 \leq k < m$. The equality hereafter is valid, for any monic simplicial map $B \xto{j} A$, for all simplicial maps $T: A \to V$ sitting over $sh$;%
\begin{subequations}
\label{eqn:P+(h,T)}
\begin{equation}
	P_+(h,T) \circ (j \times I^k_+) = P_+(h,T \circ j)
\label{eqn:P+(h,T).j}
\end{equation}
moreover, the following identities hold for all $1 \leq a \leq k$, $1 \leq i \leq k$.
\begin{gather}
	P_+(h,T) \circ (A \times \delta_{+,i|0}) = P_+(d_ih,T)
\label{eqn:P+(h,T).i|0}\\
	P_+(h,T) \circ (A \times \delta_{a,i|1}) =%
\begin{cases}
		P\bigl(t_{k-i}h,\left.P_+(s_ih,T) \circ (A \times \delta_{a,i|1})\right.\bigr) \circ (A \times \chi_a)
			&\text{$a < i$}
	\\	P\bigl(t_{k-i}h,\left.P_+(s_ih,T) \circ (A \times \delta_{i,i|1})\right.\bigr) \circ (A \times \chi)
			&\text{$a = i$}
	\\	P_+\bigl(t_{k-i}h,\left.P(s_ih,T) \circ (A \times \delta_{i|1})\right.\bigr) \circ (A \times \chi_{a-1})
			&\text{$a > i$}
\end{cases}
\label{eqn:P+(h,T).i|1}\\
	P_+(h,T) \circ (A \times \delta_\top) = \phi \circ P(h,T)
\label{eqn:P+(h,T).top}\\
	P_+(h,T) \circ (A \times \delta_\bot) = P(h,\phi \circ T)
\label{eqn:P+(h,T).bot}
\end{gather}
\end{subequations}
The above are all immediate consequences of the defining property of $P_+(h,T)$ plus the inductive validity of \eqref{eqn:S+(g,v)} for $k < m$. \end{rmk}

\paragraph*{Step 1.} \em For any given\/ $1 \leq i < m$, the maps\/ $S_{a,i|1}$ (\/$1 \leq a \leq m$) fit together in such a way as to compose a single map\/ $S_{+,i|1}: \partial_{i|1}I^{n,m}_+ \to W$.\em

Let us introduce the abbreviations $\delta^{n,m}_{i|1} = \varDelta^n \times \delta^m_{i|1}$, $\delta^{n,m}_{a,i|1} = \varDelta^n \times \delta^m_{a,i|1}$. Since $i < m$, $P_+(s_ig,v)$ is defined and restricts to a map $T$ of $\partial_{i|1}I^{n,i}_+$ to $W$. Let $S_{\leq,i|1}$ be the map \(%
	\partial_{i|1}I^{n,m}_\leq := \bigcup_{1\leq a\leq i} \partial_{i|1}I^{n,m}_a
		\simeq \partial_{i|1}I^{n,i}_+ \times I^{m-i}
			\xto{P(h,T)} W
\), where $h = t_{m-i}g$. We claim that for $1 \leq a \leq i$ the $S_{a,i|1}$ agree with $S_{\leq,i|1}$ and, therefore, with one another; for $a < i$, this follows from the computation%
\begin{alignat*}{2} &
 S_{\leq,i|1} \circ (\varDelta^n \times \delta_{a,i|1})
	= P(h,T) \circ (\varDelta^n \times \chi_a\delta_{a,i|1}) \\ &
	= P(h,T) \circ (\varDelta^n \times [\delta_{a,i|1} \times I^{m-i}] \circ \chi_a) \\ &
	= P(t_{m-i}g,\left.P_+(s_ig,v) \circ (\varDelta^n \times \delta_{a,i|1})\right.\bigr) \circ (\varDelta^n \times \chi_a)
		&\quad	&\text{by \eqref{eqn:P(h,T).j}} \\ &
	= S_{a,i|1} \circ (\varDelta^n \times \delta_{a,i|1})
\end{alignat*}
on account of the equality $\partial_{i|1}I^{n,m}_a = \im\delta^{n,m}_{a,i|1}$; for $a = i$, a computation identical to the above but for $\chi$ replacing the last two occurrences of $\chi_a$ proves our claim.

Next, $P(s_ig,v)$ restricts to a map $T$ of $\partial_{i|1}I^{n,i} := \varDelta^n \times \partial_{i|1}I^i$ to $V$. We contend that if we write $S_{>,i|1}$ for the map \(%
	\partial_{i|1}I^{n,m}_> := \bigcup_{i<a\leq m} \partial_{i|1}I^{n,m}_a
		\simeq \partial_{i|1}I^{n,i} \times I^{m-i}_+
			\xto{P_+(h,T)} W
\), where $h = t_{m-i}g$, then $S_{>,i|1}$ agrees with $S_{a,i|1}$ on $\partial_{i|1}I^{n,m}_a$ for every $a > i$; the verification is like the above except that $\chi_{a-1}$ must be substituted for $\chi_a$ and that \eqref{eqn:P+(h,T).j} takes the place of \eqref{eqn:P(h,T).j}.

The only thing left to check is that $S_{\leq,i|1}$ matches $S_{>,i|1}$ on \(%
	\partial_{i|1}I^{n,m}_\leq \cap \partial_{i|1}I^{n,m}_>
		= \im\delta^{n,m}_{i,i|1}\delta^{n,m}_{i|1}
\):%
\begin{alignat*}{2} &
 S_{\leq,i|1} \circ (\varDelta^n \times \delta_{i,i|1}\delta_{i|1})
	= S_{i,i|1} \circ (\varDelta^n \times \delta_{i,i|1}\delta_{i|1}) \\ &
	= P\bigl(t_{m-i}g,\left.P_+(s_ig,v) \circ (\varDelta^n \times \delta_{i,i|1})\right.\bigr) \circ (\varDelta^n \times \chi\delta_{i|1}) \\ &
	= P\bigl(t_{m-i}g,\left.P_+(s_ig,v) \circ (\varDelta^n \times \delta_{i,i|1})\right.\bigr) \circ \rlap{$(\varDelta^n \times [\delta_{i|1} \times I^{m-i}] \circ \chi)$} \\ &
	= P\bigl(t_{m-i}g,\left.P_+(s_ig,v) \circ (\varDelta^n \times \delta_{i,i|1}\delta_{i|1})\right.\bigr) \circ (\varDelta^n \times \chi)
		&\quad	&\text{by \eqref{eqn:P(h,T).j}} \\ &
	= P\bigl(t_{m-i}g,\left.P_+(s_ig,v) \circ (\varDelta^n \times \delta_\top\delta_{i|1})\right.\bigr) \circ (\varDelta^n \times \chi) \\ &
	= P\bigl(t_{m-i}g,\left.\phi \circ P(s_ig,v) \circ (\varDelta^n \times \delta_{i|1})\right.\bigr) \circ (\varDelta^n \times \chi)
		&\quad	&\text{by \eqref{eqn:P+(g,v).top}, $i < m$} \\ &
	= P_+\bigl(t_{m-i}g,\left.P(s_ig,v) \circ (\varDelta^n \times \delta_{i|1})\right.\bigr) \circ (\varDelta^n \times [I^{i-1} \times \delta_\bot] \circ \chi)
		&\quad	&\text{by \eqref{eqn:P+(h,T).bot}, $m - i < m$} \\ &
	= P_+\bigl(t_{m-i}g,\left.P(s_ig,v) \circ (\varDelta^n \times \delta_{i|1})\right.\bigr) \circ (\varDelta^n \times \chi_i\delta_{i,\bot}) \\ &
	= S_{i+1,i|1} \circ (\varDelta^n \times \delta_{i+1,i|1}\delta_{i,\bot})
	= S_{>,i|1} \circ (\varDelta^n \times \delta_{i,i|1}\delta_{i|1}).
\end{alignat*}

\paragraph*{Step 2.} \em The above maps\/ $S_{+,i|1}$ (\/$1 \leq i < m$) and the maps\/ $S_{+,i|0}$ (\/$1 \leq i \leq m$) are all compatible and therefore are the restrictions of a single map\/ \(%
	S_+^{(2)}: \bigcup_{(i,r)\neq(m,1)} \partial_{i|r}I^{n,m}_+ \to W
\); in other words, any two such maps\/ $S_{+,i|r}$ and\/ $S_{+,i'|r'}$ agree on the intersection\/ \(%
	\partial_{i|r}I^{n,m}_+ \cap \partial_{i'|r'}I^{n,m}_+
\) of their domains of definition. \em We may of course suppose that $(i,r) \neq (i',r')$ and $1 \leq i < i' \leq m$, since for $i = i'$, $r \neq r'$ the intersection is empty and thus nothing need be shown. We distinguish four cases, according to the possible values of $r$,~$r' \in \{0,1\}$.

\textit{First case:\/ $r = 0$, $r' = 0$.} Abbreviating $\varDelta^n \times \delta^m_{+,i|0}$ to $\delta^{n,m}_{+,i|0}$, we have \(%
	\partial_{i|0}I^{n,m}_+ \cap \partial_{i'|0}I^{n,m}_+
		= \im\delta^{n,m}_{+,i'|0}\delta^{n,m-1}_{+,i|0}
\), whence proving our claim amounts to checking that $S_{+,i'|0} \circ \delta^{n,m}_{+,i'|0}\delta^{n,m-1}_{+,i|0} = S_{+,i|0} \circ \delta^{n,m}_{+,i'|0}\delta^{n,m-1}_{+,i|0}$. This can be done at once by using the inductive validity of \eqref{eqn:P+(g,v).i|0} for $m - 1 < m$ plus the identities \(%
	d_id_{i'}g = d_{i'-1}d_ig\), \(%
	\delta_{+,i|0}\delta_{+,i'-1|0} = \delta_{+,i'|0}\delta_{+,i|0}
\).

\textit{Second case:\/ $r = 0$, $r' = 1$.} We have \(%
	\partial_{i|0}I^{n,m}_+ \cap \partial_{i'|1}I^{n,m}_+
		= \bigcup_{1\leq a\leq m,a\neq i} \im\delta^{n,m}_{+,i|0}\delta^{n,m-1}_{\upsilon_i(a),i'-1|1}
\), with $i' < m$. Let $1 \leq a \leq m$, $a \neq i$ be given. Either $a < i'$, in which case%
\begin{alignat*}{2} &
 S_{+,i|0} \circ (\varDelta^n \times \delta_{+,i|0}\delta_{\upsilon_i(a),i'-1|1})
	= P_+(d_ig,v) \circ \rlap{$(\varDelta^n \times \delta_{\upsilon_i(a),i'-1|1})$} \\ &
	= P\bigl(t_{m-i'}d_ig,\left.P_+(s_{i'-1}d_ig,v) \circ \delta^{n,i'-1}_{\upsilon_i(a),i'-1|1}\right.\bigr) \circ (\varDelta^n \times \chi_{\upsilon_i(a)})
		&\quad	&\text{by \eqref{eqn:P+(g,v).i|1}, $m - 1 < m$} \\ &
	= P\bigl(t_{m-i'}g,\left.P_+(d_is_{i'}g,v) \circ \delta^{n,i'-1}_{\upsilon_i(a),i'-1|1}\right.\bigr) \circ (\varDelta^n \times \chi_{\upsilon_i(a)}) \\ &
	= P\bigl(t_{m-i'}g,\left.P_+(s_{i'}g,v) \circ \delta^{n,i'}_{+,i|0}\delta^{n,i'-1}_{\upsilon_i(a),i'-1|1}\right.\bigr) \circ (\varDelta^n \times \chi_{\upsilon_i(a)})
		&\quad	&\text{by \eqref{eqn:P+(g,v).i|0}, $i' < m$} \\ &
	= P\bigl(t_{m-i'}g,\left.P_+(s_{i'}g,v) \circ (\varDelta^n \times \delta_{a,i'|1}\delta_{a,i|0})\right.\bigr) \circ (\varDelta^n \times \chi_{\upsilon_i(a)}) \\ &
	= P\bigl(t_{m-i'}g,\left.P_+(s_{i'}g,v) \circ \delta^{n,i'}_{a,i'|1}\right.\bigr) \circ (\varDelta^n \times [\delta_{a,i|0} \times I^{m-i'}] \circ \chi_{\upsilon_i(a)})
		&\quad	&\text{by \eqref{eqn:P(h,T).j}} \\ &
	= P\bigl(t_{m-i'}g,\left.P_+(s_{i'}g,v) \circ (\varDelta^n \times \delta_{a,i'|1})\right.\bigr) \circ (\varDelta^n \times \chi_a\delta_{a,i|0}) \\ &
	= S_{+,i'|1} \circ (\varDelta^n \times \delta_{a,i'|1}\delta_{a,i|0})
	= S_{+,i'|1} \circ (\varDelta^n \times \delta_{+,i|0}\delta_{\upsilon_i(a),i'-1|1}),
\end{alignat*}
or $a = i'$, in which case a computation just like the above except for featuring $\chi$ instead of $\chi_{\upsilon_i(a)}$ and involving the identity \(%
	\delta_{+,i|0}\delta_{i'-1,i'-1|1} = \delta_{i',i'|1}\delta_{i|0}
\) is valid, or else $a > i'$, in which case we employ \eqref{eqn:P(g,v).i|0} instead of \eqref{eqn:P+(g,v).i|0}, and \eqref{eqn:P+(h,T).j} instead of \eqref{eqn:P(h,T).j}.

\textit{Third case:\/ $r = 1$, $r' = 0$.} \(%
	\partial_{i|1}I^{n,m}_+ \cap \partial_{i'|0}I^{n,m}_+
		= \bigcup_{1\leq a\leq m,a\neq i'} \im\delta^{n,m}_{+,i'|0}\delta^{n,m-1}_{\upsilon_{i'}(a),i|1}
\). We shall limit ourselves to enumerating the principal differences from the previous case, leaving the onus of a full verification on the reader. In the case when $i = m - 1$, we must start off our computation by using the “tautology” $P(td_mg,T) = T$, rather than the inductive validity of \eqref{eqn:P+(g,v).i|1}. Apart from that, we now have \(%
	t_{m-1-i}d_{i'}g = d_{i'-i}t_{m-i}g\), \(%
	s_is_{i'}g = s_ig
\) and so we must use \eqref{eqn:P(h,T).i|0}, resp., \eqref{eqn:P+(h,T).i|0} where we previously used \eqref{eqn:P+(g,v).i|0} plus \eqref{eqn:P(h,T).j}, resp., \eqref{eqn:P(g,v).i|0} plus \eqref{eqn:P+(h,T).j}. In addition, a few obvious identities such as $[I^{i-1}_a \times \delta_{i'-i|0}] \circ \chi_a = \chi_a\delta_{a,i'-1|0}$, $1 \leq a < i$, are needed towards the end of the computation.

\textit{Fourth case:\/ $r = 1$, $r' = 1$.} \(%
	\partial_{i|1}I^{n,m}_+ \cap \partial_{i'|1}I^{n,m}_+
		= \im\delta^{n,m}_{i',i'|1}\delta^{n,m}_{i|1} \cup
			\bigcup_{1\leq a\leq m,a\neq i'} \im\delta^{n,m}_{a,i'|1}\delta^{n,m-1}_{\upsilon_{i'}(a),i|1}
\). Let us write down the computation for $a > i'$ by way of example:%
\begin{alignat*}{2} &
 S_{+,i'|1} \circ (\varDelta^n \times \delta_{a,i'|1}\delta_{a-1,i|1}) \\ &
	= P_+\bigl(t_{m-i'}g,\left.P(s_{i'}g,v) \circ \delta^{n,i'}_{i'|1}\right.\bigr) \circ (\varDelta^n \times \chi_{a-1}\delta_{a-1,i|1}) \\ &
	= P_+\bigl(t_{m-i'}g,\left.P(s_{i'}g,v) \circ \delta^{n,i'}_{i'|1}\right.\bigr) \circ (\varDelta^n \times [\delta_{i|1} \times I^{m-i'}] \circ \chi_{a-2}) \\ &
	= P_+\bigl(t_{m-i'}g,\left.P(s_{i'}g,v) \circ \delta^{n,i'}_{i'|1}\delta^{n,i'-1}_{i|1}\right.\bigr) \circ (\varDelta^n \times \chi_{a-2})
		&\quad	&\text{by \eqref{eqn:P+(h,T).j}} \\ &
	= P_+\bigl(t_{m-i'}g,\left.P(s_{i'}g,v) \circ \delta^{n,i'}_{i|1}\delta^{n,i'-1}_{i'-1|1}\right.\bigr) \circ (\varDelta^n \times \chi_{a-2}) \\ &
	= \rlap{$P_+\Bigl(t_{m-i'}g,\left.P\bigl(t_{i'-i}s_{i'}g,\left.P(s_is_{i'}g,v) \circ \delta^{n,i}_{i|1}\right.\bigr) \circ (\varDelta^n \times \chi\delta_{i'-1|1})\right.\Bigr) \circ (\varDelta^n \times \chi_{a-2})$} \\
\intertext{by \eqref{eqn:P(g,v).i|1}, so that setting $h = t_{m-i}g$, $A = \varDelta^n \times I^{i-1}$, and $T = P(s_ig,v) \circ \delta^{n,i}_{i|1}: A \to V$,} &
	= P_+\bigl(t_{m-i'}h,\left.P(s_{i'-i}h,T) \circ (\varDelta^n \times [I^{i-1} \times \delta_{i'-i|1}] \circ \chi)\right.\bigr) \circ \rlap{$(\varDelta^n \times \chi_{a-2})$} \\ &
	= P_+\bigl(t_{m-i'}h,\left.P(s_{i'-i}h,T) \circ (A \times \delta_{i'-i|1})\right.\bigr) \circ (\varDelta^n \times [\chi \times I^{m-i'}_+] \circ \chi_{a-2})
		&\quad	&\text{by \eqref{eqn:P+(h,T).j}} \\ &
	= P_+\bigl(t_{m-i'}h,\left.P(s_{i'-i}h,T) \circ (A \times \delta_{i'-i|1})\right.\bigr) \circ \rlap{$(A \times \chi_{a-i-1}) \circ (\varDelta^n \times \chi_{a-2})$} \\ &
	= P_+(h,T) \circ (\varDelta^n \times [I^{i-1} \times \delta_{a-i,i'-i|1}] \circ \chi_{a-2})
		&\quad	&\text{by \eqref{eqn:P+(h,T).i|1}} \\ &
	= P_+\bigl(t_{m-i}g,\left.P(s_ig,v) \circ (\varDelta^n \times \delta_{i|1})\right.\bigr) \circ (\varDelta^n \times \chi_{a-1} \circ \delta_{a-1,i'-1|1}) \\ &
	= S_{+,i|1} \circ (\varDelta^n \times \delta_{a,i|1}\delta_{a-1,i'-1|1})
	= S_{+,i|1} \circ (\varDelta^n \times \delta_{a,i'|1}\delta_{a-1,i|1}).
\end{alignat*}
The computations for $a < i'$ or $a = i'$ are similar: we use \eqref{eqn:P(h,T).j} on either occasion where we used \eqref{eqn:P+(h,T).j}, and the inductive validity of \eqref{eqn:P+(g,v).i|1} instead of \eqref{eqn:P(g,v).i|1}; when $a \leq i$, we must use \eqref{eqn:P(h,T).i|1} instead of \eqref{eqn:P+(h,T).i|1}. The details are straightforward and are left to the reader.

\paragraph*{Step 3.} \em Each one of\/ $S_\top: \partial_\top I^{n,m}_+ \to W$, $S_\bot: \partial_\bot I^{n,m}_+ \to W$ matches the above map\/ $S_+^{(2)}: \bigcup_{(i,r)\neq(m,1)} \partial_{i|r}I^{n,m}_+ \to W$. \em Since\/ \(%
	\partial_\top I^{n,m}_+ \cap \partial_\bot I^{n,m}_+ \subset \partial_{1|0}I^{n,m}_+
\), the two maps must then also agree with each other and so must be part of a well-defined extension\/ \(%
	S_+^{(3)}: \varDelta^n \times \varPi^m_+ \to W
\) of\/ $S_+^{(2)}$. For convenience, we shall be writing $\delta^{n,m}_\top = \varDelta^n \times \delta_\top$, $\delta^{n,m}_\bot = \varDelta^n \times \delta_\bot$.

We have \(%
	\partial_\top I^{n,m}_+ \cap \partial_{i|r}I^{n,m}_+ = \im\delta^{n,m}_\top\delta^{n,m}_{i|r}
\). A couple of straightforward computations involving either \eqref{eqn:P(g,v).i|0}, the inductive validity of \eqref{eqn:P+(g,v).top} for $m - 1 < m$, and the identity \(%
	\delta_{+,i|0}\delta_\top = \delta_\top\delta_{i|0}
\) (case $r = 0$) or else \eqref{eqn:P(g,v).i|1}, \eqref{eqn:P+(h,T).top}, and the identity \(%
	\delta_{m,i|1}\delta_\top = \delta_\top\delta_{i|1}
\) (case $r = 1$) prove our claim about $S_\top$; the task of writing them down is left to the reader.

As to our claim about $S_\bot$, we have \(%
	\partial_\bot I^{n,m}_+ \cap \partial_{i|r}I^{n,m}_+ = \im\delta^{n,m}_\top\delta^{n,m}_{i|r}
\); when $r = 0$, we invoke \eqref{eqn:P(g,v).i|0}, the inductive validity of \eqref{eqn:P+(g,v).bot} for $m - 1 < m$, and the identity \(%
	\delta_{+,i|0}\delta_\bot = \delta_\bot\delta_{i|0}
\); when $r = 1$, we use \eqref{eqn:P(g,v).i|1}, the inductive validity of \eqref{eqn:P+(g,v).bot} for $i < m$, and either the identity \(%
	\delta_\bot\delta_{1|1} = \delta_{1,1|1}\delta_{1|0}
\) (case $i = 1$) or the identity \(%
	\delta_\bot\delta_{i|1} = \delta_{1,i|1}\delta_\bot
\) (case $i > 1$).

\paragraph*{Step 4.} \em The maps\/ $S_{+,j}: \partial_jI^{n,m}_+ \to W$ (\/$n \geq 1$, $n \geq j \geq 0$) agree on the overlaps of their domains of definition and thus make up a single map\/ $S_+^{(4)}: \partial\varDelta^n \times I^m_+ \to W$. \em Let us write $\delta^{n,m}_{+,j} = \delta^n_j \times I^m_+$. For $0 \leq j < j' \leq n$ we have \(%
	\partial_jI^{n,m}_+ \cap \partial_{j'}I^{n,m}_+ = \im\delta^{n,m}_{+,j'}\delta^{n-1,m}_{+,j}
\), and the reader will be able to see at once that $S_{+,j}$ and $S_{+,j'}$ match in view of the inductive validity of \eqref{eqn:P+(g,v).j} for $n - 1 < n$.

\paragraph*{Step 5.} \em The above two maps\/ $S_+^{(3)}: \varDelta^n \times \varPi^m_+ \to W$ and\/ $S_+^{(4)}: \partial\varDelta^n \times I^m_+ \to W$ agree on the intersection of their domains of definition and hence are the restrictions of a global map\/ $S_+ = S_+^{(5)}: \varPi^{n,m}_+ \to W$. \em We limit ourselves to pointing out that%
\begin{multline*}
 (\partial\varDelta^n \times I^m_+) \cap (\varDelta^n \times \varPi^m_+) =%
	\bigcup_{0\leq j\leq n}\Bigl(\bigcup_{1\leq i\leq m} \partial_j\varDelta^n \times \partial_{i|0}I^m_+ \cup \bigcup_{1\leq i<m} \partial_j\varDelta^n \times \partial_{i|1}I^m_+ \cup {}\\[-.2\baselineskip] \shoveright{%
		(\partial_j\varDelta^n \times \partial_\top I^m_+) \cup (\partial_j\varDelta^n \times \partial_\bot I^m_+)\Bigr)} \\%
			= \bigcup_{0\leq j\leq n}\Bigl(\bigcup_{1\leq i\leq m} \im(\delta_j \times \delta_{+,i|0}) \cup \bigcup_{\substack{1\leq a\leq m\\ 1\leq i<m}} \im(\delta_j \times \delta_{a,i|1}) \cup \im(\delta_j \times \delta_\top) \cup \im(\delta_j \times \delta_\bot)\Bigr)
\end{multline*}
and therefore the proof reduces to checking that $S_+^{(3)} \circ (\delta_j \times \delta_{+,i|0}) = S_+^{(4)} \circ (\delta_j \times \delta_{+,i|0})$ etc., a task the reader will have no difficulty carrying out with the aid of \eqref{eqn:P(g,v).j}, the inductive validity of \eqref{eqn:S+(g,v)}, and, in the case of $\delta_j \times \delta_{a,i|1}$, \eqref{eqn:P(h,T).j} or \eqref{eqn:P+(h,T).j}.

\subsubsection*{Proof of Lemma \ref{prop:17A.17.7}}

An arbitrary permutation $\theta \in \mathfrak{S}_m = \operatorname{Sym}\{1,\dotsc,m\}$ gives rise to an automorphism $\theta^*: I^m \simto I^m$, \(%
	(r_1,\dotsc,r_m) \mapsto (r_{\theta(1)},\dotsc,r_{\theta(m)})
\) of the $m$-cube. Given $\pi \in \mathfrak{S}^{n,m}$, let us write $(\varDelta^n \times \theta^*)\sigma_\pi = \sigma_{\pi^\theta}$. We have $\sgn(\pi^\theta) = \sgn(\theta)\sgn(\pi)$. We start by generalizing part of the notations introduced right before the statement of Lemma \ref{lem:P+(g,v)}.

For any subset $K \subset \{1,\dotsc,m\}$ of cardinality, say, $k$ let us define the corresponding \emph{$K$-shuffle} to be the only permutation $\theta = \theta_K$ that is increasing with image $K$ on $\{1,\dotsc,k\}$ as well as increasing on the complement of $\{1,\dotsc,k\}$. Let us set \(%
	\chi^K = \chi^{K,m} = (I^m \xto{\theta^*} I^m \simeq I^k \times I^{m-k})
\) and write $\chi^K_a = \chi^{K,m}_a$ for the isomorphisms \(%
	I^m_{\theta(a)} \simto I^k_a \times I^{m-k}
\), $1 \leq a \leq k$, resp., \(%
	I^m_{\theta(a)} \simto I^k \times I^{m-k}_{a-k}
\), $k < a \leq m$, characterized by the following relations expressed in terms of the embedding \eqref{eqn:iota}: \(%
	(\iota^k_a \times I^{m-k}) \circ \chi^K_a = \chi^{K\cup\{m+1\}} \circ \iota^m_{\theta(a)}
\), resp., \(%
	(I^k \times \iota^{m-k}_{a-k}) \circ \chi^K_a = \chi^K \circ \iota^m_{\theta(a)}
\). Of course, for $K = \{1,\dotsc,k\}$ we have $\theta = \id$ and so $\chi^K = \chi$, $\chi^K_{\theta(a)} = \chi_a$.

\begin{rmks}\label{lem:17A.17.4*} Let $K \subset \{1,\dotsc,m\}$ be any subset of cardinality $k$, with corresponding $K$-shuffle permutation $\theta = \theta_K \in \mathfrak{S}_m$. There are canonical bijections%
\begin{subequations}
\begin{align} &
	\mathfrak{S}^{n,k} \times \mathfrak{S}^{n+k,m-k} \longsimto \mathfrak{S}^{n,m}, \quad (\mu,\nu) \mapsto \mu^K\nu \\ &
	\mathfrak{S}^{n,k}_a \times \mathfrak{S}^{n+k+1,m-k} \longsimto \mathfrak{S}^{n,m}_{\theta(a)}, \quad (\mu,\nu) \mapsto \mu^K_a\nu
		& & 1 \leq a \leq k & &\\ &
	\mathfrak{S}^{n,k} \times \mathfrak{S}^{n+k,m-k}_{a-k} \longsimto \mathfrak{S}^{n,m}_{\theta(a)}, \quad (\mu,\nu) \mapsto \mu^K_a\nu
		& & k < a \leq m & &
\end{align}
\end{subequations}
characterized by the following relations, respectively.%
\begin{subequations}
\begin{align} &
	(\varDelta^n \times \chi^K) \circ \sigma_{\mu^K\nu} = (\sigma_\mu \times I^{m-k}) \circ \sigma_\nu
\label{eqn:24-02-17a}\\ &
	(\varDelta^n \times \chi^K_a) \circ \sigma_{\theta(a),\mu^K_a\nu} = (\sigma_{a,\mu} \times I^{m-k}) \circ \sigma_\nu
		& & 1 \leq a \leq k & &
\label{eqn:24-02-17b}\\ &
	(\varDelta^n \times \chi^K_a) \circ \sigma_{\theta(a),\mu^K_a\nu} = (\sigma_\mu \times I^{m-k}_{a-k}) \circ \sigma_{a-k,\nu}
		& & k < a \leq m & &
\label{eqn:24-02-17c}
\end{align}
\end{subequations}
Abbreviating $\mu^K\nu$ into $\mu\nu$ for $K = \{1,\dotsc,k\}$, we have \(%
	\sgn(\mu\nu) = \sgn(\mu)\sgn(\nu)
\) \cite[p.~35, eq.~(31)]{2018a}. For $K$ arbitrary, let us write $\theta_+ = \theta_{K\cup\{m+1\}} \in \mathfrak{S}_{m+1}$. As permutations in $\mathfrak{S}_{m+1}$, \(%
	(k + 1\:\mathellipsis\:m + 1)\theta_+ = \theta \times \id
\) and hence \(%
	\sgn(\theta_+) = (-1)^{m-k}\sgn(\theta)
\). Whenever $a \leq k$, since
\begin{multline*}
 (\varDelta^n \times \chi) \circ \sigma_{[\mu^K_a\nu]^{\theta_+}}
	= (\varDelta^n \times \chi\theta_+^*) \circ \sigma_{\mu^K_a\nu}
	= (\varDelta^n \times \chi^{K\cup\{m+1\}}\iota^m_{\theta(a)}) \circ \sigma_{\theta(a),\mu^K_a\nu} = {}\\%
	  (\varDelta^n \times [\iota^k_a \times I^{m-k}] \circ \chi^K_a) \circ \sigma_{\theta(a),\mu^K_a\nu}
	= ([\varDelta^n \times \iota^k_a] \circ \sigma_{a,\mu} \times I^{m-k}) \circ \sigma_\nu = {}\\%
	  (\sigma_\mu \times I^{m-k}) \circ \sigma_\nu
	= (\varDelta^n \times \chi) \circ \sigma_{\mu\nu},
\end{multline*}
we have $[\mu^K_a\nu]^{\theta_+} = \mu\nu$ and so \(%
	\sgn(\mu^K_a\nu) = (-1)^{m-k}\sgn(\theta)\sgn(\mu)\sgn(\nu)
\). Whenever $a > k$, we have \(%
	\sgn(\mu^K\nu) = \sgn(\theta)\sgn(\mu)\sgn(\nu)\), \(%
	\sgn(\mu^K_a\nu) = \sgn(\theta)\sgn(\mu)\sgn(\nu)
\). \end{rmks}

We are now done with preliminaries and ready to start our actual proof. Given $1 \leq a < m$, let us consider $\chi^{\{a\}}_1: I^{m-1}_a \simto I^1_1 \times I^{m-2}$, $\chi^{\{a\}}: I^{m-1} \simto I^1 \times I^{m-2}$, where now $\theta = (a\:\mathellipsis\:1) \in \mathfrak{S}_{m-1}$ is a cycle of length $a$ having $\sgn(\theta) = (-1)^{a-1}$. It follows from Remarks \ref{lem:17A.17.4*}, Lemma \ref{prop:17A.7.2}, and again Remarks \ref{lem:17A.17.4*} successively that, if we write $T'$ for the only map $\varDelta^n \times I^1_1 \times I^{m-2} \to W$ such that $T' \circ (\varDelta^n \times \chi^{\{a\}}_1) = T$,
\begin{alignat*}{2}
 \sum_{k=0}^{n+m} (-1)^k\Bigl\{\sum_{\pi\in\mathfrak{S}^{n,m-1}_a} \sgn(\pi)T\sigma_{a,\pi}\Bigr\}\delta_k
	&= (-1)^{m-a-1}\sum_{\mu\in\mathfrak{S}^{n,1}_1} \sgn(\mu) \cdot {}\\ & & &\llap{$\displaystyle%
		\sum_{k=0}^{n+m} (-1)^k\Bigl\{\sum_{\nu\in\mathfrak{S}^{n+2,m-2}} \sgn(\nu)T' \circ (\sigma_{1,\mu} \times I^{m-2})\sigma_\nu\Bigr\}\delta_k$}
\\	&\hskip-15em%
	 = (-1)^{m-a-1}\sum_{\mu\in\mathfrak{S}^{n,1}_1} \sgn(\mu)\Biggl[\sum_{j=0}^{n+2} (-1)^j\sum_{\nu\in\mathfrak{S}^{n+1,m-2}} \sgn(\nu)T' \circ (\sigma_{1,\mu}\delta_j \times I^{m-2})\sigma_\nu + {}\\ & & &\llap{$\displaystyle%
		(-1)^n\sum_{i=1}^{m-2} (-1)^i\sum_{\nu\in\mathfrak{S}^{n+2,m-3}} \sgn(\nu)\{T' \circ (\sigma_{1,\mu} \times \delta_{i|0}) - T' \circ (\sigma_{1,\mu} \times \delta_{i|1})\}\sigma_\nu\Biggr]$}
\\	&\hskip-15em%
	 = \sum_{j=0}^{n+2} (-1)^{m-a-1}\sum_{\mu\in\mathfrak{S}^{n,1}_1} (-1)^j\sgn(\mu)\sum_{\nu\in\mathfrak{S}^{n+1,m-2}} \sgn(\nu)T' \circ (\sigma_{1,\mu}\delta_j \times I^{m-2})\sigma_\nu \\* &\hskip-15em\justify%
		+ (-1)^n\sum_{\substack{i\neq a\\ i=1}}^{m-1} (-1)^i\sum_{\pi\in\mathfrak{S}^{n,m-2}_{\upsilon_i(a)}} \sgn(\pi)\{T \circ (\varDelta^n \times \delta_{a,i|0}) - T \circ (\varDelta^n \times \delta_{a,i|1})\}\sigma_{\upsilon_i(a),\pi}
\tag{*}
\end{alignat*}
on account of \eqref{eqn:24-02-17b} and of the two obvious equalities \(%
	(I^1_1 \times \delta_{i|r}) \circ \chi^{\{\upsilon_i(a)\}}_1 = \chi^{\{a\}}_1 \circ \delta_{a,\delta_a(i)|r}\) and \(%
	(-1)^{m-a-1}(-1)^i = (-1)^{\delta_a(i)}(-1)^{m-\upsilon_i(a)}
\), valid for all $1 \leq i \leq m - 2$.

We now observe what follows. First, the identity $\sigma_{1,\mu}\delta_j = (\delta_{j'} \times I^1_1) \circ \sigma_{1,\mu'}$ sets up a bijection between all pairs $(j',\mu') \in [n] \times \mathfrak{S}^{n-1,1}_1$ and all those pairs $(j,\mu) \in [n + 2] \times \mathfrak{S}^{n,1}_1$ for which neither $\mu^{-1}(n + 1)$ nor $\mu^{-1}(n + 2)$ belong to $\{j,\left.j + 1\right.\}$; we evidently have $(-1)^j\sgn(\mu) = (-1)^{j'}\sgn(\mu')$. Second, the identity $\sigma_{1,\mu}\delta_j = \sigma_{1,\mu'}\delta_{j'}$ can only hold for $(j,\mu) \neq (j',\mu') \in [n + 2] \times \mathfrak{S}^{n,1}_1$ iff both $0 < j = j' < n + 2$ and exactly one of $\mu^{-1}(n + 1)$, $\mu^{-1}(n + 2)$ belongs to $\{j,\left.j + 1\right.\}$; in this case $\mu$ and $\mu' = \mu \: (j\:j + 1)$ differ by a transposition and $(-1)^j\sgn(\mu) = -(-1)^{j'}\sgn(\mu')$. The remaining possibilities for $(j,\mu)$ are:
\begin{enumerate}
\def\labelenumi{\textcircled{\alph{enumi}}}
 \item $j = 0$ and exactly one of $\mu^{-1}(n + 1)$, $\mu^{-1}(n + 2)$ belongs to $\{j,\left.j + 1\right.\}$, equivalently, $\mu^{-1}(n + 1) = 1$. In this case the identity $\sigma_{1,\mu}\delta_0 = (\varDelta^n \times \delta_{1,1|1}) \circ \sigma_{\mu'}$ gives a bijection between all $\mu' \in \mathfrak{S}^{n,1}$ and all $\mu$, and $(-1)^j\sgn(\mu) = (-1)^n\sgn(\mu')$.
 \item $\mu^{-1}(n + 1) = j$ and $\mu^{-1}(n + 2) = j + 1$. In this case $\sigma_{1,\mu}\delta_{n+2} = (\varDelta^n \times \delta_{1,\top}) \circ \sigma_{\mu'}$ gives a bijection between all $\mu' \in \mathfrak{S}^{n,1}$ and all $\mu$, and $(-1)^j\sgn(\mu) = -(-1)^n\sgn(\mu')$.
 \item $j = n + 2$ and exactly one of $\mu^{-1}(n + 1)$, $\mu^{-1}(n + 2)$ belongs to $\{j,\left.j + 1\right.\}$, equivalently, $\mu^{-1}(n + 2) = n + 2$. In this case $\sigma_{1,\mu}\delta_{n+2} = (\varDelta^n \times \delta_{1,\bot}) \circ \sigma_{\mu'}$ gives a bijection between all $\mu' \in \mathfrak{S}^{n,1}$ and all $\mu$, and $(-1)^j\sgn(\mu) = (-1)^n\sgn(\mu')$.
\end{enumerate}

On account of these observations, we can continue our computation (*) thus:%
\begin{alignat*}{2} &
	= \sum_{j=0}^n (-1)^{m-a-1}\sum_{\mu\in\mathfrak{S}^{n-1,1}_1} (-1)^j\sgn(\mu)\sum_{\nu\in\mathfrak{S}^{n+1,m-2}} \sgn(\nu)T' \circ ([\delta_j \times I^1_1]\sigma_{1,\mu} \times I^{m-2})\sigma_\nu \\* &\justify
		+ (-1)^{n+m}(-1)^{a-1}\sum_{\mu\in\mathfrak{S}^{n,1}} \sgn(\mu)\sum_{\nu\in\mathfrak{S}^{n+1,m-2}} \sgn(\nu)\bigl\{%
			T' \circ ([\varDelta^n \times \delta_{1,1|1}]\sigma_\mu \times I^{m-2}) - {}\\ & & &\llap{$%
			T' \circ ([\varDelta^n \times \delta_{1,\top}]\sigma_\mu \times I^{m-2}) +
			T' \circ ([\varDelta^n \times \delta_{1,\bot}]\sigma_\mu \times I^{m-2})
		\bigr\}\sigma_\nu$} \\* &\justify
		+ (-1)^n\sum_{\substack{i\neq a\\ i=1}}^{m-1} (-1)^i\sum_{\pi\in\mathfrak{S}^{n,m-2}_{\upsilon_i(a)}} \sgn(\pi)\{T \circ (\varDelta^n \times \delta_{a,i|0}) - T \circ (\varDelta^n \times \delta_{a,i|1})\}\sigma_{\upsilon_i(a),\pi}
\\ &
	= \sum_{j=0}^n (-1)^j\sum_{\pi\in\mathfrak{S}^{n-1,m-1}_a} \sgn(\pi)T \circ (\delta_j \times I^{m-1}_+)\sigma_{a,\pi} \\* &\justify
		- (-1)^{n+m}\sum_{\pi\in\mathfrak{S}^{n,m-1}} \sgn(\pi)\{T \circ (\varDelta^n \times \delta_{a,a|1}) - T \circ (\varDelta^n \times \delta_{a,\top}) + T \circ (\varDelta^n \times \delta_{a,\bot})\}\sigma_\pi \\* &\justify
		+ (-1)^n\sum_{\substack{i\neq a\\ i=1}}^{m-1} (-1)^i\sum_{\pi\in\mathfrak{S}^{n,m-2}_{\upsilon_i(a)}} \sgn(\pi)\{T \circ (\varDelta^n \times \delta_{a,i|0}) - T \circ (\varDelta^n \times \delta_{a,i|1})\}\sigma_{\upsilon_i(a),\pi}
\label{eqn:17A.17.10}
\end{alignat*}
in view of \eqref{eqn:24-02-17b}, \eqref{eqn:24-02-17a}, and the identities \(%
	(\delta_{1,1|1} \times I^{m-2}) \circ \chi^{\{a\}} = \chi^{\{a\}}_1 \circ \delta_{a,a|1}
\) etc., and again in virtue of Remarks \ref{lem:17A.17.4*}. Our formula \eqref{eqn:17A.17.11} results from summing up all of these expressions as $a$ runs from $1$ to $m - 1$.

\subsubsection*{Proofs of Lemmas \ref{lem:17A.20.3} and \ref{lem:17A.20.2}}

Let $h$, $T$ be as in the statement of Lemma \ref{lem:17A.20.3}. In virtue of our Remarks \ref{lem:17A.17.4*} as applied to $K = \{1,\dotsc,i - 1\}$, we have%
\begin{alignat*}{2}
&\sum_{\pi\in\mathfrak{S}^{n,m-2}_a} \sgn(\pi)P(h,T) \circ (\varDelta^n \times \chi_a\delta_{a,m-1|1})\sigma_{a,\pi}
	= (-1)^{m-2-i+1}\sum_{\mu\in\mathfrak{S}^{n,i-1}_a} \sgn(\mu) \cdot {}\\ & & &\llap{$\displaystyle%
		\sum_{\nu\in\mathfrak{S}^{n+i,m-i-1}} \sgn(\nu)P(h,T) \circ (\varDelta^n \times [I^{i-1}_a \times \delta_{m-i|1}] \circ \chi^K_a)\sigma_{a,\mu^K_a\nu}\phantom.$}
\\	&\hskip+2em%
	= (-1)^{m-i-1}\sum_{\nu\in\mathfrak{S}^{n+i,m-i-1}} \sgn(\nu)\sum_{\mu\in\mathfrak{S}^{n,i-1}_a} \sgn(\mu)P(h,T) \circ (\sigma_{a,\mu} \times \delta_{m-i|1})\sigma_\nu
\\	&\hskip+2em%
	= (-1)^{m-i-1}\sum_{\nu\in\mathfrak{S}^{n+i,m-i-1}} \sgn(\nu)P\bigl(h,\left.\textstyle\sum_{\mu\in\mathfrak{S}^{n,i-1}_a} \sgn(\mu)T\sigma_{a,\mu}\right.\bigr) \circ (\varDelta^{n+i} \times \delta_{m-i|1})\sigma_\nu.
\end{alignat*}

Let $h$, $T$ be now as in the statement of Lemma \ref{lem:17A.20.2}. For $k < a < m$, again in virtue of Remarks \ref{lem:17A.17.4*} but now applied to $K = \{1,\dotsc,k\}$, we have%
\begin{alignat*}{2}
&\sum_{\pi\in\mathfrak{S}^{n,m-1}_a} \sgn(\pi)P_+(h,T) \circ (\varDelta^n \times \chi_a\delta_{a,m|1})\sigma_{a,\pi}
	= \sum_{\mu\in\mathfrak{S}^{n,k}} \sgn(\mu)\sum_{\nu\in\mathfrak{S}^{n+k,m-1-k}_{a-k}} \sgn(\nu) \cdot {}\\ & & &\llap{$%
		P_+(h,T) \circ (\varDelta^n \times [I^k \times \delta_{a-k,m-k|1}] \circ \chi^K_a)\sigma_{a,\mu^K_a\nu}\phantom.$}
\\	&\hskip+2em%
	= \sum_{\nu\in\mathfrak{S}^{n+k,m-k-1}_{a-k}} \sgn(\nu)\sum_{\mu\in\mathfrak{S}^{n,k}} \sgn(\mu)P_+(h,T) \circ (\sigma_\mu \times \delta_{a-k,m-k|1})\sigma_{a-k,\nu}
\\	&\hskip+2em%
	= \sum_{\nu\in\mathfrak{S}^{n+k,m-k-1}_{a-k}} \sgn(\nu)P_+\bigl(h,\left.\textstyle\sum_{\mu\in\mathfrak{S}^{n,k}} \sgn(\mu)T\sigma_\mu\right.\bigr) \circ (\varDelta^{n+k} \times \delta_{a-k,m-k|1})\sigma_{a-k,\nu}.
\end{alignat*}
The proof for $a = m$ is, if anything, easier.

\subsection{The canonical interpolation cleavage}\label{sub:canonical}

In order to define our canonical $\varPi^{n,m}_+ \xto\subset \varDelta^n \times I^m_+$-cleavage and thereby our canonical solution to the lifting problem \eqref{eqn:P+(g,v)}, we need one more auxiliary simplicial map in addition to the imbedding $\iota_a: I^m_a \to I^{m+1}$ given by \eqref{eqn:iota}:
\begin{multline}
 \eta_a = \eta^m_a: I^{m+1} \longto I^m_a, \quad (r_0,\dotsc,r_m) \longmapsto {}\\
	\bigl(r_0,\dotsc,r_{a-2},\max\{r_{a-1},r_m\},r_a,\dotsc,r_{m-1};\left.r_m + \textstyle\sum_{i=1}^{a-1} r_{i-1}\right.\bigr).
\end{multline}
In the notations of Definition \ref{defn:interpolation}, writing $\partial_\top I^m_a = \im(\delta_{a,\top})$, $\partial_\bot I^m_a = \im(\delta_{a,\bot})$, we obtain a regular anodyne extension $\varPi^{n,m}_a \xto\subset \varDelta^n \times I^m_a$ upon setting%
\begin{subequations}
\begin{gather}
\textstyle
	\varPi^m_a = \bigl(\bigcup\limits_{i=1,\dotsc,\widehat{a},\dotsc,m} \partial_{i|0}I^m_a\bigr) \cup \bigl(\bigcup\limits_{1\leq i<m} \partial_{i|1}I^m_a\bigr) \cup \partial_\bot I^m_a \cup%
\begin{cases}
		\emptyset           &\text{for $a < m$}\\
		\partial_\top I^m_a &\text{for $a = m$,}
\end{cases}
\\	\varPi^{n,m}_a = \partial\varDelta^n \times I^m_a \cup \varDelta^n \times \varPi^m_a.
\end{gather}
\end{subequations}
It will be convenient at this point to make a slight expansion of the notations which we laid down at the beginning of section \ref{sub:splitting}: for every $i = 1$,~$\dotsc$,~$m$, let us set
\begin{gather*}
\textstyle
	\varPi^m_{i|1} = \bigl(\bigcup\limits_{1\leq i'\leq m} \partial_{i'|0}I^m\bigr) \cup \bigl(\bigcup_{\substack{1\leq i'\leq m\\ i'\neq i}} \partial_{i'|1}I^m\bigr) \\
	\varPi^{n,m}_{i|1} = \partial\varDelta^n \times I^m \cup \varDelta^n \times \varPi^m_{i|1}
\end{gather*}
(in particular, $\varPi^{n,m}_{m|1} = \varPi^{n,m}$). Clearly $\varPi^{n,m}_{i|1} \xto\subset \varDelta^n \times I^m$ is a regular anodyne extension for every $i$. Let us put $A = \varDelta^n \times I^{m-1}$, $B = \partial(\varDelta^n \times I^{m-1})$ and define $c^{n,m}_{i|1}$ to be the $\varPi^{n,m}_{i|1} \xto\subset \varDelta^n \times I^m$-cleavage that corresponds to $c_\sqcup^{A,B}$ under the identification
\begin{equation*}
	A \times I \simto \varDelta^n \times I^m, \quad
		(j,r_1,\dotsc,r_{m-1};r) \mapsto (j,r_1,\dotsc,r_{i-1},r,r_i,\dotsc,r_{m-1}).
\end{equation*}
Given $h: \varDelta^n \times I^m_a \to G$, $w: \varPi^{n,m}_a \to W$ compatible in the sense that $qw = h \mathbin| \varPi^{n,m}_a$, let us write $c^{n,m}_a(h,w)$ for the only map $\varDelta^n \times I^m_a \to W$ satisfying the following two conditions.%
\begin{subequations}
\begin{gather}
	c^{n,m}_a(h,w) \mathbin| \varPi^{n,m}_a = w \\
	c^{n,m}_a(h,w) = c^{n,m+1}_{m|1}\bigl(h \circ (\id \times \eta_a),\left.c^{n,m}_a(h,w) \circ (\id \times \eta_a) \mathbin| \varPi^{n,m+1}_{m|1}\right.\bigr) \circ (\id \times \iota_a)
\end{gather}
\end{subequations}
Such a map exists and is unique (normality of the cleavage being required in the case $a = m$). The correspondence $(h,w) \mapsto c^{n,m}_a(h,w)$ is a $\varPi^{n,m}_a \xto\subset \varDelta^n \times I^m_a$-cleavage. (See Figure \ref{fig:cleavage+}.)%
\begin{figure}
\begin{tikzpicture}[scale=0.5]
% m = 1, a = 1:
 \path[fill=lightgray] (0,0)--(5,0)--(5,5)--(0,0);
 \draw[->,thick] (0,0)--(5,0); \draw[->,thick] (0,0)--(5,5);
 \draw[->,thick] (0,0)--(0,5); \draw[dotted] (0,5)--(5,5);
 \draw[->,dashed,lightgray,ultra thick] (0,-1)--(5,-1);
 \node[left] at (0,0) {$00$}; \node[right] at (5,0) {$10$};
 \node[above left] at (0,5) {$11$}; \node[above right] at (5,5) {$11$};
% m = 2, a = 1:
 \path[fill=gray] (10,-1)--(15,-1)--(17,2)--(17,7)--(15,4)--(10,-1);
 \path[fill=lightgray] (10,-1)--(15,4)--(17,7)--(12,2)--(10,-1);
 \path[fill=lightgray] (10,-1)--(10,4)--(12,7)--(12,2)--(10,-1);
 \draw[dashed,thick] (10,-1)--(17,2) (12,2)--(17,2); \draw[thick] (15,-1)--(15,4) (15,-1)--(17,7) (17,2)--(17,7);
 \draw[->,thick] (10,-1)--(15,-1); \draw[->,thick] (10,-1)--(12,2); \draw[->,thick] (15,-1)--(17,2);
 \draw[->,thick] (10,-1)--(15,4); \draw[->,thick] (15,4)--(17,7); \draw[->,thick] (10,-1)--(10,4);
 \draw[->,thick] (10,4)--(12,7); \draw[dotted] (10,4)--(15,4) (12,7)--(17,7);
 \draw[->,dashed,lightgray,ultra thick] (17,-1)--(19,2);
 \node[below left] at (10,-1) {$000$}; \node[below right] at (15,-1) {$100$};
 \node[left] at (12,2) {$010$}; \node[right] at (17,2) {$110$};
 \node[left] at (10,4) {$101$}; \node[right] at (15,4) {$101$};
 \node[left] at (12,7) {$111$}; \node[right] at (17,7) {$111$};
% m = 2, a = 2:
 \path[fill=lightgray] (21,-1)--(26,4)--(28,12)--(23,7)--(21,-1);
 \path[fill=gray] (26,4)--(28,7)--(28,12)--(26,4); \draw[->,thick] (21,-1)--(23,7);
 \path[fill=gray] (21,-1)--(26,4)--(26,9)--(21,4)--(21,-1);
 \draw[->,thick] (21,-1)--(26,4); \draw[->,thick] (26,4)--(28,7); \draw[dashed,thick] (23,2)--(28,7);
 \draw[->,thick] (26,4)--(28,12); \draw[->,dashed,thick] (21,-1)--(23,7); \draw[thick] (21,4)--(26,9) (28,7)--(28,12);
 \draw[->,thick] (21,-1)--(21,4); \draw[->,thick] (26,4)--(26,9); \draw[thick] (21,-1)--(26,9);
 \draw[->,dashed,thick] (21,-1)--(23,2); \draw[dashed,thick] (21,-1)--(28,7) (23,2)--(23,7);
 \draw[dotted] (21,4)--(23,7) (26,9)--(28,12);
 \draw[->,dashed,lightgray,ultra thick] (26,+1)--(28,4);
 \node[below left] at (21,-1) {$000$}; \node[right] at (26,4) {$101$};
 \node[left] at (23,2) {$010$}; \node[right] at (28,7) {$111$};
 \node[above left] at (21,4) {$011$}; \node[above right] at (26,9) {$112$};
 \node[above left] at (23,7) {$011$}; \node[above right] at (28,12) {$112$};
\end{tikzpicture}
\caption{Canonical interpolation cleavages $c^{0,1}_1$, $c^{0,2}_1$, $c^{0,2}_2$}
\label{fig:cleavage+}
\end{figure}

The newly defined cleavages $c^{n,m}_a$ enjoy a couple of useful properties. Let $w: \varDelta^n \times I^m_a \to W$ be an arbitrary map sitting over $h = qw$. We claim that
\begin{equation}
 c^{n,m}_a(h,\left.w \mathbin| \varPi^{n,m}_a\right.)
	= c^{n,m+1}_{m|1}\bigl(h \circ (\id \times \eta_a),\left.w \circ (\id \times \eta_a) \mathbin| \varPi^{n,m+1}_{m|1}\right.\bigr) \circ (\id \times \iota_a).
\label{eqn:Claim*}
\end{equation}
For $a < m$, this is a consequence of the intuitive fact that, if $w'$,~$w'': \varDelta^n \times I^{m+1} \to W$ both cover the same $h' = h'': \varDelta^n \times I^{m+1} \to G$ and satisfy \(%
	(w' - w'') \circ (\id \times \iota_a) \mathbin| \varPi^{n,m}_a = 0
\), then \[%
 c^{m+1}_{m|1}(h',\left.w' \mathbin| \varPi^{n,m+1}_{m|1}\right.) \circ (\id \times \iota_a)
	= c^{m+1}_{m|1}(h'',\left.w'' \mathbin| \varPi^{n,m+1}_{m|1}\right.) \circ (\id \times \iota_a).
\] For $a = m$, it holds by definition. We further claim that, for $a = m$ or $w = c^{n,m}_a(h,\left.w \mathbin| \varPi^{n,m}_a\right.)$,
\begin{equation}
	c^{n,m+1}_{m|1}\bigl(h \circ (\id \times \eta_a),\left.w \circ (\id \times \eta_a) \mathbin| \varPi^{n,m+1}_{m|1}\right.\bigr) = w \circ (\id \times \eta_a).
\label{eqn:16A.25.3}
\end{equation}

\begin{defn}\label{defn:cleavage+} Let $h: \varDelta^n \times I^m_+ \to G$, $w: \varPi^{n,m}_+ \to W$ be compatible in the sense that $qw = h \mathbin| \varPi^{n,m}_+$. There is a unique map $c^{n,m}_+(h,w): \varDelta^n \times I^m_+ \to W$ with the property that for every $a = 1$,~$\dotsc$,~$m$%
\begin{equation}
	c^{n,m}_+(h,w) \mathbin| \varDelta^n \times I^m_a = c^{n,m}_a\bigl(h \mathbin| \varDelta^n \times I^m_a,\left.c^{n,m}_+(h,w) \mathbin| \varPi^{n,m}_a\right.\bigr).
\label{eqn:cleavage+}
\end{equation}
The operator $c^{n,m}_+$ arising in this way is easily recognized to be a $\varPi^{n,m}_+ \xto\subset \varDelta^n \times I^m_+$-cleavage. We refer to it as the ($n,m$-th) \emph{canonical interpolation cleavage}. \end{defn}

\subsubsection*{Proof of Lemma \ref{lem:16A.7.1+}}

The argument is substantially the same we exploited to prove Lemma \ref{lem:16A.7.1}. We start by recording the following generalization of equation \eqref{eqn:16A.24.3*}.

\begin{lem}\label{lem:16A.24.3} Let\/ $w: \varDelta^n \times I^m \to W$ satisfy the condition\/ $w = c^{n,m}_{i|1}(qw,\left.w \mathbin| \varPi^{n,m}_{i|1}\right.)$ for a given\/ $i = 1$,~$\dotsc$,~$m$. Then, for every\/ $j = 0$,~$\dotsc$,~$n$,
\begin{equation}
	c^{n+1,m}_{i|1}\bigl(qw \circ (\upsilon_j \times \id),\left.w \circ (\upsilon_j \times \id) \mathbin| \varPi^{n+1,m}_{i|1}\right.\bigr) = w \circ (\upsilon_j \times \id).
\label{eqn:16A.24.3}
\end{equation} \end{lem}

\begin{proof} This is an immediate corollary of Lemma \ref{lem:16A.24.1} and our definitions. \end{proof}

Equation \eqref{eqn:16A.7.1+} may be established by induction on the recursive definition of the maps $P_+(g,v)$ with the aid of the following counterpart of equation \eqref{eqn:16A.24.3*}.

\begin{lem}\label{lem:16A.6.1+} Let\/ $w: \varDelta^n \times I^m_+ \to W$ be a simplicial map satisfying\/ $w = c^{n,m}_+(qw,\left.w \mathbin| \varPi^{n,m}_+\right.)$. For every\/ $j = 0$,~$\dotsc$,~$n$,
\begin{equation}
	c^{n+1,m}_+\bigl(qw \circ (\upsilon_j \times \id),\left.w \circ (\upsilon_j \times \id) \mathbin| \varPi^{n+1,m}_+\right.\bigr) = w \circ (\upsilon_j \times \id).
\label{eqn:16A.6.1+}
\end{equation} \end{lem}

\begin{proof} Proceeding by induction on $a = 1$,~$\dotsc$,~$m$, it is immediate to verify that \eqref{eqn:16A.6.1+} must be valid upon restriction to $\varPi^{n+1,m}_+ \cup \varDelta^{n+1} \times I^m_1 \cup \dotsb \cup \varDelta^{n+1} \times I^m_a$:%
\begin{alignat*}{2} &
 c^{n+1,m}_+\bigl(qw \circ (\upsilon_j \times \id),\left.w \circ (\upsilon_j \times \id) \mathbin| \varPi^{n+1,m}_+\right.\bigr) \mathbin| \varDelta^{n+1} \times I^m_a \\ &
	= c^{n+1,m}_a\bigl(qw \circ (\upsilon_j \times \id) \mathbin| \varDelta^{n+1} \times I^m_a,\left.c^{n+1,m}_+(\mathellipsis) \mathbin| \varPi^{n+1,m}_a\right.\bigr)
			&\quad	&\text{by \eqref{eqn:cleavage+}} \\ &
	= c^{n+1,m}_a\bigl(qw \circ (\upsilon_j \times \id) \mathbin| \varDelta^{n+1} \times I^m_a,\left.w \circ (\upsilon_j \times \id) \mathbin| \varPi^{n+1,m}_a\right.\bigr)
			&\quad	&\text{by induction} \\ &
	= c^{n+1,m+1}_{m|1}\bigl(qw \circ (\upsilon_j \times \eta_a),\left.w \circ (\upsilon_j \times \eta_a) \mathbin| \varPi^{n+1,m+1}_{m|1}\right.\bigr) \circ (\id \times \iota_a)
			&\quad	&\text{by \eqref{eqn:Claim*}} \\ &
	= w \circ (\id \times \eta_a)(\upsilon_j \times \id)(\id \times \iota_a)
			&\quad	&\text{by \eqref{eqn:16A.25.3} plus \eqref{eqn:16A.24.3}} \\ &
	= w \circ (\upsilon_j \times \id) \mathbin| \varDelta^{n+1} \times I^m_a.
\tag*{\qedhere}
\end{alignat*} \end{proof}

\subsubsection*{Proof of Lemma \ref{cor:16A.10.3+}}

We need a suitable analog of Lemma \ref{lem:16A.10.1}. To begin with, for any given $i = 0$,~$\dotsc$,~$m$, we seek a map $I^{m+1}_+ \to I^m_+$ whose composition with $I^m_+ \xto{\alpha_+} \varDelta^m$ equals $I^{m+1}_+ \xto{\alpha_+} \varDelta^{m+1} \xto{\upsilon_i} \varDelta^m$; an appropriate such turns out to be provided by what we call the \emph{$i$-th interpolation degeneracy map}:%
\begin{multline}
\label{eqn:epsilon+}
	\varepsilon_{+,i} = \varepsilon^m_{+,i}: I^{m+1}_+ \longto I^m_+, \quad (r_0,r_1,\dotsc,r_m;r_+) \longmapsto {}\\%
\begin{cases}
		(r_1,\dotsc,r_m;\left.r_+ - \min\{r_0,r_+\}\right.)
			&\text{for $i = 0$}\\
		(r_0,\dotsc,r_{i-2},\max\{r_{i-1},r_i\},\rlap{$r_{i+1},\dotsc,r_m;r_+)$} \\
			&\text{for $i > 0$ and $r_+ \leq \sum\limits_{a=1}^i r_{a-1}$} \\
		(r_0,\dotsc,r_{i-2},\max\{r_{i-1},r_i\},\rlap{$r_{i+1},\dotsc,r_m;\left.r_+ - \min\{r_{i-1},r_i\}\right.)$} \\
			&\text{for $i > 0$ and $r_+ \geq \sum\limits_{a=1}^{i+1} r_{a-1}$.}
\end{cases} \quad%
\end{multline}
By way of example, Figure \ref{fig:epsilon+} illustrates the effect of $\varepsilon^0_{+,0}$, $\varepsilon^1_{+,0}$, and $\varepsilon^1_{+,1}$.%
\begin{figure}
\begin{tikzpicture}[scale=0.5]
% m = 0, i = 0:
 \draw[dotted] (0,1)--(5,1) (0,1)--(5,6) (5,1)--(5,6);
 \draw[fill] (0,1) circle [radius=.05] (5,1) circle [radius=.05] (5,6) circle [radius=.05];
 \node[left] at (0,1) {$0$}; \node[right] at (5,1) {$0$}; \node[above right] at (5,6) {$0$};
% m = 1, i = 0:
 \path[fill=gray] (10,0)--(12,3)--(12,8)--(10,0); \draw[thick] (12,3)--(12,8);
 \path[fill=lightgray] (10,0)--(17,8)--(17,13)--(10,0); \draw[dashed,thick] (12,3)--(12,8);
 \draw (10,0)--(17,13) (10,0)--(17,8); \draw[dashed] (12,3)--(17,13);
 \path[fill=gray] (15,5)--(17,8)--(17,13)--(15,5); \draw[dashed] (10,0)--(17,8);
 \draw[thick] (17,8)--(17,13); \draw (10,0)--(17,3) (15,0)--(17,8);
 \draw[->,thick] (15,0)--(17,3); \draw[->,thick] (15,5)--(17,8); \draw[->,thick] (15,5)--(17,13);
 \draw[->,dashed,thick] (10,0)--(12,3); \draw[->,thick] (10,0)--(12,8);
 \draw[dotted] (10,0)--(15,0) (12,3)--(17,3) (15,0)--(15,5) (17,3)--(17,8);
 \draw[dotted] (10,0)--(15,5) (12,3)--(17,8) (12,8)--(17,13);
 \node[below left] at (10,0) {$00$}; \node[below right] at (15,0) {$00$}; \node[right] at (15,5) {$00$};
 \node[left] at (12,3) {$10$}; \node[right] at (17,3) {$10$}; \node[right] at (17,8) {$10$};
 \node[above left] at (12,8) {$11$}; \node[above right] at (17,13) {$11$};
% m = 1, i = 1:
 \path[fill=gray] (20,0)--(22,3)--(22,8)--(20,0); \draw[thick] (22,3)--(22,8);
 \path[fill=lightgray] (20,0)--(22,3)--(27,13)--(20,0); \draw[dashed,thick] (22,3)--(22,8);
 \path[fill=lightgray] (20,0)--(22,3)--(27,8)--(20,0); \draw (20,0)--(27,13) (22,3)--(27,13);
 \path[fill=lightgray] (20,0)--(27,3)--(27,8)--(20,0); \draw (20,0)--(27,3) (27,3)--(27,8);
 \path[fill=lightgray] (20,0)--(25,0)--(27,8)--(20,0); \draw (20,0)--(27,8) (22,3)--(27,8);
 \path[fill=gray] (20,0)--(25,0)--(25,5)--(20,0); \draw[dashed] (20,0)--(27,3); \draw (25,0)--(27,8);
 \draw[dotted] (22,3)--(27,3) (25,0)--(27,3) (25,5)--(27,8) (25,5)--(27,13) (27,8)--(27,13);
 \draw[->,thick] (20,0)--(22,3); \draw[->,thick] (20,0)--(22,8); \draw[dotted] (22,8)--(27,13);
 \draw[->,thick] (20,0)--(25,0); \draw[->,thick] (20,0)--(25,5); \draw[thick] (25,0)--(25,5);
 \node[below left] at (20,0) {$00$}; \node[below right] at (25,0) {$10$}; \node[right] at (25,5) {$11$};
 \node[left] at (22,3) {$10$}; \node[right] at (27,3) {$10$}; \node[right] at (27,8) {$11$};
 \node[above left] at (22,8) {$11$}; \node[above right] at (27,13) {$11$};
\end{tikzpicture}
\caption{Interpolation degeneracy maps $\varepsilon^m_{+,i}: I^{m+1}_+ \to I^m_+$, $m = 0$,~$1$}
\label{fig:epsilon+}
\end{figure}

\begin{lem}\label{lem:16A.16.1} Let\/ $w: \varDelta^n \times I^m_+ \to W$ be any simplicial map such that\/ $w = c^{n,m}_+(qw,\left.w \mathbin| \varPi^{n,m}_+\right.)$. The following identity holds for every\/ $i = 1$,~$\dotsc$,~$m$.
\begin{equation}
	c^{n,m+1}_+\bigl(qw \circ (\id \times \varepsilon_{+,i}),\left.w \circ (\id \times \varepsilon_{+,i}) \mathbin| \varPi^{n,m+1}_+\right.\bigr) = w \circ (\id \times \varepsilon_{+,i})
\label{eqn:16A.16.1}
\end{equation}
If in addition\/ $w\delta_\bot = c^{n,m}(qw\delta_\bot,\left.w\delta_\bot \mathbin| \varPi^{n,m}\right.)$, the same identity holds for\/ $i = 0$ as well. \end{lem}

\begin{proof} Let $a = 1$,~$\dotsc$,~$m + 1$, $i = 0$,~$\dotsc$,~$m$. In consequence of Lemmas \ref{lem:16A.24.1} and \ref{lem:16A.24.2}, the following generalization of equation \eqref{eqn:16A.9.1} involving the degeneracy maps \eqref{eqn:epsilon} of section \ref{sub:splitting} must hold for every map $w': \varDelta^n \times I^m \to W$ such that either $a = 1$, $i = 0$ or $w' = c^{n,m}_{\upsilon_i(a)|1}(qw',\left.w' \mathbin| \varPi^{n,m}_{\upsilon_i(a)|1}\right.)$.%
\begin{equation}
	c^{n,m+1}_{a|1}\bigl(qw' \circ (\id \times \varepsilon_i),\left.w' \circ (\id \times \varepsilon_i) \mathbin| \varPi^{n,m+1}_{a|1}\right.\bigr) = w' \circ (\id \times \varepsilon_i)
\label{eqn:16A.24.4+}
\end{equation}
As in the proof of Lemma \ref{lem:16A.6.1+}, we argue by induction on $a = 1$,~$\dotsc$,~$m + 1$. To begin with,
\begin{alignat*}{2} &
 c^{n,m+1}_+\bigl(qw \circ (\id \times \varepsilon_{+,i}),\left.w \circ (\id \times \varepsilon_{+,i}) \mathbin| \varPi^{n,m+1}_+\right.\bigr) \mathbin| \varDelta^n \times I^{m+1}_a \\ &
	= c^{n,m+1}_a\bigl(qw \circ (\id \times \varepsilon_{+,i}) \mathbin| \varDelta^n \times I^{m+1}_a,\left.c^{n,m+1}_+(\mathellipsis) \mathbin| \varPi^{n,m+1}_a\right.\bigr)
			&\quad	&\text{by \eqref{eqn:cleavage+}} \\ &
	= c^{n,m+1}_a\bigl(qw \circ (\id \times \varepsilon_{+,i}) \mathbin| \varDelta^n \times I^{m+1}_a,\left.w \circ (\id \times \varepsilon_{+,i}) \mathbin| \varPi^{n,m+1}_a\right.\bigr)
			&\quad	&\text{by induction} \\ &
	= c^{n,m+2}_{m+1|1}\bigl(qw \circ (\id \times \varepsilon_{+,i}\eta_a),\left.w \circ (\id \times \varepsilon_{+,i}\eta_a) \mathbin| \varPi^{n,m+2}_{m+1|1}\right.\bigr) \circ (\id \times \iota_a)
			&\quad	&\text{by \eqref{eqn:Claim*}.}
\tag{*}
\end{alignat*}
We observe that \(%
	\varepsilon_{+,i}\eta_a = \eta_{\upsilon_i(a)}\varepsilon_i
\) for $a \neq i + 1$, that \(%
	\varepsilon_{+,i}\eta_{i+1} = \eta_i \: (i\:m + 1)^*\varepsilon_i \: (i + 1\:m + 2)^*
\) for $i \geq 1$ (we are making use of notations we introduced while proving Lemma \ref{prop:17A.17.7} in section \ref{sub:auxiliary}), and that \(%
	\varepsilon_{+,0}\eta_1 = \delta_\bot\varepsilon_0\varepsilon_0 \: (m + 2\:\mathellipsis\:2)^*
\). By how $c^{n,m}_{i|1}$ was defined, Lemma \ref{lem:16A.24.1} implies that
\begin{equation}
 c^{n,m}_{i|1}\bigl(qw \circ (\id \times \theta^*),\left.w \circ (\id \times \theta^*) \mathbin| \varPi^{n,m}_{i|1}\right.\bigr)
	= c^{n,m}_{\theta^{-1}(i)|1}(qw,\left.w \mathbin| \varPi^{n,m}_{\theta^{-1}(i)|1}\right.) \circ (\id \times \theta^*)
\label{eqn:theta*}
\end{equation}
for all $\theta \in \operatorname{Sym}\{1,\dotsc,m\}$. Upon combining our hypotheses with \eqref{eqn:16A.25.3} and \eqref{eqn:16A.24.4+}, we conclude that
\begin{align*}
 \text{(*)}
	&=%
\begin{cases}
	w \circ (\id \times \eta_{\upsilon_i(a)}\varepsilon_i) \circ (\id \times \iota_a)
		&\text{if $a \neq i + 1$} \\
	w \circ \bigl(\id \times \eta_i \: (i\:m + 1)^*\varepsilon_i \: (i + 1\:m + 2)^*\bigr) \circ (\id \times \iota_{i+1})
		&\text{if $a = i + 1 \geq 2$} \\
	w \circ \bigl(\id \times \delta_\bot\varepsilon_0\varepsilon_0 \: (m + 2\:\mathellipsis\:2)^*\bigr) \circ (\id \times \iota_1)
		&\text{if $a = i + 1 = 1$}
\end{cases}
\\	&= w \circ (\id \times \varepsilon_{+,i}\eta_a) \circ (\id \times \iota_a)
\\	&= w \circ (\id \times \varepsilon_{+,i}) \mathbin| \varDelta^n \times I^{m+1}_a.
\tag*{\qedhere}
\end{align*} \end{proof}

\begin{lem}\label{lem:16A.10.1+} The identity below is valid for all\/ $g \in G_m$, $v \in V^{-n}_{sg}$ for every\/ $i = 0$,~$\dotsc$,~$m$.
\begin{equation}
	P_+(u_ig,v) = P_+(g,v) \circ (\id \times \varepsilon_{+,i}): \varDelta^n \times I^{m+1}_+ \longto W
\label{eqn:16A.10.1+}
\end{equation} \end{lem}

\begin{proof} It is not hard to establish \eqref{eqn:16A.10.1+} with the help of \eqref{eqn:16A.16.1} by induction on the recursive definition of the maps $P_+(g,v)$. \end{proof}

With the present lemma at disposal, it is straightforward to check that $\spl{\phi}$ does indeed satisfy the vanishing conditions \eqref{eqn:16A.10.3+}.

\subsection{Further applications}\label{sub:final}

Among the chief contributions of \cite{2018a,2022a} there is the elaboration of a functorial correspondence $\Rep^\infty_+(G) \to \VB^\infty(G)$ from unital representations up to homotopy of $G$ vanishing in positive degrees to vector fibrations over $G$. We conclude our paper with a synopsis of the key ideas involved in the process of upgrading this correspondence to a categorical equivalence between the associated derived or homotopy categories. Our splitting construction for morphisms and, specifically, our interpolation construction will prove instrumental in the argument sketched below to the effect that the derived functor must be a categorical equivalence. Familiarity with the cited references is of course desirable, yet we have tried to present our ideas in such a way that they may be grasped even by someone who has not read those works.

\paragraph{The semidirect product functor.} Let $N$ be a nonnegative integer or $\infty$. Let $\VB^N(G)$ be the full subcategory of $\VB^\infty(G)$ comprising all $N + 1$-strict vector fibrations; for the notion of $N + 1$-strictness, see the beginning of section \ref{sub:fibrations}. As in section \ref{sub:intertwiners}, let $\Rep^N_+(G)$ be the category of all $N + 1$-term unital representation up to homotopy of $G$ and their intertwiners. The \emph{semidirect product} construction of \cite[§3.1]{2018a}, whose defining formulas we are now going to reproduce for the reader's convenience, transforms each object $E = (E,R)$ of $\Rep^N_+(G)$ into one of $\VB^N(G)$. Let “$\nu: [l] \into_0 [n]$” be short for “$\nu$ is a poset injection of $[l]$ into $[n]$ sending $0 \mapsto 0$”. Given $\nu: [l] \into_0 [n]$, let $x_\nu: G_n \to G_0$ be the smooth map assigned by $G$ to $[0] \to [n]$, $0 \mapsto \nu(l)$. For each $n$ we obtain a smooth vector bundle $\sdp{E}_n \to G_n$ upon setting%
\begin{subequations}
\label{eqn:sdp}
\begin{equation}
	\sdp{E}_n = \bigoplus_{\nu:[l]\into_0[n]} {x_\nu}^*E^{-l} \in \VB(G_n)
\label{eqn:sdp1}
\end{equation}
(the direct sum, here, is supposed to run over all possible $\nu: [l] \into_0 [n]$, $0 \leq l \leq n$); each vector in $\sdp{E}_n$ can be written uniquely as the sum of $2^n$ ones of the form $(g,e,\nu)$, $g \in G_n$, $e \in E^{-l}_{x_\nu g}$. Viewing the projection onto the $\nu$-th summand as a morphism $\pr_\nu: \sdp{E}_n \to E^{-l}$ in $\VB$ covering $G_n \xto{x_\nu} G_0$, for every poset map $\theta: [m] \to [n]$ sending $0 \mapsto 0$ the prescription%
\begin{equation}
	\pr_\mu\sdp{E}_\theta(g,e,\nu) =%
\begin{cases}
		e &\text{if $k = l$ and $\theta\mu = \nu$}
	\\	0 &\text{otherwise}
\end{cases} \text{\qquad $\forall\mu: [k] \into_0 [m]$}
\label{eqn:sdp2}
\end{equation}
defines a morphism $\sdp{E}_\theta: \sdp{E}_n \to \sdp{E}_m$ in $\VB$ covering $G_\theta: G_n \to G_m$; on taking $\theta = \delta_i$, we obtain the face morphism $d_i: \sdp{E}_n \to \sdp{E}_{n-1}$, $1 \leq i \leq n$; on taking $\theta = \upsilon_j$, we obtain the degeneracy morphism $u_j: \sdp{E}_n \to \sdp{E}_{n+1}$, $0 \leq j \leq n$. The remaining face morphisms, $d_0: \sdp{E}_n \to \sdp{E}_{n-1}$, $n \geq 1$, stand quite apart: they involve the tensors $R_m$ and are given by our next formula, where we write $\mu^+$ for the poset map $[k + 1] \into_0 [n]$ obtained from $\mu: [k] \into_0 [n - 1]$ upon setting $\mu^+(j + 1) = \mu(j) + 1$.%
\begin{equation}
	\pr_\mu d_0(g,e,\nu) =%
\begin{cases}
		(-1)^{i-1}e                           &\text{if $k = l$ and $\mu^+\delta_i = \nu$, $\exists i \leq k$}
	\\	(-1)^kR_{k+1-l}(t_{k+1-l}G_{\mu^+}g)e &\text{if $k \geq l - 1$ and $\mu^+ \mathbin| [l] = \nu$}
	\\	0                                     &\text{otherwise}
\end{cases}
\label{eqn:sdp3}
\end{equation}
\end{subequations}

It is not hard to show (under no hypotheses on $G$ other than its being a Lie $\infty$-groupoid) that $\sdp{E} \to G$ given by \eqref{eqn:sdp} is a simplicial vector bundle \cite[§3.2]{2018a}, in fact, an $N + 1$-strict vector fibration \cite[Prop.~3.4]{2018a}, for every $E \in \Rep^N_+(G)$. This vector fibration comes with a \emph{canonical} cleavage, hereafter denoted by $\sdp{c} = \{\sdp{c}_{n,k}\}_{n>k\geq 0}$, which is both normal and coherent (Definitions \ref{defn:16A.3.1} and \ref{defn:coherent}) and moreover “weakly flat” in the sense of \cite[§2]{2022a}: by definition, the image $\sdp{C}_{n,k} \subset \sdp{E}_n$ of $\sdp{c}_{n,k}$ equals $\sdp{C}_n = \ker(\pr_{\smash{[n]\xto=[n]}})$ for all $k < n$. Thus $(\sdp{E},\sdp{c})$ is an object of the category whose objects are the pairs $(V,c)$ consisting of an $N + 1$-strict vector fibration $V \to G$ and a normal cleavage $c = \{c_{n,k}\}_{n>k\geq 0}$ of $V \to G$ and whose morphisms are the ones in $\VB^N(G)$ between the underlying vector fibrations of such pairs. In fact, if we denote the latter category by $\VB^N_+(G)$, the correspondence $E \mapsto (\sdp{E},\sdp{c})$ is part of a \emph{functor} $\Rep^N_+(G) \to \VB^N_+(G)$: the following formula for turning intertwiners $\varPhi: E \to F \in \Rep^N_+(G)$ into morphisms of simplicial vector bundles $\sdp{\varPhi}: \sdp{E} \to \sdp{F} \in \VB^N(G)$ appears in \cite[§6]{2022a}.
\begin{equation}
	\pr_\mu\sdp{\varPhi}_n(g,e,\nu) =%
\begin{cases}
		\varPhi_{k-l}(t_{k-l}G_\mu g)e &\text{if $k \geq l$ and $\mu \mathbin| [l] = \nu$}
	\\	0                              &\text{otherwise}
\end{cases} \text{\qquad $\forall\mu: [k] \into_0 [n]$}
\label{eqn:sdp+}
\end{equation}

The above constructions plus those of sections \ref{sub:splitting} and \ref{sub:overview} may be organized for mnemonic ease into the following scheme.
\begin{equation*}
\xymatrix@C=9em{%
 \VB^N(G)
 \save[]+<+4em,+13ex>
	*+{\VB^N_+(G)} \ar[]_{(V,c)\mapsto V}^\simeq \ar@/_.4pc/[r]_(.45){(V,c)\mapsto\spl{V}\!\quad}
	\ar@{<-}@/^.4pc/[r]^(.45){\quad\! E\mapsto(\sdp{E},\sdp{c})}
 \restore
 \ar@{<-->}[r]_-{\text{“virtual” equivalence?}}^-\simeq
 &	\Rep^N_+(G)}
\end{equation*}
The leftward descending arrow (forgetful functor) is a categorical equivalence (by Proposition \ref{prop:normal}). The rightward descending arrow (Dold–Kan “pseudofunctor”) is given by the splitting constructions of sections \ref{sub:splitting} and \ref{sub:overview}. Although the latter correspondence is almost always not functorial, the composite one, $E \mapsto (\sdp{E},\sdp{c}) \mapsto \spl{\sdp{E}} =: \spl{E}$, nevertheless invariably is, in fact, we have a strict identification of representations up to homotopy $\spl{E} \simeq E$ which we claim is \emph{natural} in $E$ and which extends the obvious identification of the underlying cochain complexes \cite[Thm.~5.1, part 1]{2018a}; accordingly, $E \mapsto (\sdp{E},\sdp{c})$ must be a categorical embedding. \em We contend that upon passing to the associated\/ \emph{derived} or\/ \emph{homotopy} categories of\/ $\Rep^N_+(G)$ and\/ $\VB^N_+(G)$ the two correspondences\/ $E \mapsto (\sdp{E},\sdp{c})$ and\/ $(V,c) \mapsto \spl{V}$ descend to a pair of mutually quasi-inverse functors. \em Before we substantiate our assertion, let us review the relevant concepts involved in it.

\paragraph{Simplicial homotopies and the derived functor.} Let $V \xto{p} G$, $W \xto{q} G$ be simplicial vector bundles over $G$. Let $\phi_0$,~$\phi_1: V \to W$ be morphisms in $[\Cat{\Delta}^\op,\VB]$ covering the identity transformation of $G$. A \emph{homotopy} $\omega: \phi_0 \hto \phi_1$ is a morphism $\omega: I \times V \to W$ in $[\Cat{\Delta}^\op,\VB]$ covering $I \times G \xto{\pr} G$ such that \[%
	V \simeq I^0 \times V \xto{\delta_{1|r}\times\id} I \times V \xto{\omega} W
\quad \text{equals}\quad
	V \xto{\phi_r} W \text{\quad for $r = 0$,~$1$,}
\] where we regard the fundamental interval $I = I^1$ as a discrete simplicial manifold and identify $I \times V$ with the pullback of $V$ along the smooth simplicial map $I \times G \xto{\pr} G$. For every $n$-simplex, $v \in \varDelta^n(V)$, in $V$ we obtain a corresponding “prism”, $\hmt{\omega}(v) \in (I \times \varDelta^n)(W)$, in $W$ by setting \[%
	\hmt{\omega}(v) = (I \times \varDelta^n \xto{\id\times v} I \times V \xto{\omega} W).
\] Since $\omega$ is a simplicial map, the correspondence $v \mapsto \hmt{\omega}(v)$ must be compatible with taking faces and degeneracies, in addition of course to satisfying $\hmt{\omega}(v) \circ (\delta_{1|r} \times \id) = \phi_rv$ for $r = 0$,~$1$:
\begin{gather*}
	\hmt{\omega}(v) \circ (\id \times \delta_i) = \hmt{\omega}(d_iv), \quad
	\hmt{\omega}(v) \circ (\id \times \upsilon_j) = \hmt{\omega}(u_jv).
\end{gather*}
The homotopy itself is completely encoded in the maps $V_n \to (I \times \varDelta^n)(W)$, $v \mapsto \hmt{\omega}(v)$, each one of which depends smoothly and linearly on $v$: for all $\epsilon \in \varDelta^n(I)$, $v \in \varDelta^n(V)$, \[%
	\omega_n(\epsilon,v) = (\varDelta^n \xto{(\epsilon,\id)} I \times \varDelta^n \xto{\hmt{\omega}(v)} W).
\] In fact, giving a homotopy $\omega$ is tantamount to giving a family of maps $v \mapsto \hmt{\omega}(v)$ with the indicated properties; apart from issues of smoothness and linearity, the equivalence between the two viewpoints is the content of \cite[p.~16, Prop.~6.2]{May67}. See also \cite{GZ67}.

Let $\varOmega: \varPhi \hto \varPsi$ be a homotopy between two intertwiners $\varPhi$,~$\varPsi: E \to F \in \Rep^N_+(G)$; see section \ref{sub:intertwiners} for the relevant definitions. It is not hard to write down a corresponding simplicial homotopy $\sdp{\varOmega}: \sdp{\varPhi} \hto \sdp{\varPsi}$ between the two morphisms $\sdp{\varPhi}$,~$\sdp{\varPsi}: \sdp{E} \to \sdp{F} \in \VB^N(G)$ associated via \eqref{eqn:sdp+} with these intertwiners: writing $\epsilon_j \in \varDelta^n(I)$ for the $n$-simplex corresponding to the poset map $[n] \to [1]$ that hits $1$ exactly $j$ times, where $j = 0$,~$\dotsc$,~$n + 1$, this is accomplished by our next formula, where $\mu: [k] \into_0 [n]$, $\nu: [l] \into_0 [n]$, and where $j_\mu \in [k + 1]$ denotes the only integer for which $\epsilon_j\mu = \epsilon_{j_\mu}$.
\begin{multline}
	\pr_\mu\sdp{\varOmega}_n(\epsilon_j;g,e,\nu) = {}\\*
\begin{cases}
%smash[b]{%
		\varPhi_{k-l}(t_{k-l}G_\mu g)e + \sum\limits_{i=0}^{j_\mu-1} (-1)^i\varOmega_i(t_iG_\mu g)R_{k-l-i}(s_{k-l-i}t_{k-l}G_\mu g)e%}
	\\	                               &\qquad\llap{\makebox[14em][l]{if $k - j_\mu \geq l \geq 0$ and $\mu \mathbin| [l] = \nu$}}
	\\	\varPsi_{k-l}(t_{k-l}G_\mu g)e &\qquad\llap{\makebox[14em][l]{if $k \geq l > k - j_\mu$ and $\mu \mathbin| [l] = \nu$}}
	\\	0                              &\qquad\llap{\makebox[14em][l]{otherwise}}
\end{cases} \quad
\label{eqn:sdp++}
\end{multline}

The existence of the latter construction on homotopies has immediate implications. Let $\hcat{\Rep}^N_+(G)$ be the quotient of $\Rep^N_+(G)$ under the categorical congruence which identifies two intertwiners $\varPhi$,~$\varPsi: E \to F$ whenever there exists some homotopy $\varOmega: \varPhi \hto \varPsi$ between them; we refer to this as the \emph{derived category} of $\Rep^N_+(G)$ \cite[Rmk.~3.8]{AAC13}. Similarly, let $\hcat{\VB}^N_+(G)$ be the quotient of $\VB^N_+(G)$ obtained by identifying two morphisms whenever they are simplicially homotopic; we call this the \emph{homotopy category} of $\VB^N_+(G)$ \cite[p.~57, §IV.1]{GZ67}. Because of the existence of the operation $\varOmega \mapsto \sdp{\varOmega}$ given by \eqref{eqn:sdp++}, the functor $\Rep^N_+(G) \to \VB^N_+(G)$, $E \mapsto (\sdp{E},\sdp{c})$ drops down to a corresponding \emph{derived functor} $\hcat{\Rep}^N_+(G) \to \hcat{\VB}^N_+(G)$.

\paragraph{Essential surjectivity, faithfulness, and fullness.} The \emph{essential surjectivity} of the derived functor $\hcat{\Rep}^N_+(G) \to \hcat{\VB}^N_+(G)$, $E \mapsto (\sdp{E},\sdp{c})$ is precisely the content of \cite[Thm.~5.1, part 2]{2018a}; here are the keynote insights. Given $V = (V,c) \in \VB^N_+(G)$, let us write $\sdp{V} = \sdp{\spl{V}}$ for the result of first applying the Dold–Kan “pseudofunctor” $V \mapsto \spl{V}$ and then the semidirect product functor $E \mapsto \sdp{E}$ to $V$; we may think of $\sdp{V}$ as a suitable “straightening” of $V$. As explained in \cite[§5.3]{2018a}, we have a canonical morphism $\varTheta = \varTheta^V: V \to \sdp{V}$ in $\VB^N_+(G)$ whose construction involves ideas similar to those we exploited in section \ref{sub:overview} and proceeds as follows. We start by introducing certain morphisms $v \mapsto Q(v)$ of $V_n$ to $I^n_+(V)$ in $\VB$ covering $G_n \simeq \varDelta^n(G) \xto{\alpha_+^*} I^n_+(G)$ (interpolation blow-up, page \pageref{fig:alpha+}) and satisfying the boundary conditions $Q(v)\delta_\top = v\alpha$ (cubical blow-up, page \pageref{fig:alpha}), $Q(v)\delta_\bot = P(pv,sv)$, and a few recursive others such as $Q(v)\delta_{+,i|0} = Q(d_iv)$. We then take the normalized alternating sum $\spl{\vartheta}_n(v) \in \spl{V}^{-n}_{tpv}$ of the maximal simplices of $Q(v) \mathbin| \partial_{n|1}I^n_+$. We finally set $\pr_\mu\varTheta_n(v) = \spl{\vartheta}_k(v\mu)$ for all $\mu: [k] \into_0 [n]$. It turns out that $\varTheta$ admits a canonical homotopy inverse \cite[§5.4]{2018a} and, therefore, becomes an isomorphism within the homotopy category $\hcat{\VB}^N_+(G)$ of $\VB^N_+(G)$; in particular, down to the homotopy category, the functor $E \mapsto \sdp{E}$ becomes an essentially surjective one.

The reason why our derived functor is not only essentially surjective but also \emph{faithful} is that any simplicial homotopy $\omega: \phi \hto \psi$ between two morphisms $\phi$,~$\psi: V \to W$ in $\VB^N_+(G)$ gives rise canonically to a homotopy $\spl{\omega}: \spl{\phi} \hto \spl{\psi}$ between the corresponding Dold–Kan intertwiners $\spl{\phi}$,~$\spl{\psi}: \spl{V} \to \spl{W}$ in $\Rep^N_+(G)$. The idea behind the construction of $\spl{\omega}$ could hardly be less original at this point of our exposition: we replace $\varDelta^n \times I^m_+$ with $I \times \varDelta^n \times I^m_+$ and, proceeding recursively as in Lemma \ref{lem:P+(g,v)}, we define suitable simplicial maps \(%
	\hmt{P}_+(g,v): I \times \varDelta^n \times I^m_+ \to W
\) restricting on $\{0\} \times \varDelta^n \times I^m_+$ to $P^\phi_+(g,v)$ and on $\{1\} \times \varDelta^n \times I^m_+$ to $P^\psi_+(g,v)$; we then take $\spl{\omega}_m^{-n}(g)v$ to be the normalized alternating sum of the maximal simplices of the restriction of $\hmt{P}_+(g,v)$ to $I \times \varDelta^n \times \partial_{m|1}I^m_+$; the extra dimension coming from the $I$~factor accounts for $\spl{\omega}$ having degree $-1$. Because of the existence of a \emph{natural} identification $\spl{E} \simeq E$, whenever two intertwiners $\varPhi$,~$\varPsi: E \to F \in \Rep^N_+(G)$ give rise to simplicially homotopic morphisms $\sdp{\varPhi} \hto \sdp{\varPsi}$ in $\VB^N_+(G)$, they had to be already homotopic in $\Rep^N_+(G)$; the functor $E \mapsto \sdp{E}$ thus remains faithful after passing to the derived or homotopy categories.

The functor $\Rep^N_+(G) \to \VB^N_+(G)$, $E \mapsto \sdp{E}$ is \emph{full} only in exceptional cases. It nevertheless becomes full after passing to the derived or homotopy categories: although the diagram
\begin{equation*}
\xymatrix@C=4em@R=7ex{%
 V \ar[r]^\phi
 \ar[d]_(.4){\varTheta^V}
 &	W
	\ar[d]^(.4){\varTheta^W}
\\ \sdp{V} \ar[r]^{\sdp{\phi}}
 \ar@{=>}[]+<1.3em,3.5ex>;[ur]-<1.3em,2.5ex>^\omega
 &	\sdp{W}}
\end{equation*}
does not commute on the nose in $\VB^N_+(G)$, it does commute \emph{up to homotopy}; thus in a sense $\varTheta$ is “natural” (in spite of $V \mapsto \sdp{V}$ not even being a functor), and $E \mapsto \sdp{E}$ induces an honest categorical equivalence (a canonical quasi-inverse for it being provided by our Dold–Kan “pseudofunctor” $V \mapsto \spl{V}$) at the derived level. Our idea for the construction of a homotopy $\omega$ “filling” the above diagram is, as we shall see, somewhat more interesting than the idea behind the construction of $\hmt{P}_+(g,v)$ sketched in the previous paragraph.

In order to specify such an $\omega$ we must for every $v \in V_n$ specify a lift of $I \times \varDelta^n \xto{\pr} \varDelta^n \xto{g=pv} G$ to a prism $\hmt{\omega}(v): I \times \varDelta^n \to \sdp{W}$ in a way compatible with taking faces and degeneracies in the $\varDelta^n$~factor and also so that the restriction of $\hmt{\omega}(v)$ to $\{0\} \times \varDelta^n$ resp.~$\{1\} \times \varDelta^n$ be equal to $\sdp{\phi}\varTheta v$ resp.~$\varTheta\phi v$. Now, while for $n = 0$ we have $\sdp{\phi}\varTheta v = \varTheta\phi v$ and so we may take $\hmt{\omega}(v) = 1\phi v$, for $n \geq 1$ we might try building our prism recursively by using the already defined $\hmt{\omega}(d_iv)$, the $\delta_{1|0} \times \id$~face $\sdp{\phi}\varTheta v$, and the canonical cleavage of $\sdp{W}$. Unfortunately, the prism thus obtained will not work in that its $\delta_{1|1} \times \id$~face won't necessarily be equal to $\varTheta\phi v$. However, by our inductive hypothesis, all the components of that face in $\sdp{W}_n$ except perhaps the top i.e.~$\mu = \id$-th one will coincide with those of $\varTheta\phi v$. In order to obtain agreement in that component as well, the only thing we can possibly do is add something to the top component of the last maximal $n + 1$-simplex of our provisional prism. Let us call this “something” $\spl{\omega}_n(v)$; it is not clear whether we can always find it, since according to our formula for the $d_0$~face map of the semidirect product, cf.~\eqref{eqn:sdp3} above or \cite[eq.~(13a)]{2018a}, any such will in principle only contribute a $\pm\spl{S}_0(tg)\spl{\omega}_n(v)$ term to the top component of the $\delta_{1|1} \times \id$~face. We contend that the appropriate “correction” $\spl{\omega}_n(v)$ not only exists but in fact is provided canonically by the resolution of a certain recursive lifting problem which follows the same general pattern as the recursive lifting problems we encountered earlier in the present and in the former section.

\begin{defn*} Let $m \geq a \geq b \geq 1$ be integers. Let $I^m_{a,b}$ be the simplicial subset of $\nerve([1]^m \times [m]^2)$~(= the nerve of the poset $[1]^m \times [m] \times [m]$) whose $k$-simplices are all those maps $[k] \to [1]^m \times [m]^2$ ranging within the subset
\begin{multline*}
\textstyle%
	\bigl\{(r_1,\dotsc,r_m;r_+,r_-) \in [1]^m \times [m]^2: r_+ \geq r_- \text{, $\sum_{i=1}^{a-1} r_i \leq r_+ \leq \sum_{i=1}^a r_i$, and} \\ \textstyle
		\sum_{i=1}^{b-1} r_i \leq r_- \leq \sum_{i=1}^b r_i\bigr\}
\end{multline*}
that are order preserving with respect to the following partial ordering $\leq_{a,b}$ of $[1]^m \times [m]^2$:
\begin{multline*}
	(r_1,\dotsc,r_m;r_+,r_-) \leq_{a,b} (s_1,\dotsc,s_m;s_+,s_-) \text{\quad iff} \\
		(r_1,\dotsc,r_m;r_+) \leq_a (s_1,\dotsc,s_m;s_+) \text{\quad and\quad}
			(r_1,\dotsc,r_m;r_-) \leq_b (s_1,\dotsc,s_m;s_-).
\end{multline*}
We shall be interested in the simplicial subset $I^m_{++}$ of $\nerve([1]^m \times [m]^2)$ given by $\bigcup_{m\geq a\geq b\geq 1} I^m_{a,b}$ for $m \geq 1$ and by $I^0_+ = \nerve([1]^0 \times [0]^2) \simeq \varDelta^0$ for $m = 0$. \end{defn*}

Let $\alpha_{++}: I^m_{++} \to \varDelta^m$ be the simplicial map given by $(r_1,\dotsc,r_m;r_+,r_-) \mapsto \alpha(r_1,\dotsc,r_m)$; by way of example, the restriction of $\alpha_{++}$ to the boundary components $\partial_{i|r}I^m_{++}$ of $I^m_{++}$ spanned by all those vertices for which $r_i = r$ is illustrated in Figures \ref{fig:i|0++} and \ref{fig:i|1++} for $m = 2$.%
\begin{figure}
\begin{tikzpicture}[scale=0.5]
% i = 1, r = 0:
 \draw[->,thick] (0,0)--(5,0); \draw[->,thick] (0,0)--(7,3); \draw[->,thick] (0,0)--(7,8);
 \draw[dotted] (5,0)--(7,3) (5,0)--(7,8) (7,3)--(7,8);
 \node[below left] at (0,0) {$0$}; \node[below right] at (5,0) {$2$};
 \node[right] at (7,3) {$2$}; \node[above right] at (7,8) {$2$};
% i = 2, r = 0:
 \draw[->,thick] (12,0)--(17,0); \draw[->,thick] (12,0)--(19,3); \draw[->,thick] (12,0)--(19,8);
 \draw[dotted] (17,0)--(19,3) (17,0)--(19,8) (19,3)--(19,8);
 \node[below left] at (12,0) {$0$}; \node[below right] at (17,0) {$1$};
 \node[right] at (19,3) {$1$}; \node[above right] at (19,8) {$1$};
\end{tikzpicture}
\caption{Restriction of $\alpha_{++}: I^2_{++} \to \varDelta^2$ to $\partial_{i|0}I^2_{++} \simeq I^1_{++}$, $i = 1$,~$2$}
\label{fig:i|0++}
\vskip\baselineskip
\begin{tikzpicture}[scale=0.5]
% i = 1, r = 1:
 \draw[->,thick] (-2,-3)--(3,-3); \draw[->,thick] (0,0)--(5,0); \draw[->,thick] (0,0)--(7,3);
 \draw[->,thick] (0,5)--(5,5); \draw[->,thick] (0,5)--(7,8); \draw[->,thick] (0,5)--(7,13);
 \draw[dotted] (-2,-3)--(0,0) (-2,-3)--(0,5) (-2,-3)--(5,0) (-2,-3)--(5,5) (3,-3)--(5,0) (3,-3)--(5,5);
 \draw[dotted] (0,0)--(0,5) (0,0)--(5,5) (0,0)--(7,8) (5,0)--(5,5) (5,0)--(7,3) (5,0)--(7,8);
 \draw[dotted] (5,5)--(7,8) (5,5)--(7,13) (7,3)--(7,8) (7,8)--(7,13);
 \node[below left] at (-2,-3) {$1$}; \node[below right] at (3,-3) {$2$};
 \node[left] at (0,0) {$1$}; \node[right] at (5,0) {$2$}; \node[right] at (7,3) {$2$};
 \node[left] at (0,5) {$1$}; \node[right] at (5,5) {$2$}; \node[right] at (7,8) {$2$};
 \node[above right] at (7,13) {$2$};
% i = 2, r = 1:
 \draw[fill] (10,-3) circle [radius=.05] (15,-3) circle [radius=.05];
 \draw[fill] (12,0) circle [radius=.05] (17,0) circle [radius=.05] (19,3) circle [radius=.05];
 \draw[fill] (12,5) circle [radius=.05] (17,5) circle [radius=.05] (19,8) circle [radius=.05];
 \draw[fill] (19,13) circle [radius=.05];
 \draw[dotted,thick] (10,-3)--(17,0) (10,-3)--(17,5) (12,0)--(19,3) (12,0)--(19,8);
 \draw[dotted,thick] (10,-3)--(15,-3) (12,5)--(19,13);
 \draw[dotted] (10,-3)--(12,0) (10,-3)--(19,3) (10,-3)--(19,8) (15,-3)--(17,0) (15,-3)--(17,5);
 \draw[dotted] (17,0)--(17,5) (17,0)--(19,3) (17,0)--(19,8) (17,5)--(19,8) (19,3)--(19,8) (19,8)--(19,13);
 \draw[dotted] (10,-3)--(12,5) (10,-3)--(19,13) (12,0)--(12,5) (12,0)--(19,13) (17,5)--(19,13);
 \node[below left] at (10,-3) {$2$}; \node[below right] at (15,-3) {$2$};
 \node[left] at (12,0) {$2$}; \node[right] at (17,0) {$2$}; \node[right] at (19,3) {$2$};
 \node[left] at (12,5) {$2$}; \node[right] at (17,5) {$2$}; \node[right] at (19,8) {$2$};
 \node[above right] at (19,13) {$2$};
\end{tikzpicture}
\caption{Restriction of $\alpha_{++}: I^2_{++} \to \varDelta^2$ to $\partial_{i|1}I^2_{++}$, $i = 1$,~$2$}
\label{fig:i|1++}
\end{figure}
Besides the latter, there are three other boundary components, namely, the simplicial subsets $\partial_{\top\ast}I^m_{++}$, $\partial_{\ast\bot}I^m_{++}$, and $\partial_{\ast\ast}I^m_{++}$ of $I^m_{++}$ spanned by all those vertices for which, respectively, $r_+ = \sum_{i=1}^m r_i$, $r_- = 0$, and $r_+ = r_-$; each one of these components is~$\simeq I^m_+$.

Now, by a recursive procedure entirely analogous to that described in section \ref{sub:overview} and involving the given cleavage of $W$, we can define suitable simplicial maps \[%
	Q_+(v): I^n_{++} \to W \text{,\quad $v \in V_n$}
\] satisfying the lifting requirements \(%
	(I^n_{++} \xto{Q_+(v)} W \to G) = (I^n_{++} \xto{\alpha_{++}} \varDelta^n \xto{pv} G)
\) plus a bunch of more or less familiar boundary conditions, for instance, that $Q_+(v)$ restricted to $\partial_{i|0}I^n_{++} \simeq I^{n-1}_{++}$ be equal to $Q_+(d_iv)$, along with the following three, where as we discussed previously \(%
	Q^V: V_n \to I^n_+(V)\), \(%
	Q^W: W_n \to I^n_+(W)
\) are the $\VB$-morphisms entering in the construction of $\varTheta^V$, $\varTheta^W$, respectively.
\begin{gather*}
	Q_+(v) \mathbin| \partial_{\top\ast}I^n_{++} = \phi \circ Q^V(v) \\
	Q_+(v) \mathbin| \partial_{\ast\bot}I^n_{++} = P_+(pv,sv) \\
	Q_+(v) \mathbin| \partial_{\ast\ast}I^n_{++} = Q^W(\phi v)
\end{gather*}
We then define $\spl{\omega}_n(v)$ to be the normalized alternating sum of the maximal simplices of $Q_+(v) \mathbin| \partial_{n|1}I^n_{++}$; each one of these is an $n + 1$-simplex sitting over $1_{n+1}tpv$ so that, as expected, $\spl{\omega}_n(v) \in \spl{W}^{-1-n}_{tpv}$. Finally, on computing the total boundary of $\spl{\omega}_n(v)$, we obtain an equation which tells us precisely that $\spl{S}_0(tg)\spl{\omega}_n(v)$ contributes the appropriate correction to the top component of the $\delta_{1|1} \times \id$~face.

\begin{rmks*} Surprisingly enough, the simplicial complex $I^m_{++}$ can also be used to construct, for any two composable morphisms $W = W(\phi) \xfrom{\phi} V(\phi) = W(\psi) \xfrom{\psi} V(\psi) = V$ in $\VB^N_+(G)$, a canonical homotopy $\spl{\varOmega}: \spl{\phi} \circ \spl{\psi} \hto \append\spl{\phi \circ \psi}$ between the two intertwiners $\spl{\phi} \circ \spl{\psi}$,~$\append\spl{\phi \circ \psi}: \spl{V} \to \spl{W}$ in $\Rep^N_+(G)$. For that purpose, we first build certain simplicial maps $P_{++}(g,v): \varDelta^n \times I^m_{++} \to W$ recursively for all $g \in G_m$, $v \in V^{-n}_{sg}$ by imposing the boundary conditions below in addition to a few more or less obvious others. \[%
\begin{split}
	P_{++}(g,v) \mathbin| \varDelta^n \times \partial_{\top\ast}I^m_{++}&= \phi \circ P^\psi_+(g,v) \\
	P_{++}(g,v) \mathbin| \varDelta^n \times \partial_{\ast\bot}I^m_{++}&= P^\phi_+(g,\psi v) \\
	P_{++}(g,v) \mathbin| \varDelta^n \times \partial_{\ast\ast}I^m_{++}&= P^{\phi\circ\psi}_+(g,v)
\end{split}
\] We then define $\varOmega_m^{-n}(g)v \in W^{-1-m-n}_{tg}$ to be the alternating sum of the maximal simplices of the restriction of $P_{++}(g,v)$ to $\varDelta^n \times \partial_{m|1}I^m_{++}$. The desired homotopy term $\spl{\varOmega}_m^{-n}(g)v \in \spl{W}^{-1-m-n}_{tg}$ is obtained as usual by normalizing this. This indicates that, even though the Dold–Kan construction is not a functor, it may nonetheless be possible to understand it as a suitable “infinity-functor” or as a “functor up to homotopy”. We do not intend to pursue this point of view. \end{rmks*}

% references %%%%%%%%%%%%%%%%%%%%%%%%%%%%%%%%%%%%
%
\end{document}